\documentclass[11pt, a4paper]{article}
\title{Circle homeomorphisms with square summable diamond shears}

\usepackage[T1]{fontenc}

\usepackage[latin9]{inputenc}

\usepackage{mathtools}
\usepackage{lmodern}
\usepackage{amsthm,amsmath,amsfonts,amssymb}
\usepackage{graphicx}
\usepackage{enumerate,enumitem}
\usepackage{authblk}
\usepackage{cite,url}
\usepackage[normalem]{ulem}
\usepackage{calrsfs}

\usepackage{multirow}
\usepackage{array}
\newcolumntype{P}[1]{>{\centering\arraybackslash}p{#1}}
\usepackage{longtable,booktabs}

\usepackage{hyperref}
\hypersetup{
    colorlinks=true, 
    linkcolor=blue, 
    urlcolor=black, 
    citecolor=black,
    linktoc=all 
}

\usepackage[activate={true,nocompatibility},final,tracking=true,kerning=true,spacing=true,factor=1100,stretch=20,shrink=20]{microtype}

\microtypecontext{spacing=nonfrench}

\allowdisplaybreaks

\setlist[enumerate]{topsep = 1ex, leftmargin=1cm, itemsep= -2pt}

\let\OLDthebibliography\thebibliography
\renewcommand\thebibliography[1]{
  \OLDthebibliography{#1}
  \setlength{\parskip}{1pt}
  \setlength{\itemsep}{2pt}
}

\newtheorem{thm}{Theorem}[section]
\newtheorem{cor}[thm]{Corollary}

\newtheorem{lem}[thm]{Lemma}
\newtheorem{prop}[thm]{Proposition}

\theoremstyle{definition} 
\newtheorem{df}[thm]{Definition}
\newtheorem{ex}[thm]{Example}
\newtheorem{remark}[thm]{Remark}

\numberwithin{equation}{section}

\usepackage[font=normal]{caption}
\usepackage{floatrow}

\usepackage{mathtools}

\usepackage[dvipsnames]{xcolor}

\global\long\def\ii{\mathfrak{i}}

\newcommand{\abs}[1]{\left\lvert #1 \right \rvert}

\newcommand{\brac}[1]{\left \langle #1 \right \rangle}
\newcommand{\norm}[1]{\lVert #1 \rVert}

\newcommand{\mc}[1]{\mathcal{#1}}
\newcommand{\m}[1]{\mathbb{#1}}

\newcommand{\mf}[1]{\mathfrak{#1}}

\renewcommand\Re{\operatorname{Re}}
\renewcommand\Im{\operatorname{Im}}

\def\PSL{\operatorname{PSL}}
\def\PSU{\operatorname{PSU}}

\def\mob{\mathrm{M\ddot ob}}
\def\Diff{\operatorname{Diff}}
\def\QS{\operatorname{QS}}

\def\WP{\operatorname{WP}}

\def\a{\alpha}
\def\b{\beta}
\def\g{\gamma}
\def\G{\Gamma}

\def\t{\theta}

\def\i{\iota}

\def\l{\lambda}

\def\L{\Lambda}
\def\s{\sigma}

\def\o{\omega}
\def\O{\Omega}
\def\vare{\varepsilon}

\def\dd{\mathrm{d}}

\def\fan{\mathrm{fan}}

\newcommand{\ad}[1]{\overline{#1}}

\def\1{\mathbf{1}}

 \newcommand{\cayley}{\mf c}
 \newcommand{\Farey}{\mf F}
 \newcommand{\splus}{{\scriptstyle +}}
 \newcommand{\sminus}{{\scriptstyle -}}
  \newcommand{\spm}{{\scriptstyle \pm}}
 \newcommand{\dualtree}{\Farey^*}
 \newcommand{\tri}{\tau}
\newcommand{\gen}{\operatorname{gen}}
\newcommand{\child}{\operatorname{child}}
\newcommand{\Cr}{\operatorname{cr}}
\newcommand{\emb}{\Xi}

 \def \1{\mathbf{1}}

\def\Id{\operatorname{Id}}

\author{Dragomir {\v S}ari\'c\thanks{\protect\url{dragomir.saric@qc.cuny.edu} PhD Program in Mathematics, The Graduate Center, City University of New York, New York, NY, USA, and Queens College, City University of New York, Flushing, NY, USA
},\quad Yilin Wang\thanks{\protect\url{yilin@ihes.fr} Institut des Hautes \'Etudes Scientifiques, Bures-sur-Yvette, France
},\quad Catherine Wolfram\thanks{\protect\url{wolframc@mit.edu} Massachusetts Institute of Technology, Cambridge, MA, USA}}
\begin{document}

\maketitle
\begin{abstract}
We introduce and the study the space of homeomorphisms of the circle (up to M\"obius transformations) which are in $\ell^2$ with respect to modular coordinates called \textit{diamond shears} along the edges of the Farey tessellation. Diamond shears are related combinatorially to shear coordinates, and are also closely related to the $\log \Lambda$-lengths of decorated Teichm\"uller space introduced by Penner. We obtain sharp results comparing this new class to the Weil--Petersson class and H\"older classes of circle homeomorphisms. We also express the Weil--Petersson metric tensor and symplectic form in terms of infinitesimal shears and diamond shears.


\end{abstract}

\begin{figure}[ht]
\centering
\includegraphics[width=.6\textwidth]{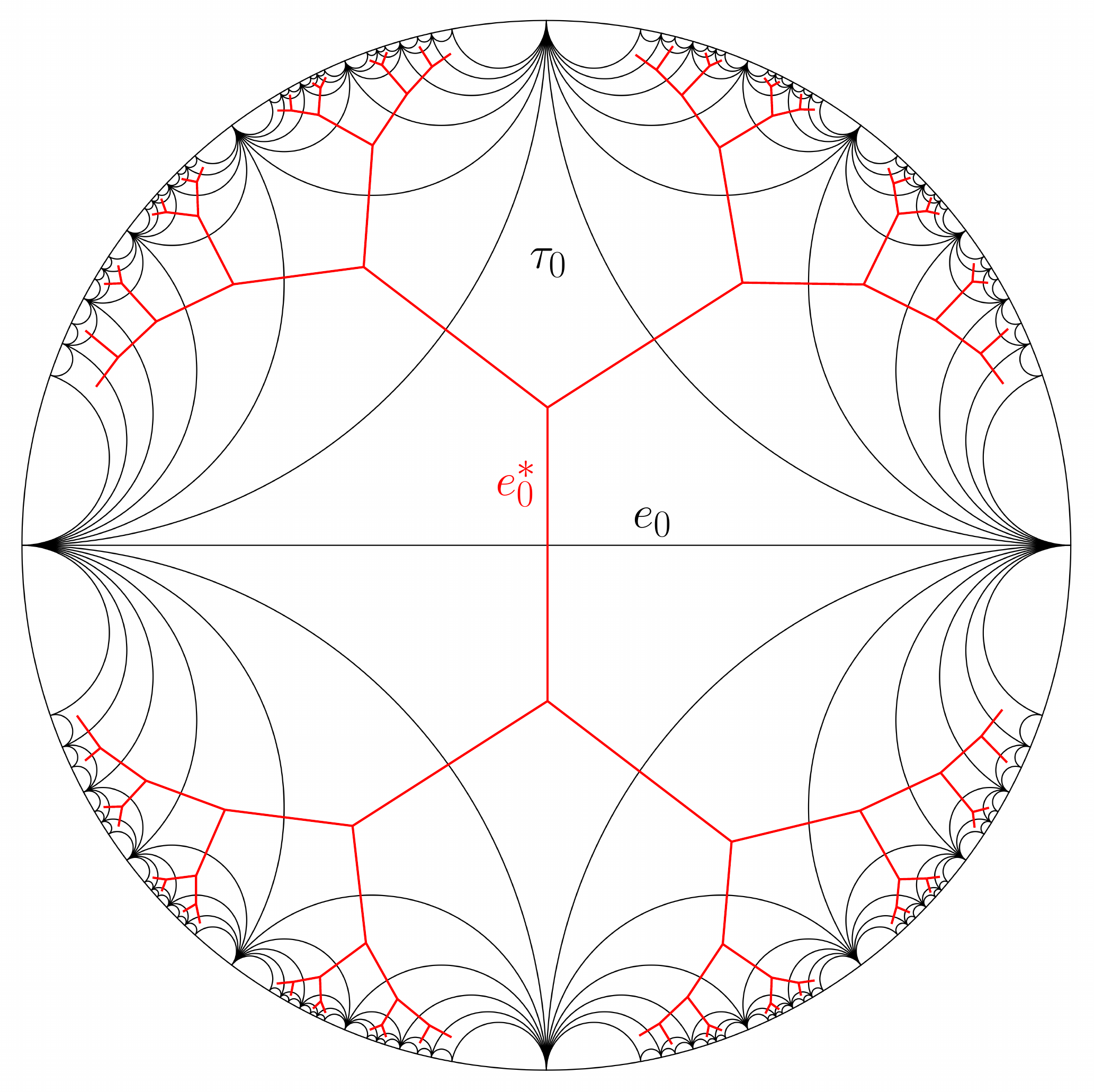}
\caption{\label{fig:Farey_D} Farey tessellation in $\m D$ (black) and the dual tree (red) up to generation 5.}
\end{figure}

\newpage

\tableofcontents


\section{Introduction}

The Farey tessellation $\Farey$ is an ideal triangulation of the disk $\m D$, characterized by the modular invariance property that it is preserved by the action of $\PSL(2,\m Z)$ on the disk. Its vertices $V = \m Q^2 \cap \m T$ are the rational points on the circle $\m T$, and we use $E$ to denote its edges. The Farey tessellation has various connections to number theory \cite{topologyofnumbers}, but our motivation comes from Teichm\"uller theory, where the universal Teichm\"uller space 
$$T(\m D) = \QS(\m T)/\mob(\m T)$$
can be identified with quasisymmetric circle homeomorphisms fixing three points $-1,\ii,1$.

Any element $h$ of the larger homogeneous space of orientation preserving circle homeomorphisms
\begin{align*}
    \text{Homeo}^+(S^1)/\mob(\m T) \simeq \{h\in \text{Homeo}^+(\m T)~:~ h\text{ fixes } -1, \ii, 1\} 
\end{align*}
can be uniquely encoded by a \textit{shear coordinate} on the edges of the Farey tessellation, namely a function $s:E\to \m R$. However, not all functions $s: E\to \m R$ encode circle homeomorphisms. Penner posed the question of classifying which functions $s:E\to \m R$ encode various smoothness classes of homeomorphisms. In \cite{Saric_circle,Saric_new}, the shear coordinates for homeomorphisms, symmetric homeomorphisms, and quasisymmetric homeomorphisms were characterized by the first author. 

In this paper we take an opposite perspective. We define two $\ell^2$ classes of shear functions, and then study the regularity properties these $\ell^2$ conditions imply for circle homeomorphisms 
and the relation to the existing Hilbert manifold structure on the universal Teichm\"uller space. For precise descriptions of the Farey tessellation, shears and shear coordinates, see the preliminaries in Section~\ref{sec:prelim} or e.g.\ \cite[Chapter 8]{Bonahon} by Bonahon. 

Na\"ively, the first class to consider is the set of square summable shear functions, i.e.\
\begin{align*}
    \mc S:= \{s:E\to \m R: \sum_{e\in E} s(e)^2<\infty\}.
\end{align*}
However, we find that not all $s\in \mc S$ even encode circle homeomorphisms. 
Conversely, we also find that there are quasisymmetric homeomorphisms which are not in $\mc S$. See Proposition~\ref{prop:S_not_homeo_qs} for simple examples to illustrate both of these facts. 
In summary,
\begin{align*}
    T(\m D) \simeq \text{QS}(\m T)/\mob(\m T) \not\subset \mc S \qquad \text{and} \qquad \mc S\not\subset \text{Homeo}(\m T)/\mob(\m T).
\end{align*}
These observations show that a basis of shear functions each supported on a single edge is ``too large'' to define an $\ell^2$ space of circle homeomorphisms.

Motivated by this, we investigate shear functions supported on finitely many edges. Finitely supported shear functions always induce homeomorphisms, which in particular are \textit{piecewise M\"obius} with pieces bounded by rational points in $V$ (see Lemma \ref{lem:finite_shear_mobius}).  This class of circle homeomorphisms has been studied in, e.g., \cite{MRW1,MRW2,MALIKOV1998282,Frenkel_Penner,Penner1993UniversalCI}.  We then show that a homeomorphism $h$ with finitely supported shear $s_h$ is piecewise M\"obius \textit{and $C^1$} with breakpoints in $V$ if and only if it belongs to a linear subspace of all shear functions spanned by \emph{diamond shears}. See Lemma~\ref{lem:finite_balance} and Proposition~\ref{prop:s_to_theta_finite}. Combinatorially, this condition is 
equivalent to requiring that the shears on all edges incident to the same vertex sum up to zero (which we call the finite balanced condition). 

To define \emph{diamond shears} more precisely in terms of shears, choose an edge $e\in E$, and let $e_1 = (a,b)$, $e_2=(b,c)$, $e_3 = (c,d)$, $e_4 = (d,a)$ in $E$ be the boundary edges in counterclockwise order of quadrilateral  $Q_e = (a,b,c,d)$  consisting of the two triangles from $\Farey$ containing $e=(a,c)\in E$. 
A unit of diamond shear supported at the edge $e$, i.e. $\vartheta_h(e) = 1$ and $\vartheta_h(e') = 0$ for all $e' \in E$ and $e' \neq e $, 
is equivalent to four nonzero shears where $s_h(e_1) = s_h(e_3) = 1$ and $s_h(e_2) = s_h(e_4) = -1$.
The name ``diamond'' comes from the picture that the support of one diamond shear corresponds to a quad/diamond of regular shears. 
See Section~\ref{sec:examples} for concrete examples of the correspondence between diamond shear coordinates and circle homeomorphisms.

When $s_h$ has \emph{infinite} support, we define the diamond shear coordinate $\vartheta_h$ combinatorially as an infinite sum denoted as $\Psi(s_h)$ whenever $s_h$ is in a certain subclass $\mc P$ (which can be characterized analytically in terms of differentiability of $h$). 
See Section~\ref{sec:combinatorial_diamond}.
It is often more convenient to define the diamond shear coordinate on the edges of the dual tree $\dualtree$. As the edges of $\Farey$ and $\dualtree$ are in one-to-one correspondence, this identification should not add any ambiguity.  See Figure~\ref{fig:Farey_D}. 

We show that diamond shears can be described analytically:

\begin{prop}[{See Proposition~\ref{prop:theta_analytic}}] \label{prop:intro_theta_analytic}
Assume that $s_h \in \mc P_0 \subset \mc P$ \textnormal(which implies that $h$ is differentiable at all $v\in V$\textnormal). The diamond shear coordinate $\vartheta_h$ of $h$ is given by
    \begin{align} 
        \vartheta_h (e
        ) &= \frac{1}{2} \log h'(a) h'(b) - \log \frac{h(a)-h(b)}{a-b}
    \end{align}
    for all $e = (a,b) \in E$. 
\end{prop}

This proposition connects diamond shears to the \textit{$\log \Lambda$-lengths} defined by Penner 
on decorated Teichm\"uller space, which is a (trivial) bundle over $T(\m D)$, with fiber $\m R_{>0}^V$ over $h\in T(\m D)$ corresponding to choosing a horocycle at each $h(v) \in h (V)$. See \cite{Penner1993UniversalCI,Penner2002OnHF,MALIKOV1998282} or the book \cite{PennerBook}. Roughly speaking, the decoration allows one to truncate and define the ``renormalized hyperbolic length'' of an infinite geodesic $h(e) \in h(E)$. A homeomorphism $h$ that is differentiable on $V$ gives a canonical way to \emph{fix} as in \cite{MALIKOV1998282} a decoration on $h(V)$, and $\vartheta_h(e)$ is equal to $-1/2$ times the renormalized length of $e$. 
\begin{cor} [See {Lemma~\ref{lem:theta_logL}}] \label{cor:intro_theta_logL}
    If $s_h \in \mc P_0$, 
    then for any $e = (a,b) \in E$, 
    $$\vartheta_h (e) = -\log \Lambda_h (e) = -\frac{1}{2} \, \mathrm{length} (h(e))$$
    where $\mathrm{length} (h(e))$ is the signed hyperbolic length of the part of $h(e)$ between the horocycles centered at $h(a)$ and $h(b)$ chosen from the fixed decoration.
\end{cor}

Our main object of study in this paper is the set of shear functions with $\ell^2$ summable diamond shear coordinates:
\begin{align*}
    \mc H := \{s: E\to \m R : \sum_{e\in E} \vartheta(e)^2 < \infty\}.
\end{align*}
It is relatively straightforward to show that this corresponds to a class of circle homeomorphisms. Given this, we use the abuse of notation $h\in \mc H$ to mean $s_h\in \mc H$ throughout. 
\begin{prop}[See Corollary \ref{cor:H_in_QS}]
   If $s\in \mc H$, then $s$ induces a quasisymmetric circle homeomorphism. In other words, $\mc H\subset \QS(\m T)/\mob(\m T) \simeq T(\m D).$
\end{prop}

Our first main result is to characterize the H\"older classes of circle homeomorphisms that are contained in $\mc H$. Define for $\a \in (0,1]$,
\begin{align}\label{eq:holder_def}
    \mc C^{1,\a} := \{h: \m T \to \m T \text{ homeomorphism} \colon \log h' \text{ is $\a$-H\"older}\}.
\end{align}
In particular, the welding homeomorphisms of $C^{1,\a}$ Jordan curves belong to $\mc C^{1,\a}$. 
\begin{thm}[See Theorem \ref{thm:C1alpha}] \label{thm:intro_holder}
    If $\alpha>1/2$, then $\mc C^{1,\alpha}\subset \mc H$. 
\end{thm}
This result is sharp as Theorem~\ref{thm:intro_wp_H} will show that $\mc C^{1,1/2}$ is not in $\mc H$.
The proof of this result relies on Proposition~\ref{prop:intro_theta_analytic} 
and the $\ell^{2\a}$ summability of the lengths 
of the shorter arcs in $\m T$ between $a$ and $b$ for each $(a,b)\in E$. The $\ell^2$ summability was studied and implied by results in \cite{Hall,Penner2002OnHF}, we improve it to $\ell^{2\a}$ summability. 
See Proposition~\ref{prop:farey-lengths}.

Another main result of our work is an explicit construction of a quasiconformal extension $f: \m D \to \m D$ of $h\in \mc H$ inspired by a construction in \cite{KahnMarkovic} by Kahn and Markovic. The construction is adapted to the cell decomposition of the Farey tessellation, and is one of the places where its discrete structure is essential. This construction crucially uses the \textit{generalized balanced condition} satisfied by shear functions that can be written in terms of diamond shears. While characterizations of shear functions for quasisymmetric homeomorphisms are known, analogous methods for constructing their quasiconformal extensions using the shear function are not known.  We further find that if $h\in \mc H$, then the Beltrami differential $\mu_f=\bar \partial f / \partial f$ of the extension $f$ is in $L^2(\m D, \dd_{\text{hyp}})$. 

This leads to a connection between $\mc H$ and the \textit{Weil--Petersson class} of circle homeomorphisms $\WP(\m T)$. The \textit{Weil--Petersson Teichm\"uller space} $\WP(\m T)/\mob(\m T) = : T_0(\m D)$ is a subspace of $T(\m D)$ defined as the completion of $\Diff(\m T)/\mob(\m T)$ under its unique homogeneous K\"ahler metric (the Weil--Petersson metric) \cite{TT06}. The space of Weil--Petersson homeomorphisms $\WP(\m T)$ is characterized analytically by Shen \cite{shen13}, see Definition~\ref{def:WP}, and also by, e.g.,  \cite{Cui,bishop-function-theoretic,W2,VW2}. In particular one definition of \textit{Weil--Petersson homeomorphisms} $\WP(\m T)$ is that they admit \textit{a} quasiconformal extension to the disk whose {Beltrami differential} $\mu_f=\bar \partial f / \partial f$ is in $L^2(\m D, \dd_{\text{hyp}})$ (see \cite{TT06} or Theorem~\ref{thm:WP_Beltrami_L2}). 
Cui \cite{Cui} showed that the Douady--Earle quasiconformal extension of a Weil--Petersson homeomorphism satisfies this property, and we remark that it is notable that our construction using shears has the desired property for all $h\in \mc H$.



Ultimately we prove the following relationships between $\mc H, \mc S$ and the Weil--Petersson class.

\begin{thm}\label{thm:intro_wp_H}
We have $\mc H \subset \WP(\m T)$. Additionally if $h\in \WP(\m T)$, then $s_h\in \mc S$. Both inclusions are strict. 
\end{thm}
See Theorem \ref{thm:H_Beltrami_bound} for the first inclusion, and Section~\ref{sec:counterexample} for why it is strict. See Theorem \ref{thm:WP_in_S} for the last inclusion. For comparison, note that this result implies that Theorem~\ref{thm:intro_holder} is sharp, since $\mc C^{1,1/2} \nsubseteq \mc H$ as otherwise it would also be in $\WP(\m T)$ (which contradicts Lemma~\ref{lem:C_alpha_in_WP}). As smooth diffeomorphisms are dense in $\WP(\m T)$, so is $\mc H$.

In fact, our construction of the quasiconformal extension for functions in $\mc H$ can be adapted to show the following stronger result that convergence in $\mc H$ endowed with its $\ell^2$ topology implies convergence in the Weil--Petersson metric.
\begin{thm}[See {Corollary \ref{cor:H_convergence_cor}}]
    Suppose that $h,(h_n)_{n\geq 1}\in \mc H$ with diamond shear coordinates $\vartheta,\vartheta_n$ respectively. If 
\begin{align*}
    \lim_{n\to \infty} \sum_{e\in E} (\vartheta_n(e) - \vartheta(e))^2 = 0,
\end{align*}
then $h_n$ converges to $h$ in the Weil--Petersson metric. 
\end{thm}
We obtain immediately the following corollary.
\begin{cor}
    The class of continuously differentiable and piecewise M\"obius circle homeomorphisms \textnormal(with break points in $V$\textnormal) is dense in $\mc H$ and in $\WP(\m T)$.
\end{cor}
Indeed, this class is equal to the class of circle homeomorphisms with finitely supported diamond shear coordinates  (Lemma~\ref{lem:finite_balance} and Proposition~\ref{prop:s_to_theta_finite}) which is dense in $\mc H$ for the Weil--Petersson metric by the above theorem. 

\bigskip
 
Finally, we study infinitesimal shear and diamond shear coordinates on the tangent spaces of 
$\mc H$. Since $\mc H \subset \WP(\m T)$, we compute the Weil--Petersson metric in terms of diamond shears. 

\begin{thm}[See {Corollary~\ref{cor:h_is_in_H32}, Theorem~\ref{thm:WP_one_diamond} and Corollary~\ref{cor:full_metric}}]
Each $\ell^2$-summable infinitesimal diamond shear gives rise to a $H^{3/2}$ vector field on $\m T$. 
    Let  $u_1,u_2$ be the $H^{3/2}$ vector fields corresponding to the $\ell^2$-summable infinitesimal diamond shears $\dot \vartheta_1, \dot \vartheta_2 \in T_{h} \mc H \subset T_h\WP(\m T)$.
    Then 
    $$\brac{u_1, u_2}_{\WP} = \sum_{e_1 \in E} \sum_{e_2 \in E} \dot \vartheta_1(e_1) \dot \vartheta_2 (e_2) \, g (h(Q_{e_1}), h(e_1), h(Q_{e_2}), h(e_2)),$$
    where for $Q = (a_1,a_2,a_3,a_4)$, $e = (a_1,a_3)$,  $Q' = (b_1,b_2,b_3,b_4)$, $e' = (b_1,b_3)$,
    $$g (Q, e, Q', e') = \frac{2}{\pi} \Re \sum_{j,k=1}^4 \frac{(-1)^{j+k} a_j^2 \bar b_k^2 (a_{j+1}-a_{j-1})(\bar b_{k+1} -\bar b_{k-1})}{(a_{j+1}-a_{j})(a_{j}-a_{j-1})(\bar b_{k+1}-\bar b_{k})(\bar b_{k} - \bar b_{k-1})} \sigma(a_j,b_k),$$
    and for $a,b \in \m T$,
    $$
\sigma (a,b)=\sum_{p=0}^{\infty} \frac{(a\bar{b})^{p+1}}{(1+p)(2+p)(3+p)}.$$
\end{thm}

The expression of the metric tensor is relatively complicated. In contrast, the symplectic form has a very simple expression first noticed by Penner in \cite{Penner1993UniversalCI,Penner_WPvol}. Using the formula in \cite[Thm.\,5.5]{Penner1993UniversalCI} and the relationship between diamond shears and $\log \Lambda$-lengths that we describe in Section \ref{sec:diamond_logL}, we 
can rewrite the Weil--Petersson symplectic form in terms of a mixture of infinitesimal shears and diamond shears as follows.

\begin{thm}[See {Theorem~\ref{thm:symplectic_form}}]
Let $\omega$ denote the Weil--Petersson symplectic form on $\WP(\m T)$ and fix $h\in \mc H$. Suppose that $u_1,u_2$ are the $H^{3/2}$ vector fields corresponding to the $\ell^2$-summable infinitesimal diamond shears $\dot \vartheta_1, \dot \vartheta_2 \in T_{h} \mc H \subset T_h\WP(\m T)$ with  infinitesimal shear coordinates $\dot s_1, \dot s_2$ respectively. Then 
\begin{align*}
    \omega(u_1,u_2) = \sum_{e\in E} \dot \vartheta_1(e) \dot s_2(e) =  -\sum_{e\in E}  \dot s_1(e) \dot  \vartheta_2(e).
\end{align*}
\end{thm} 
We note the resemblance of this formula with the Weil--Petersson symplectic form 
on the finite dimensional Teichm\"uller spaces $T_{g,n}$ using the Fenchel-Nielson coordinates due to Wolpert~\cite{Wolpert_symplectic}:
$$\o = - \frac{1}{2} \sum_{\g \in \G} \dd l \wedge \dd \tau,$$
where $\G$ is a maximal multicurve on a Riemann surface of finite type.
Here, one may draw the analogy by interpreting $\dot s$ as the deformation by twisting along closed geodesics corresponding to $\dd \tau$, and $\dot \vartheta$ as the deformation by changing the length of geodesics corresponding to $ -\frac 12 \dd l$ by Corollary~\ref{cor:intro_theta_logL}. 

\bigskip 

\textbf{Outline of the paper.} In Section \ref{sec:prelim}, we recall definitions and basic results about the Farey tessellation, shears, and the classes of homeomorphisms of the circle that we consider. In Section~\ref{sec:diamond}, we relate the Weil--Petersson class to shears in the finite support case, and motivated by this define diamond shears coordinates combinatorially (on a class of shear functions called $\mc P$), analytically (in terms of $h'$ on $V$), and in terms of $\log \Lambda$-lengths. We also define the classes $\mc H, \mc S$. In Section~\ref{sec:Holder}, we prove that $\mc C^{1,\a} \subset \mc H$ (Theorem~\ref{thm:C1alpha}). In Section \ref{sec:relation_wp}, we prove the theorems relating $\WP(\m T)$, $\mc H$, and $\mc S$.  Section~\ref{sec:infinitesimal} is devoted to the infinitesimal theory of the Weil--Petersson metric and symplectic form. We define infinitesimal shears and diamond shears and compute the Weil--Petersson metric tensor (Theorem~\ref{thm:WP_one_diamond}, Corollary~\ref{cor:full_metric}) and symplectic form (Theorem~\ref{thm:symplectic_form}) in terms of shears and diamond shears on $\mc H$.

\subsection*{Acknowledgments}
The first and the second author are grateful to Chris Bishop for providing the motivation that led to this work. 
We thank Robert Penner for suggestions on the citations. 
D.S. is  partially supported by Simons grant 346391 and by a PSC CUNY grant. Y.W. is partially supported by NSF award DMS-1953945. C.W. was partially supported by NSF award DMS-1712862 and NSF award DMS-2153742. This work is also supported by the NSF Grant DMS-1928930 while Y.W. and C.W. participated in a program hosted
by the Mathematical Sciences Research Institute in Berkeley, California, during
the Spring 2022 semester. 

\section{Preliminaries}\label{sec:prelim}
\subsection{Farey tessellation} \label{sec:Farey}

For  $D = \m D$ or $\m H$, we say that a triangle $\tri\subset D$ is \textit{geodesic} if all its edges are geodesics for the hyperbolic metric on $D$. All triangles in this discussion will be geodesic. An \textit{ideal triangle} is a geodesic triangle with all its vertices on $\partial D$. An \textit{ideal triangulation} of $D$ is a (necessarily infinite) locally finite collection of non-overlapping ideal triangles that cover $D$. The data of an ideal triangulation is encoded in its edges and vertices. For $a\neq b  \in \partial D$, we write $(a,b)$ for the hyperbolic geodesic connecting $a,b$.

The \textit{Farey tessellation} is an ideal triangulation of $\mathbb{D}$ with many natural symmetries and properties. Since the Farey tessellation will be ubiquitous throughout this paper, we denote it just by $\Farey$, its set of edges by $E$, and its set of vertices by $V$. Unless otherwise specified, the edges in $E$ are unoriented and denoted $e=(a,b)$ where $a,b\in V$ are the endpoints of $e$.

Let $\tri_0$ be the ideal triangle with vertices $(1,\ii,-1)$. From $\tri_0$, all the other triangles in $\Farey$ are images of $\tri_0$ by reflections over edges which are hyperbolic (orientation reversing) isometries.
The dual tree to the Farey tessellation, which we denote $\dualtree$, will also play a central role in our discussions. Every triangle $\tri\in \Farey$ corresponds to a vertex $\tri^*$ in $\dualtree$, and every vertex in $\Farey$ corresponds to a face $v^*$ in $\m D \smallsetminus \dualtree$. Every edge $e\in E$ corresponds to the dual edge $e^*\in E^*$. We call the  edge $e_0^*$ dual to $(-1,1)$ the \textit{root edge} of $\dualtree$. 

The dual edges are sorted into \textit{generations} based on their graph distance to the $e_0^*$. We write $E_n^*\subset E^*$ for the set of dual edges within distance $n$ of $e_0^*$. From this we extend the definition of generations to the vertices, edges, and faces of $\Farey$ and $\dualtree$. We define $E_n\subset E$ to be the collection of edges dual to $E_n^*$, $V_n\subset V$ to be the vertices which are endpoints of edges in $E_n$, and $T_n^*$ to be the vertices of $\dualtree$ which are vertices of $E_n^*$. We then extend the definition of generation via duality to faces of $\Farey$ and $\dualtree$.
We say that a vertex, an edge or a face in $\Farey$ or in $\dualtree$ \emph{has generation $n$} if it belongs to the corresponding set of index $n$ but not $n-1$, and we write $\gen(\cdot)$ for the generation function.
It is easy to see:
\begin{lem}\label{lem:same-gen}
If an edge $e = (a,b) \neq e_0$, then $\gen (a) \neq \gen (b)$.
\end{lem}

The Farey tessellation is nicely represented in  the upper half plane $\mathbb{H}$.
We choose the following Cayley map 
$\cayley$ which maps $\mathbb{D}$ conformally onto $\mathbb{H}$: 
\begin{align*} 
\cayley: z\mapsto -\ii \frac{z+1}{z-1}, \quad \text{which maps } -1 \mapsto 0, \, 1 \mapsto \infty, \, \ii \mapsto -1.
\end{align*}
Under this identification, $V$ is sent to $\mathbb{Q}\cup \{\infty\}$, therefore, $V = \m Q^2 \cap \m T$. 
In fact, the modular group $$\PSL(2,\m Z) = \left \{ A
= \begin{psmallmatrix} a & b \\ c & d
\end{psmallmatrix} \colon a,b,c,d \in \m Z, \, ad - bc = 1\right\}_{/A\sim - A}$$
acts on $\m H$ by fraction linear transformations
\begin{equation} \label{eq:mobius}
z \mapsto \frac {az +b}{cz+d} =: A(z).
\end{equation}
The image of the Farey tessellation under $\cayley$ contains the ideal triangle $\cayley(\tri_0)$ of vertices $\{-1,0,\infty\}$ and is preserved by the action of $\PSL(2,\mathbb Z)$ on $\mathbb{H}$ which is generated by maps $z \mapsto z+1$ and $z \mapsto - 1/z$. It is then not hard to see that $\PSL(2, \m Z)$ acts transitively on $\cayley(E)$ (faithfully on oriented edges) and $\cayley(V) = \mathbb{Q} \cup \{\infty\}$.

Each element in $\m Q \smallsetminus \{0\}$ will be written in the form $\frac pq$, where $q \in \m Z_{\ge 1}$, $p \in \m Z$, and $p,q$ are co-prime. We use the convention $0 = \frac 01$ and $\infty = \frac 10 = \frac{-1}{0}$.
We recall the following basic fact about the Farey tessellation in $\m H$.  For readers' convenience we also include the elementary proof of this classical result.

\begin{figure}[ht]
\centering
\includegraphics[scale=0.45]{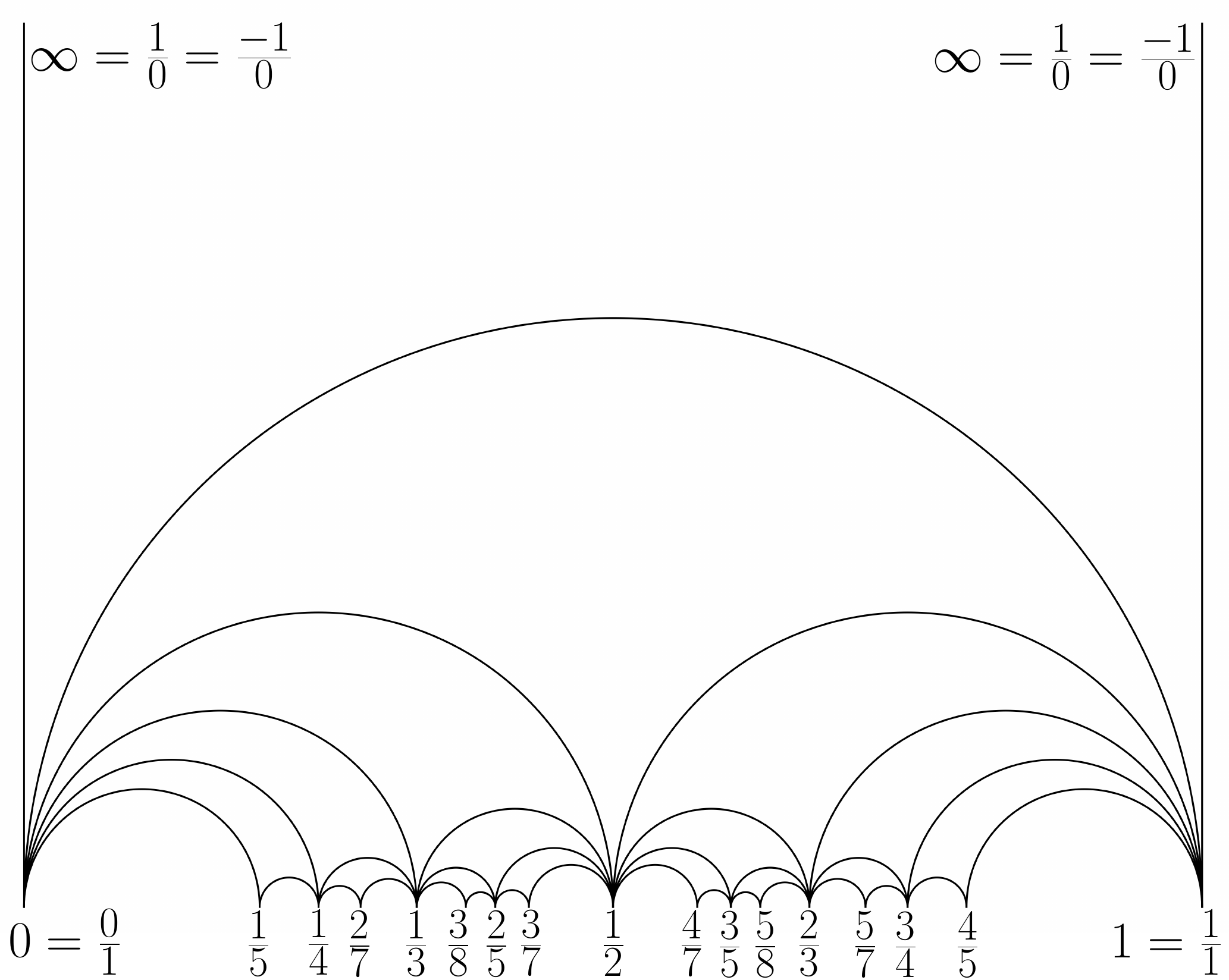}
\caption{\label{fig:Farey_H} Farey tessellation in $\m H$ between $0$ and $1$ up to generation 4 with vertices labeled.}
\end{figure}

We say that $c \in V$ is a \emph{child} of $(a,b)\in E$, if $c$ has one generation larger than $(a,b)$ and $(a,b,c)$ is a triangle in $\Farey$. Apart from $(-1, 1)$ which has two children $\{\ii, -\ii\}$, all other edges have only one child.

\begin{lem}\label{lem:farey-children}
An edge $(\frac pq,\frac rs) \in \cayley(E)$, if and only if $\abs{ps -rq} = 1$. Moreover, $p+r$ and $q+s$ are co-prime and the child of $(\frac pq,\frac rs)$ is $\frac{p+r}{q+s}$. Here, we choose the convention $\infty = \frac 10$ if $p \ge 0, q = 1$, and $\infty = \frac{-1}0$ if $p \le 0, q = 1$ \textnormal(for $p = 0$, we use both conventions\textnormal). 
\end{lem}

\begin{proof}
Assume that $(\frac pq, \frac rs) \in \cayley(E)$. Since $\PSL(2,\m Z)$ acts transitively on $\cayley(E)$, there exists an element $A \in \PSL(2,\m Z)$, such that $A (\infty) = \frac pq $ and $A(0) = \frac rs$. Hence $A =  \begin{psmallmatrix} \a p & \b r \\ \a q & \b s\end{psmallmatrix}$ for some $\a, \b \in \m Z$. Since $\det A = 1 = (ps -rq) \a\b$ and $p, r, q, s \in \m Z$, we have $\a, \b  \in \{1, -1\}$ and $\abs{ps -rq} = 1$.
Conversely, if $\abs{ps- rq} = 1$, we let $A =  \begin{psmallmatrix} \a p & r \\ \a q & s\end{psmallmatrix}$ with $\a \in \{ 1, -1\}$ such that $A \in \PSL(2, \m Z)$. Then $(\frac pq,\frac rs)$ is the image of $(0,\infty)$ under the fractional linear transformation $A$. Therefore $(\frac pq,\frac rs) \in  \cayley(E)$.

 Now we compute the child of $(\frac pq, \frac rs)$. We treat first the case where $s$ or $q = 0$. By symmetry, we assume that $s = 0$ and $q \neq 0$.  Using our convention, this happens only if $p \ge 0$, $q = 1$, $r = 1$ (or resp. $p \le 0$, $q = 1$, $r = -1$). The child of $(\infty = \frac 10, p)$ is $p+1$ and child of $(\infty = \frac{-1}{0},p)$ is $p-1$ as claimed. 
 
 Now we consider the case where $s,q \neq 0$. By symmetry, we assume that $ps - rq =  1$. Note that 
$$ \frac pq = \frac{1}{sq} + \frac rs, $$
$s,q \ge 0$ implies that $\frac rs < \frac pq$.

The matrix $A' =  \begin{psmallmatrix} p + r & r \\ q + s & s\end{psmallmatrix} \in \PSL(2, \m Z)$. Therefore $p+r$ and $q+s$ are co-prime. The previous result shows that $\frac{p+r}{q+s}$ is adjacent to $\frac rs$. Consider similarly the matrix $A'' =  \begin{psmallmatrix} p  & p+ r \\ q &  q+ s\end{psmallmatrix}$, we obtain that $\frac{p+r}{q+s}$ is adjacent to $\frac pq$. Moreover, we have the inequalities
$$ \frac rs <\frac{p+r}{q+s}< \frac pq,$$
which shows that $\frac{p+r}{q+s}$ is the child of the edge $(\frac pq, \frac rs)$. 
\end{proof}

\subsection{Shear along an edge}\label{sec:shear}

Let $e$ be a hyperbolic geodesic in the disk connecting $a,c\in \m T$. A quad $Q$ around $e$ is an ideal quadrilateral in $\m D$ with 
vertices $a,b,c,d\in\m T$ in counterclockwise order, for some $b,d \in \m T$. 
Recall the \emph{cross ratio} of 
$a,b,c,d$ is 
\begin{align*}
    \Cr(a,b,c,d) = \frac{(b-a)(d-c)}{(c-b)(d-a)}.
\end{align*}

\begin{df}
The \emph{shear} of $Q = (a,b,c,d)$ along $e = (a,c)$ is $s(Q, {e}) := \log \Cr(a,b,c,d)$.
\end{df}
The shear of $Q$ along a diagonal $e$ does not depend on the orientation of $e$ since $\Cr(a,b,c,d) = \Cr(c,d,a,b)$. The cross ratio is also invariant under M\"obius transformations, and hence so is the shear of a quad around an edge. We can use the Cayley transform $\cayley$ to easily compute the shear for quads around $e_0 = (-1,1)$. 
\begin{ex}\label{ex: e0quad}
Consider a quad around the edge ${e}_0 = (-1,1)$ of the form $Q = \{1,\ii,-1,x_s\}$. Under the Cayley transform $\cayley: (1,\ii,-1,x_s) \mapsto (\infty, -1,0,\cayley(x_s))$. Since the cross ratio is preserved by M\"obius transformations, we get that
\begin{align*}
    s(Q,{e}_0) = \log \Cr(\infty, -1,0,\cayley(x_s)) = \log \cayley(x_s).
\end{align*}
\end{ex}

While $s(Q,e)$ does not depend on the orientation of $e$, orienting $e$ is useful for the following geometric interpretation of the shear (which also explains the name). The quad $Q$ can be thought of as two triangles glued along $e$. Choosing an orientation of ${e}$, we call the triangle on the left of ${e}$ (when $e$ is pointing up) with respect to the orientation $\tri_L$ and the other $\tri_R$. Geometrically, the shear of $Q$ along ${e}$ measures how $\tri_L,\tri_R$ are glued together along ${e}$ to construct $Q$. 

\begin{lem} \label{lem:single_shear_hyperbolic}
Let $\vec{e}=(c,a)$ be oriented from $c$ to $a$ so that $b\in \tri_L$ and $d\in \tri_R$. For any point $x\in \m T$, define $m_x(e)$ to be the intersection between $e$ and the hyperbolic geodesic through $x$ perpendicular to $e$. Then 
\begin{align*}
    s(Q,{e}) = \pm d_{\text{hyp},\m D}(m_b({e}), m_d(e))
\end{align*}
The sign is positive if $m_b$ is before $m_d$ along $\vec{e} = (c,a)$ and negative otherwise. See Figure~\ref{fig:single_shear}.
\end{lem}

\begin{proof}
    Let $\vec{e} = (c,a)$ and $\vec{e}_0 =(-1,1)$ as oriented edges. Let $A$ be a M\"obius transformation that sends 
    $a,b,c$ to $1,\ii,-1$ respectively. Under this map, $\vec{e}$ is sent to $\vec{e}_0$, the quad $Q = (a,b,c,d)$ around $\vec{e}$ is sent to the quad $Q' = (1,\ii,-1,A(d) =:x_s)$ around $\vec{e}_0$, and the intersection points $m_b(e),m_d(e)$ are sent to the intersection points $m_{\ii}(e_0),m_{x_s}(e_0)$ respectively. Since shears are preserved under M\"obius transformations and using the calculation in Example \ref{ex: e0quad}, \begin{align*}
        s(Q,{e}) = s(Q',{e}_0) = \log \cayley(x_s).
    \end{align*}
    Since hyperbolic lengths are also preserved by M\"obius transformations, 
    \begin{align*}
        d_{\text{hyp},\m D}(m_{b}(e),m_{d}(e)) &= d_{\text{hyp},\m D}(m_{\ii}(e_0), m_{x_s}(e_0)).
    \end{align*}
    We can use the Cayley transform again to compute this distance. Under $\cayley$, $Q' = \{1,\ii,-1,x_s\}\mapsto \{\infty, -1,0,\cayley(x_s)\}$, $m_{\ii}(e_0) \mapsto \ii$, and $m_{x_s}(e_0) \mapsto \cayley(x_s)\ii$. Hence 
    \begin{align*}
        d_{\text{hyp},\m D}(m_{\ii}(e_0), m_{x_s}(e_0)) = d_{\text{hyp},\m H}(\ii, \cayley(x_s)\ii) = \bigg\lvert \int_1^{\cayley(x_s)} \frac{1}{y}\,\dd y \bigg\rvert = |\log \cayley(x_s)|.
    \end{align*}
    As for the sign, under the composition $\cayley\circ A$, $\vec{e}$ is sent as an oriented edge to $(0,\infty)$. Hence $m_b(e)$ is before $m_d(e)$ along $\vec{e}$ if and only if $1<\cayley(x_s)$.
\end{proof}

\begin{figure}[ht]
    \centering 
 \includegraphics[scale=.75]{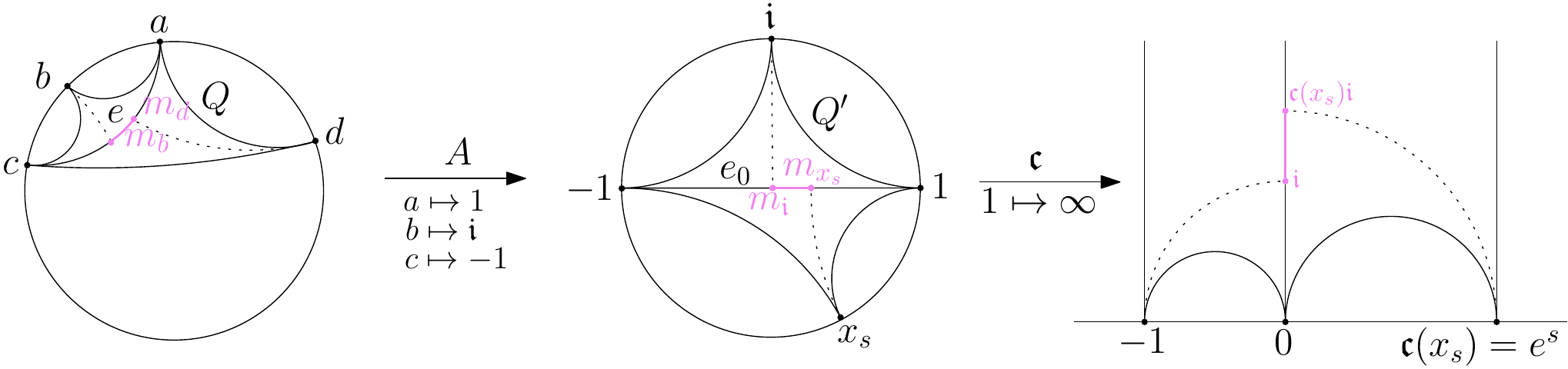}    
 \caption{\label{fig:single_shear} The quad $Q = \{a,b,c,d\}$ around $e=(c,a)$ with intersection points drawn, sent to $Q' = \{1,\ii,-1,x_s\}$ around $e_0 = (-1,1)$ by $A$, and then the image under the Cayley transform $\cayley$.}
\end{figure}

The pair of triangles around an edge $e$ in a tessellation forms a quad around $e$. We define the \textit{Farey quad}, denoted $Q_e$, to be the quad around $e$ in $\Farey$. The Farey tessellation is characterized by 
\begin{align*}
    s(Q_e,e) = 0, \qquad \text{for all } e\in E.
\end{align*}
This follows from the construction of the Farey tessellation via reflection. 

\subsection{Shear coordinates and 
classes of circle homeomorphisms}\label{sec:shear_coordinate_wp}

We now introduce shear coordinates for orientation-preserving homeomorphisms of the circle $\m T$. 
As one would often identify two homeomorphisms up to post-compositions by M\"obius transformations, namely by the group $\mob(\m T) \simeq \PSU(1,1) \simeq \PSL(2,\m R)$, e.g., 
in the context of Teichm\"uller theory, we assume throughout the paper that all circle homeomorphisms fix $\{-1, \ii, 1\}$ unless otherwise specified. 

If $h \colon \m T \to \m T$ is a circle homeomorphism, 
$h$ induces a map $V \to V_h = h (V)$ by sending $v \mapsto h(v) \in \m T$. For each edge $e \in E$ with end points $v_1$ and $v_2$, we write $h(e)$ for the hyperbolic geodesic connecting $h(v_1)$ and $h(v_2)$.
The image of $E$ under $h$ forms a new tessellation $h(\Farey)$ of the unit disk where the ideal triangle $\tri_0$ is fixed. 
 We define similarly $h(Q_e)$ to be the quad around $h(e)$ in $h(\Farey)$.

\begin{df}
The \emph{shear coordinate} of $h$ is the map $s_h \colon E \to \m R$ such that $s_h(e)$ is the shear of the edge $h(e)$ in the quad $h (Q_e)$. Namely, $s_h(e) = s(h(Q_e), h(e))$.
\end{df}
Notice that $s_h$ is determined by $h|_V$. However, not all functions $s: E \to \m R$ arise as a shear coordinate for some circle homeomorphism. In fact, from $s$ we recover a map $h : V \to \m T$ which is strictly increasing and fixes $1, \ii, -1$ such that $s_h = s$. However, $h (V)$ does not have to be dense in $\m T$ and cannot always be extended continuously to a homeomorphism.
In \cite{Saric_circle} the first author characterized the class of shear functions which arise from a circle homeomorphism, as well as those from quasisymmetric and symmetric homeomorphisms. See also \cite{Saric_new}. 
In the present article, we are particularly interested in the following classes of circle homeomorphisms.
\begin{df}[See \cite{TT06,shen13}]\label{def:WP}
A circle homeomorphism $h$ is called \emph{Weil--Petersson} if $h$ is absolutely continuous (with respect to arclength measure) and $\log h'$ belongs to the Sobolev space $H^{1/2}$.
In other words, 
\begin{equation}\label{eq:H1/2_WP}
    \iint_{\m T \times \m T} \abs{\frac{\log h'(\zeta) - \log h'(\xi)}{\zeta - \xi}}^2 \,\dd \zeta \dd \xi < \infty.
\end{equation}
\end{df}
We write $\WP(\m T)$ for the class of all Weil--Petersson homeomorphisms. The Weil--Petersson class has been studied extensively since the 80s because of its rich geometric structure and links to string theory \cite{bowick1987holomorphic,Witten,NagVerjovsky,pekonen}, Teichm\"uller theory \cite{Cui,Guo,TT06,shen13,Shen-Tang,Shen-Tang-Wu}, computer vision \cite{sharon20062d}, periodic KdV equations \cite{STZ_KdV}, and more recently, the discovery of links to SLE \cite{W2,W3,VW1,VW2}, hyperbolic geometry \cite{bishop-wp}, Coulomb gases \cite{johansson2021strong,WiegmannZabrodin_2022}, etc. See, e.g., \cite{bishop-wp} by Bishop for a survey as well as a number of new characterizations of the Weil--Petersson class.

Every quasisymmetric circle homeomorphism admits a quasiconformal extension $\m D \to \m D$, see, e.g, \cite{lehto2012univalent}. For a quasisymmetric homeomorphism to be Weil--Petersson, it has to satisfy the following equivalent $L^2$ condition.

\begin{thm}[See \cite{TT06,shen13}]\label{thm:WP_Beltrami_L2}
A circle homeomorphism $h$ is Weil--Petersson if and only if there is a quasiconformal extension $f \colon \m D\to \m D$ of $h$, such that the Beltrami coefficient $\mu = \ad \partial f / \partial f$ satisfies
$$\norm{\mu}^2_2 = \int_{\m D} \frac{4|\mu (z)|^2}{(1-\abs {z}^2)^2} \dd A(z) < \infty,$$
where $\dd A$ is the Euclidean area measure.
\end{thm}

For $\a \in (0,1]$, we let $\mc C^{1,\a}$ denote the class of circle homeomorphisms $h$ such that $\log h'$ is $\a$-H\"older continuous. Or equivalently, in terms of the H\"older classes $C^{1,\a}$,
\begin{equation}
    \mc C^{1,\a} = \{ h \in C^{1,\a} : h \text{ is a circle homeomorphism and } \inf_{\m T} |h'| > 0\}.
\end{equation}
In particular, it follows from the Kellogg theorem \cite[Thm.\,II.4.3]{GM} that the welding homeomorphisms of $C^{1,\a}$ Jordan curves belong to $\mc C^{1,\a}$. It is easy to see from \eqref{eq:H1/2_WP} the following lemma.
\begin{lem}\label{lem:C_alpha_in_WP}
We have $\mc C^{1,\a} \subset \WP(\m T)$ for all $\a > 1/2$ and $\mc C^{1,1/2} \nsubseteq \WP(\m T)$.
\end{lem}

\section{Diamond shear}\label{sec:diamond}
\subsection{Circle homeomorphisms with finite shear}

We will introduce a new shear coordinate system (diamond shears), essential to describe the class $\mc H$ of circle homeomorphisms at the center of this work (see Definition~\ref{df:H_class}). To motivate the definition of diamond shears, let us first consider circle homeomorphisms with finitely many nonzero shears (i.e. $s_h$ has finite support).

We write 
$$\PSU(1,1) = \left\{\begin{psmallmatrix} \a & \b \\ \ad \b & \ad \a
\end{psmallmatrix} \colon \a, \b \in \m C, \, |\a|^2 - |\b|^2 = 1\right\}_{/A\sim - A}$$
for the group of M\"obius transformations preserving $\m T$. 
This group is conjugate to $\PSL(2,\m R)$ via the Cayley transform $\cayley$.
For two distinct points $z, w \in \m T$, we denote by $I(z,w) \subset \m T$ the closed circular arc going counterclockwise from $z$ to $w$. 

\begin{lem}\label{lem:finite_shear_mobius}
A circle homeomorphism $h$ has finitely many nonzero shears if and only if $h$ is piecewise M\"obius with rational breakpoints. Namely, there exist $k \ge 2$ distinct points $v_1, \ldots, v_k \in V$ in counterclockwise order such that $h|_{I(v_i, v_{i+1})} \in \PSU(1,1)$, where $v_{k+1} = v_1$. 
\end{lem}
\begin{proof}
Let $h$ be a piecewise M\"obius homeomorphism with break points $v_1, \ldots, v_k \in V$. By possibly subdividing further, we assume that $(v_i, v_{i+1}) \in E$.
We write $\mc E_i = \{(a,b) \in E \colon a, b\in I(v_i, v_{i+1})\}$.
If $(a,b) \in \mc E_i$ and $(a,b) \neq (v_i, v_{i+1})$, then the four vertices of the Farey quad $Q_{(a,b)}$  are all in $I(v_i, v_{i+1})$. Since $h|_{I(v_i, v_{i+1})}$ is M\"obius, which preserves the cross-ratio, $s_h(a,b) = 0$. 
We obtain that $s_h$ has finite support since $E \smallsetminus \bigcup_{i = 1}^k \mc E_i$ is finite. 

Conversely, if $s_h$ has finite support, then $\m T$ can be partitioned to $\bigcup_{i = 1}^k I(v_i, v_{i +1})$ such that $s_h|_{\mc E_i} \equiv 0$ for all $i = 1, \cdots, k$. By possibly further subdividing the intervals we may assume that $(v_i, v_{i+1}) \in E$ (and therefore in $\mc E_i$). Let $x_i \in I(v_i, v_{i+1})$ be the unique vertex which is adjacent to $v_i$ and $v_{i+1}$ and $\tilde h_i$ the unique M\"obius map in $\PSU(1,1)$ which maps respectively $v_i, x_i, v_{i+1}$ to $h(v_i), h(x_i), h(v_{i+1})$. 
The image of $V \cap  I(v_i, v_{i+1})$ by $h$ is determined by $s_h|_{\mc E_i}$ and the image of one triangle. 
Therefore $h = \tilde h_i$ on $V \cap  I(v_i, v_{i+1})$. As $\tilde h_i$ is continuous and $V$ is dense in $\m T$, we obtain that $h = \tilde h_i$ on $I(v_i, v_{i+1})$.
\end{proof}
\begin{remark}
Notice that in the proof of the converse direction, we showed that the finitely supported shear coordinate defines a map $h: V \to \m R$ which extends continuously to a piecewise M\"obius homeomorphism (and we do not need to assume that $h$ is a homeomorphism to start). 
\end{remark}

To state the next result, we organize the edges of $\Farey$ into \textit{fans}. We define 
\begin{align*}
    \fan(v) = \{e\in E: e\text{ is incident to } v\}.
\end{align*}
We will index the edges in $\fan(v)$ as $(e_n)_{n \in \m Z}$ in a way that $n$ increases in the counterclockwise manner and $e_0$ is an arbitrary choice of edge in $\fan (v)$.  The order is chosen such that after mapping $\m D$ conformally onto $\m H$ and $v$ is sent to $\infty$, the image of $(e_n)_{n\in \m Z}$ are equally spaced vertical lines with index increasing from left to right.

We write $v\splus$ (resp. $v\sminus$) for a point on $\m T$ approaching $v$ infinitesimally counterclockwise (resp. clockwise), which corresponds to $x \to \infty$ (resp. $x \to -\infty$) in the upper half-plane model. For instance, $f(1\splus)$ means the limit of $f(z)$ as $z \in \m T$ approaches $1$ from below, if it exists.

\begin{df}\label{df:finite_balance}
We say that $s : E \to \m R$ satisfies the \emph{finite balanced condition} if for all $v \in V$,  $\{n \in \m Z \colon s(e_n) \neq 0\}$ is finite and $\sum_{n \in \m Z} s(e_n) = 0$, where $\fan(v) = (e_n)_{n \in \m Z}$. 
\end{df}
\begin{lem}\label{lem:finite_balance}
If a circle homeomorphism $h$ has finitely many nonzero shears, the following are equivalent: 
\begin{enumerate}[label=\roman*),itemsep=-3pt]
    \item $h$ is Weil--Petersson; \label{it:finite_wp}
    \item $h$ is $\mc C^{1,1}$ with rational breakpoints; \label{it:finite_C11}
    \item $s = s_h$  satisfies the finite balanced condition. \label{it:finite_balanced}
\end{enumerate}
\end{lem}

\begin{proof}
Since $h$ has finitely many nonzero shears, Lemma~\ref{lem:finite_shear_mobius} shows that $h$ is piecewise M\"obius with rational breakpoints that we denote by $v_1, \ldots, v_k \in V$. One and only one of the following is true: 
\begin{itemize}[itemsep=-2pt]
    \item For all $i = 1, \ldots, k$, $h'(v_i \splus) = h'(v_i \sminus)$. In this case, $h \in \mc C^{1,1}$. Lemma~\ref{lem:C_alpha_in_WP} shows that $h \in \WP(\m T)$.
    \item There exists $i$ such that  $h'(v_i \splus) \neq h'(v_i \sminus)$. In this case, $\log h'(v_i \splus) \neq \log h'(v_i \sminus)$. We see from \eqref{eq:H1/2_WP} that $h \notin \WP(\m T)$.
\end{itemize}
This proves the equivalence between \ref{it:finite_wp} and \ref{it:finite_C11}.

Now we show that $h'(v_i \splus) = h'(v_i \sminus)$ is equivalent to $\sum_{n \in \m Z} s(e_n) = 0$, where  $(e_n)_{n \in \m Z} = \fan(v_i)$. 
We define $\varphi : \m R \to \m R$ to be the homeomorphism $\varphi = \cayley_1 \circ h \circ \cayley_2^{-1}$, where $\cayley_1, \cayley_2$ are two M\"obius transformations $\m D \to \m H$ such that $\cayley_1 (h(v_i)) = \infty$, $\cayley_2 (v_i) = \infty$, $\cayley_2 (e_0) = (0,\infty)$, and $\cayley_2 (e_1) = (1,\infty)$. Given this, $\varphi$ fixes $\infty$, and $\cayley_2 (e_n) = (n,\infty)$ for all $n \in \m Z$. 

From Lemma~\ref{lem:single_shear_hyperbolic}, we know that 
\begin{equation}\label{eq:ratio_shear_sum}
\frac{\varphi(n+1) - \varphi(n)}{\varphi(n) - \varphi(n-1)} = \exp(s(e_n)).
\end{equation}
Since $s$ has finite support, there exists $n_0 \ge 0$ such that $s(e_n) = 0$ if $|n| \ge n_0$. Therefore, there exists $\ell, \ell' > 0 $ such that  for all $n,m \ge n_0$, $\varphi (m +1) - \varphi(m) = \ell$, $\varphi (-n) - \varphi(-n-1) = \ell'$, and 
$$\frac{\varphi(m+1) - \varphi(m)}{\varphi (-n) - \varphi (-n -1)} = \frac{\ell}{\ell'} =\exp(\sum_{n \in \m Z} s(e_n)).$$
We have
\begin{align*}
    \varphi(m+1) - \varphi (m)& = \int_m^{m+1} \abs{\varphi'(x)} \,\dd x \\
    &= \int_m^{m+1} \abs{\cayley_1'(h\circ \cayley_2^{-1}(x))} \abs{h'(\cayley_2^{-1}(x))} \abs{(\cayley_2^{-1})'(x)} \, \dd x.
\end{align*}
Since $\cayley_1, \cayley_2$ are M\"obius transformations $\m D \to \m H$ sending $v_i, h(v_i)$ to $\infty$ respectively, there exist $\a, \b >0 $ such that 
$$\abs{\cayley_1' (z)} = \frac{\a}{|z - h(v_i)|^2} + O\Big (\frac{1}{|z-h(v_i)|}\Big), \quad \abs{\cayley_2' (z)} = \frac{\b}{|z - v_i|^2} + O \Big(\frac{1}{|z-v_i|}\Big).$$
On the other hand, since $h$ is piecewise M\"obius, we have
$$|h' (z)| = |h'(v_i{\scriptstyle \pm})| + O(|z - v_i|),$$
depending on which side of $v_i$ does $z$ approaches from.
Therefore, for $x \in [n_0, \infty)$,
\begin{align*}
  & \abs{\cayley_1'(h\circ \cayley_2^{-1}(x))} \abs{h'(\cayley_2^{-1}(x))} \abs{(\cayley_2^{-1})'(x)} \\
    &=\left(  \a \abs{h \circ \cayley_2^{-1} (x) - h(v_i)}^{-2} + O\left(\abs{{h \circ \cayley_2^{-1} (x) - h(v_i)}}^{-1} \right)\right) (|h'(v_i \splus)| + O(|\cayley_2^{-1} (x) - v_i|)) \\
    & \qquad \Big(\frac{1}{\b} \abs{\cayley_2^{-1}(x) - v_i}^2 + O(|\cayley_2^{-1}(x) - v_i|^3)\Big)\\
    & = \frac{\a}{\b}\frac{1}{ |h'(v_i \splus)|} +O(|\cayley_2^{-1}(x) - v_i|).
\end{align*}
We obtain
\begin{align*}
   \ell =  \lim_{m \to \infty} \varphi(m+1) - \varphi (m)& = \frac{\a}{\b}\frac{1}{ |h'(v_i \splus)|}.
    \end{align*}
    Similarly,
    \begin{align*}
   \ell' =  \lim_{n \to \infty} \varphi(-n) - \varphi (-n-1)& = \frac{\a}{\b}\frac{1}{ |h'(v_i \sminus)|}.
    \end{align*}
    Equation \eqref{eq:ratio_shear_sum} then shows that $h'(v_i \splus) =h'(v_i \sminus)$ if and only if $\sum_{n \in \m Z} s(e_n) = 0$ which concludes the proof.
\end{proof}

We now introduce the diamond shear coordinates  which are well-adapted to describe finite shears with the finite balanced condition. 
We say that $e, e' \in E$ are \emph{adjacent} if $e$ and $e'$ share a vertex $v$ and are consecutive in $\fan(v)$ (their indices in $\fan (v)$ differ by exactly $1$). We say that $e^* \in E^*$ and $e' \in E$ are \emph{adjacent} if the dual edge $e$ of $e^*$ is adjacent to $e'$.
Note that $e^*$ is adjacent to exactly $4$ edges in $E$, that we denote by $(e_1, e_2, e_3, e_4)$, in counterclockwise order, such that $e_1$ is the edge on the left of $\vec e$ and has the same head as $\vec e$. We do not distinguish between $(e_1,e_2,e_3,e_4)$ and $(e_3,e_4,e_1, e_2)$ which results from inverting the orientation of $e$. 

A \emph{diamond shear function} is a map $\vartheta : E^* \to \m R$. The space of diamond shear functions has a basis 
$\{\vartheta_{e^*}(\cdot)\}_{e^* \in E^*}$, where $\vartheta_{e^*} (\cdot)$ equals $1$ at $e^*$ and $0$ elsewhere.
For each $e^* \in E^*$, the corresponding shear function
of $\vartheta_{e^*}$ has $s(e_1) = s(e_3) = 1$ and $s(e_2) = s(e_4) = -1$. See Figure~\ref{fig:singlediamond}.
More generally, the diamond shear functions translate to the shear functions via the map $\Phi : \m R^{E^*} \to \m R^E$:
\begin{equation}\label{eq:def_Phi}
s(e) = \Phi (\vartheta) (e) = -\vartheta (e_1^*) + \vartheta (e_2^*) - \vartheta (e_3^*) + \vartheta (e_4^*), \quad \forall e \in E.
\end{equation}
Here, $(e_i^*)_{i = 1, \ldots, 4}$ are the dual edges of $(e_i)_{i = 1, \ldots, 4}$ ordered as described above.

\begin{figure}[ht]
    \centering
 \includegraphics[scale=.7]{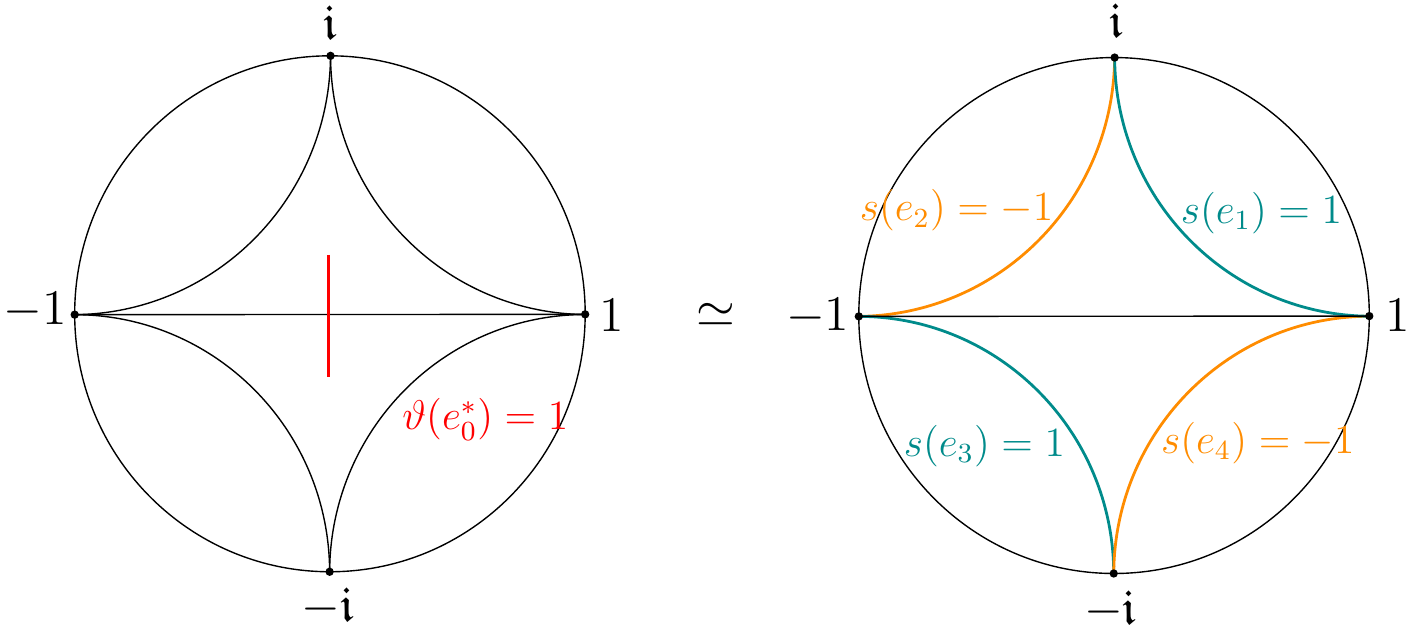}    
 \caption{ \label{fig:singlediamond} A single diamond shear along the edge $(-1,1)$ (left) is equivalent to four shears with alternating sign (right).}
\end{figure}

\begin{remark} \label{rem:kernel}
\begin{itemize}[itemsep=-2pt]
    \item We note that $\Phi$ is linear and if $\vartheta \equiv 1$, then $\Phi (\vartheta) \equiv 0$. In other words, constant diamond shear coordinates belong to the kernel of $\Phi$. 
    \item If $\vartheta$ has finite support, then $\Phi(\vartheta)$ has finite support and satisfies the finite balanced condition. The converse follows from the next proposition.
\end{itemize}
\end{remark}

\begin{prop}\label{prop:s_to_theta_finite}
    Assume that $s : E \to \m R$ has finite support and satisfies the finite balanced condition. There exists a unique $\vartheta : E^* \to \m R$ with finite support such that $\Phi(\vartheta) = s$.
\end{prop}
\begin{proof}
Every $e \in E$ belongs to two triangles which are dual to two vertices in $\dualtree$. Let $\tri^* (e)$ be the dual vertex that has the lower generation, if $e \neq e_0$. We define $\mf a_s$ to be the union of $e_0^*$ and the geodesic path from $\tri^* (e)$ to $e_0^*$ in $\dualtree$ for all $e \neq e_0$ such that $s(e) \neq 0$. (We call $\mf a_s$ the \emph{convex hull} of $\{\tri^* (e) \colon s(e) \neq 0\} \cup e_0^*$.) Since $s$ has finite support, $\mf a_s$ is a finite tree. 

We prove the existence of $\vartheta$ by induction on $\mf a_s$. 
If $\mf a_s$ contains only $e_0^*$, then $\{ e \colon s(e)\neq 0\} \subset \{(-1,1), (\ii, 1), (\ii, -1), (-\ii, 1), (-\ii, -1)\}$. The finite balanced condition shows that the only possibility is
$$s(\ii, 1) = -s(\ii, -1)= s(-\ii, -1)= -s(-\ii, 1) = \a$$
for some $\a \in \m R$. For convenience, we write $s(a,b)$ for $s((a,b))$.
Therefore, $s = \Phi (\a \vartheta_{e_0^*})$.

Now assume that $\mf a_s$ is a general finite tree containing $e_0^*$ and $e^*$ is a leaf of $\mf a_s$. Assume that $e^*$ has generation $n$. The dual edge $e \in E$ has two vertices $\{a,c\}$, their child $b \in V$ has generation $n+1$. We assume that $a,b,c$ are in the counterclockwise order.
Fan$(b)$ contains at most two edges, $(a,b)$ and $(b,c)$, on which $s$ is nonzero. 
From the finite balanced condition, there is $\a =\a (e^*)  \in \m R$ such that 
$$\a(e^*) = s(a,b) = - s(b,c).$$
Therefore $s' := s - \Phi (\a(e^*) \vartheta_{e^*})$ is a shear coordinate with finite support, and $\mf a_{s'} = \mf a_s \smallsetminus e^*$. By the assumption of induction, let $\vartheta'$ be a finite support diamond shear coordinate such that $\Phi (\vartheta') = s'$. 
The linearity of $\Phi$ shows that $\Phi (\vartheta' + \a(e^*) \vartheta_{e^*})  = s$.

Now we show the uniqueness. Assume that $\vartheta$ and $\vartheta'$ have finite support and $\Phi (\vartheta) = \Phi (\vartheta')$. Then $\Phi (\vartheta - \vartheta') \equiv 0$. Let $\mf a$ be the convex hull of $\{e^* \in E^* \colon \vartheta (e^*) \text{ or } \vartheta'(e^*) \neq 0\}.$ The above argument shows that for any leaf $e^*$ of $\mf a$, $\vartheta(e^*) - \vartheta'(e^*) = 0$. By induction, we have $\vartheta = \vartheta'$ which concludes the proof.
\end{proof}

\begin{df} \label{df:finite_theta_coor}
If a circle homeomorphism $h$ satisfies the conditions in Lemma~\ref{lem:finite_balance}, the \emph{diamond shear coordinate} of $h$ is the unique finitely supported diamond shear function $\vartheta_h$ such that $\Phi(\vartheta_h) = s_h$.
\end{df}

\subsection{Examples and developing algorithm}\label{sec:examples}

In this section we provide a few explicit examples to show concretely how circle homeomorphisms are related to shear and diamond shear coordinates.  Recall that for $a \neq b \in \m T$, $I(a,b)\subset \m T$ denotes the circular arc going counterclockwise from $a$ to $b$.

\begin{ex}[Single shear] \label{ex:single_shear}
For $t\in \m R$, let $h_t$ be the normalized circle homeomorphism with shear coordinate $s_t (e_0) = t$, and $s_t (e) = 0$ for all $e \in E$, $ e\neq e_0 = (-1,1)$. Then 
$$\cayley \circ h_t \circ \cayley^{-1} (x)= \begin{cases}
x & \qquad \forall x \le 0\\
e^{t} x &  \qquad \forall x \ge 0.
\end{cases}$$ 
In other words, 
$$ h_t(z)= \begin{cases}
z & \qquad \forall z \in I(1,-1)\\
\dfrac{\a_t z + \b_t}{\overline{ \b_t} z + \overline {\a_t}} &  \qquad \forall z \in I(-1,1)
\end{cases}$$ 
with $\a_t = \cosh (t/2)$ and $\b_t =  \sinh(t/2)$. In particular $(h_t)_{t\in \m R}$ forms a one-dimensional subgroup of the group of piecewise $\PSU(1,1)$ circle homeomorphisms. 

For $e = (a,b)$, not necessarily an edge of $\Farey$, there exists $A \in \PSU(1,1)$ such that $A (a) = -1$, $A(b) = 1$. Then 
$h_{e,t} := A^{-1} \circ h_t \circ A$ is also a one-dimensional subgroup of (non-normalized) circle homeomorphisms (and independent of the choice of $A$) fixing the circular arc $I(b,a)$. Explicitly,
$$h_{e,t} =\begin{cases}
z & \qquad \forall z \in I (b, a) = A^{-1} I(1,-1)\\
A_t (z) &  \qquad \forall z \in I (a, b) = A^{-1} I(-1,1)
\end{cases}$$ 
where $A_t = A^{-1} \begin{psmallmatrix} \a_t & \b_t \\ \overline {\b_t} & \overline{\a_t}
\end{psmallmatrix} A$. We note that $h_{e,t}$ is not $C^1$ if $t \neq 0$ and we find 
$$h_{e,t}'(a\splus) = 1, \quad h_{e,t}'(a\sminus) = e^t, \qquad h_{e,t}'(b\splus) = e^{-t}, \quad h_{e,t}'(b\sminus) = 1,$$
where $ a \splus$ means approaching $a$ counterclockwisely, and $a \sminus$ clockwisely.
\end{ex}

\begin{ex}[Standard single diamond shear]\label{ex:standard_diamond}
For $t\in \m R$, let $H_t$ be a 
normalized circle homeomorphism satisfying the conditions in Lemma~\ref{lem:finite_balance} (we can hence talk about its diamond shear coordinates $\vartheta_t$). In particular, suppose it has diamond shear coordinates such that $\vartheta_t (e_0^*) = t$, and $\vartheta_t (e^*) = 0$ for all $e^* \in E^*$, $ e^*\neq e_0^*$. Then
the corresponding shear coordinate $S_t = \Phi(\vartheta_t)$ of $H_t$ is given by 
$$S_t (1, \ii) = t, \quad S_t (\ii,-1) = - t, \quad S_t (-1, -\ii) = t, \quad S_t (-\ii,1) = - t.$$
It is easy to see that $H_t$ fixes $1, \ii, -1, -\ii$.  We obtain the following explicit expression of $H_t (z)$ (by symmetry it suffices to compute $H_t$ on $I(1, \ii)$):
\begin{equation}\label{eq:single_diamond_explicit}
\begin{cases}
h_{(1, \ii), t} (z) = \dfrac{\a_{1,t} z + \b_{1,t}}{\overline{ \b_{1,t}} z + \overline {\a_{1,t}}}& \text{ with } \begin{cases}
\a_{1,t} = \cosh (\frac t2) - \ii \sinh (\frac t2) \\ \b_{1,t} = ( \ii-1) \sinh (\frac t2),
\end{cases} \hfill \forall z \in I(1, \ii);\\

h_{(\ii,-1), -t} (z) = \dfrac{\a_{2,t} z + \b_{2,t}}{\overline{ \b_{2,t}} z + \overline {\a_{2,t}}}& \text{ with } \begin{cases}
\a_{2,t} = \overline{\a_{1,t}} \\ \b_{2,t}  = -\overline{\b_{1,t}} ,
\end{cases}  \hfill  \forall  z \in I(\ii,-1);\\

h_{(-1, -\ii), t} (z) = \dfrac{\a_{3,t} z + \b_{3,t}}{\overline{ \b_{3,t}} z + \overline {\a_{3,t}}}& \text{ with } \begin{cases}
\a_{3,t}=\a_{1,t}  \\ \b_{3,t} = -\b_{1,t},
\end{cases}  \hfill \forall  z \in I(-1, -\ii);\\

h_{(-\ii,1), -t} (z) = \dfrac{\a_{4,t} z + \b_{4,t}}{\overline{ \b_{4,t}} z + \overline {\a_{4,t}}}& \text{ with } \begin{cases}
\a_{4,t} = \overline{\a_{1,t}} \\ \b_{4,t}  = \overline{\b_{1,t}},
\end{cases}  \hfill  \forall z \in I(-\ii,1).
\end{cases}
\end{equation}

We observe 
$$H_t'(1) = H_t'(-1) = e^t, \qquad H_t'(\ii) = H_t'(-\ii) = e^{-t}$$
and that $(H_t)_{t\in \m R}$ is a one-dimensional subgroup of the group of  $C^{1,1}$ and piecewise $\PSU(1,1)$ circle homeomorphisms.
\end{ex}

The circle homeomorphism satisfying the conditions in Lemma~\ref{lem:finite_balance} whose diamond shear is supported on a single dual edge $e^* \in E^*$ can be obtained by  $A^{-1} \circ H_t \circ A$ for some $A \in \PSU(1,1)$ up to normalization.  We can also define the homeomorphism associated to a diamond shear on a \emph{non-standard} quad.
 
 \begin{df}[Single diamond shear on non-standard quad]
  Let $Q$ be a quad with vertices $a, b, c, d \in \m T$ in counterclockwise order. (We do not require $Q$ is a quad in $\Farey$, in particular, $\Cr (a,b,c,d)$ might not be zero.) We define 
 $H_{Q, (a,c), t} \in C^{1,1}$ to be the (non-normalized) circle homeomorphism which fixes the vertices of $Q$, that is piecewise $\PSU(1,1)$ with break points at the vertices, and 
\begin{equation}\label{eq:non-standard}
H_{Q, (a,c), t}'(a) = e^{t}.
\end{equation}
 \end{df}
 \begin{remark}
  Since for any M\"obius transform $A$ and $x\neq y$, $A'(x)A'(y) = \frac{(A(x) - A(y))^2}{(x-y)^2}$, the $C^{1,1}$ condition and \eqref{eq:non-standard} uniquely determine $H_{Q, (a,c), t}$ on $\m T$ and give us
  $$H_{Q, (a,c), t}'(c) = e^t, \qquad H_{Q, (a,c), t}'(b) = H_{Q, (a,c), t}'(d) = e^{-t}.$$
From this we obtain 
\begin{align*}
H_{Q, (a,c), t} |_{I(a,b)} &= h_{(a,b),t}|_{I(a,b)}, \qquad H_{Q, (a,c), t} |_{I(b,c)} = h_{(b,c),-t}|_{I(b,c)},\\
H_{Q, (a,c), t} |_{I(c,d)} &= h_{(c,d),t}|_{I(c,d)}, \qquad H_{Q, (a,c), t} |_{I(d,a)} = h_{(d,a),-t}|_{I(d,a)}.
\end{align*}
This relation similar to \eqref{eq:single_diamond_explicit} justifies the name of \emph{homeomorphism associated to non-standard diamond shear} and also shows that $(H_{Q, (a,c), t})_{t \in \m R}$ is a one-dimensional subgroup of $C^{1,1}$ and piecewise $\PSU(1,1)$ circle homeomorphisms.  Note that this definition coincides with the one for the single diamond shear supported on an edge (as diagonal of a Farey quad) as in Example~\ref{ex:standard_diamond}.
 \end{remark}

 Non-standard diamond shears are useful in the following developing algorithm for finding the associated circle homeomorphism given the diamond shear coordinate inductively when it has finite support.
 
 \begin{prop}\label{prop:add}
 Let $h$ be a circle homeomorphism satisfying the conditions in Lemma~\ref{lem:finite_balance} and $\vartheta$ its \textnormal(finitely supported\textnormal) diamond shear coordinates. Let $t \in \m R$ and $e \in E$. The homeomorphism with diamond shear coordinate $\vartheta + t \vartheta_{e^*}$ is 
 $H_{h(Q_e), h(e), t} \circ h$ after normalizing to fix $-1, \ii, 1$.
 \end{prop}
 \begin{proof}
 Let $s = \Phi (\vartheta)$ denote the shear coordinate of $h$. 
 Let $h_t : = H_{h(Q_e), h(e), t} \circ h$ and $s_t$ its shear coordinate.
 We write the Farey quad $Q_e = (a,b,c,d)$ such that $e = (a,c)$. Let $e_1 = (a,b)$, $e_2 = (b,c)$, $e_3 = (c,d)$, and $e_4 = (d,a)$ be the adjacent edges in $\Farey$.
 We need to show that 
 \begin{equation}\label{eq:shear_add}
 S_t (e_j) : = s (h_t (Q_{e_j}), h_t (e_j))  = s (e_j) + (-1)^{j-1} t, \qquad j = 1, \cdots, 4. 
 \end{equation}
 We see it from the geometric interpretation of the shear (Lemma~\ref{lem:single_shear_hyperbolic}). 
In fact, $s (e_1) =  s (h (Q_{e_1}), h (e_1))$ is the signed distance between the geodesics normal to $h (e_1)$ and starting from the third vertex of the two ideal triangles that we call $\tau_L$, $\tau_R$, where $\tau_L \cup \tau_R = h(Q_{e_1})$ and $\tau_R \subset h(Q_e)$. 
Since  $ H_{h(Q_e), h(e), t}$ fixes $h(Q_e)$, it also fixes $\tau_R$. On the arc $I_1 \subset \m T$ which has the same vertices as $h(e_1) = (h(a), h(b))$ and contains the vertices of $\tau_L$,
$ H_{h(Q_e), h(e), t}$ shears further the normal starting from any point of $I_1$ by hyperbolic distance $t$ in the direction from $h(a)$ to $h(b)$. We obtain \eqref{eq:shear_add} for $j = 1$. See Figure~\ref{fig:add}.
The same argument works for other $j$. 
 \end{proof}

\begin{figure}[ht]
    \centering
 \includegraphics[width =\textwidth]{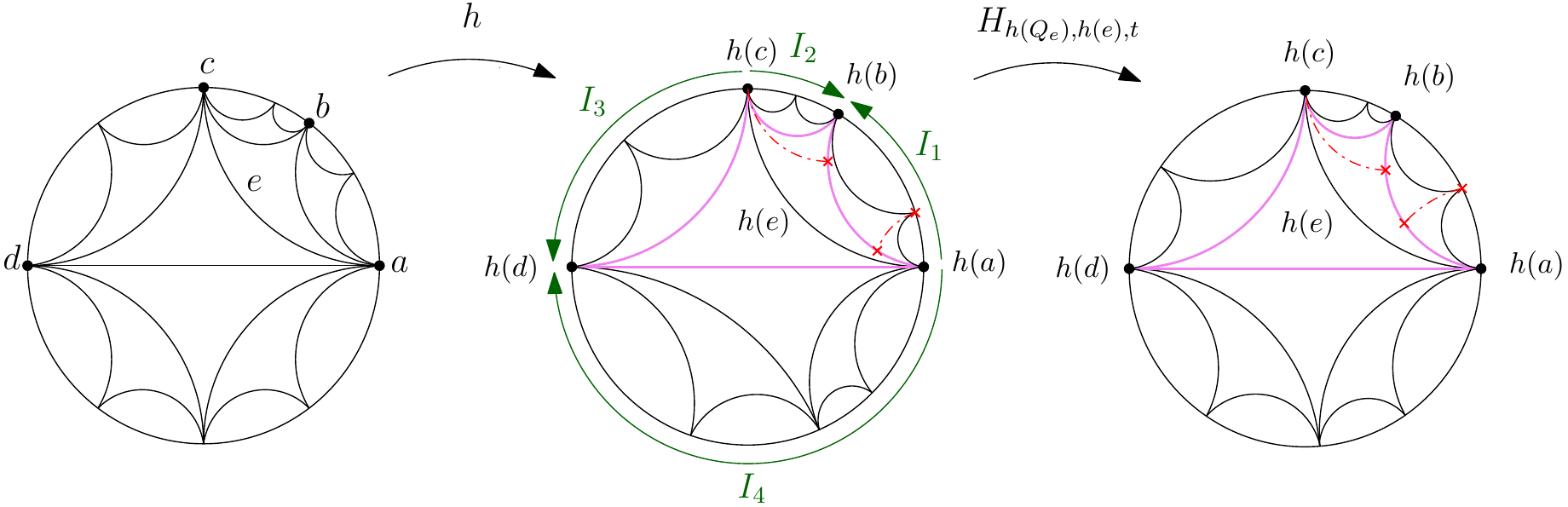}    
 \caption{ \label{fig:add} Illustration of the maps in the proof of Proposition~\ref{prop:add}. Left: Farey tessellation $\Farey$ with an edge $e \in E$ and $Q_e$ with vertices $a,b,c,d$ marked. Middle: $h(\Farey)$ with the edges of $h(Q_e)$ in pink, the green arrows indicate the direction in which of the piecewise M\"obius circle homeomorphism $H_{h(Q_e), h(e), t}$ moves the points on each arcs, when $t >0$. Red dashed lines indicate the normals to $h(e_1) = (h(a), h(b))$. Right: the tessellation $h_t (\Farey)$.}
\end{figure}

\subsection{Combinatorial definition of diamond shear} \label{sec:combinatorial_diamond}

The goal of this section is to extend the definition of diamond shear coordinates to a more general class of circle homeomorphisms. 
Definition~\ref{df:finite_theta_coor} suggests that $\vartheta$ should be defined as the image of $s$ by the inverse of $\Phi$ defined in \eqref{eq:def_Phi}. However, the map $\Phi$ does not map onto $\m R^{E}$, nor is it injective by Remark~\ref{rem:kernel}. 
Therefore, we will restrict to the following family of shear coordinates to define a right-inverse map of $\Phi$:
\begin{align*}
    \mc P = \{s \in \m R^E \colon \forall\, v\in V, \,\fan(v) = (e_n)_{n \in \m Z}, \,  \lim_{n\to \infty} \sum_{k=-n}^{-1} s(e_k) \in \mathbb{R} \text{ and } \lim_{n\to \infty} \sum_{k=0}^{n} s(e_k) \in \mathbb{R}\}.
\end{align*}
Similar to Definition~\ref{df:finite_balance}, we say that $s \in \mc P$ satisfies the (generalized) \emph{balanced condition}, if 
\begin{equation}\label{eq:P0}
s\in \mc P_0 =  \{ s \in \mc P \colon \forall\, v\in V, \,\fan(v) = (e_n)_{n \in \m Z}, \,   \sum_{k=-\infty}^{\infty} s(e_k) = 0\}.    
\end{equation}
Note that if $s$ satisfies the finite balanced condition, then $s \in \mc P_0$.

\begin{df}\label{df:Psi}
We define for $s\in \mc P$, $v \in V$, and $e \in \fan (v)$,
$$p_{s,v} (e\splus) = \sum_{e' >_v e} s(e'), \quad p_{s,v} (e\sminus) = \sum_{e' <_v e} s(e')$$
where $e'>_v e$ means that $e'\in \fan (v)$ and has strictly larger index than $e$, and similarly, $e' <_v e$ for strictly smaller index than $e$.
Define $\Psi : \mc P \to \m R^{E^*}$ by
\begin{equation}\label{eq:def_Psi}
\Psi (s) (e^*) = \frac{1}{4} \Big (p_{s,a} (e\sminus) -p_{s,a} (e\splus) + p_{s,b} (e\sminus)  - p_{s,b}(e\splus) \Big)
\end{equation}
where $e = (a,b)$ is dual to $e^*$. See Figure~\ref{fig:Psi} for an illustration of $\Psi(s)(e_0^*)$.
\end{df}

\begin{figure}[ht]
    \centering
 \includegraphics[scale=.45]{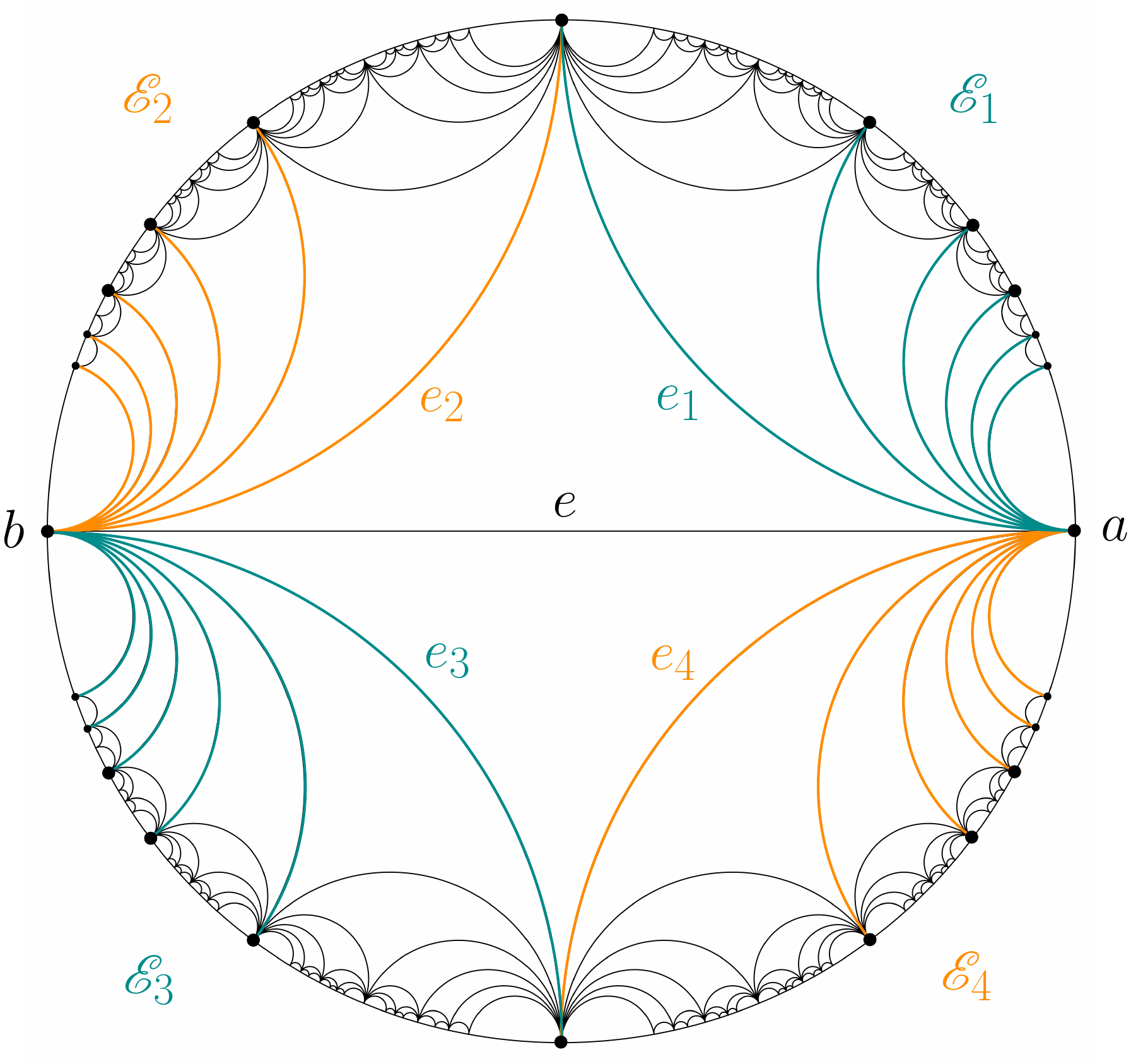}    
 \caption{ \label{fig:Psi} The shears on blue edges are counted as positive and the shears on orange edges are counted as negative in $\Psi(s)(e^*)$.}
\end{figure}

\begin{prop} \label{prop:right-inverse}
The function $\Psi$ is a right-inverse of $\Phi$, namely,
$\Phi \circ \Psi = \Id_{\mc P}$. 
\end{prop}

\begin{proof}

The maps $\Phi$ and $\Psi$ are both linear. Combining Equations \eqref{eq:def_Phi} and \eqref{eq:def_Psi}, $\Phi(\Psi(s))(e)$ is a sum of $p_{s,a}(e\spm)$ for sixteen different choices of $a,e,\pm$. Since $s\in \mc P$, the limits $p_{s,a}(e\spm)$ are well-defined for all $a\in V,e\in \fan(a)$, and hence we can switch the finite sum with sixteen terms with the limits defining $p_{s,a}(e\spm)$. Therefore $\Phi\circ \Psi$ is linear on $\mc P$ even for infinite linear combinations, so it is enough to show that $\Phi\circ \Psi(s_{e}) = s_{e}$ where $s_{e}$ takes value $1$ on the edge $e = (a,b) \in E$ and $0$ elsewhere. 

Let $e_1,e_2,e_3,e_4$ be the edges around $Q_e$ in counterclockwise order starting from $a$.
We denote the four half-fans around $Q_e$ in counterclockwise order by $\mc E_1 = \{e'\in \fan(a) : e'\leq_a e_1\}$, $\mc E_2 = \{e'\in \fan(b) : e' \geq_b e_2\}$, $\mc E_3 = \{e'\in \fan(b) : e' \leq_b e_3\}$, $\mc E_4 = \{e'\in \fan(a) : e'\geq_a e_4\}$. See Figure~\ref{fig:Psi}. To simplify notation, we identify the dual edges $(e')^*$ with the corresponding edge $e'$. By Equation \eqref{eq:def_Psi}
\begin{align*}
    \Psi(s_{e})(e') = \begin{cases} -1/4 &\quad e'\in \mc E_1\cup \mc E_3\\
    +1/4 &\quad e'\in \mc E_2\cup \mc E_4\\
    0 &\quad \text{ otherwise.}
    \end{cases} 
\end{align*}
To compute $\Phi(\Psi(s_{e}))(e')$ we look at the the values of $\Psi(s_e)$ on the edges $e_1',e_2',e_3',e_4'$ around $Q_{e'}$ and use Equation \eqref{eq:def_Phi}. 

If $e' = e$,
\begin{align*}
    \Phi(\Psi)(s_{e})(e') = -\bigg(-\frac{1}{4}\bigg)+\frac{1}{4} -\bigg(-\frac{1}{4}\bigg) + \frac{1}{4} = 1.
\end{align*}
For $e'\neq e$, we check that $\Phi(\Psi(s_{e}))(e') = 0$. 

If $e' \neq e_i$ for $i=1,2,3,4$, the edges around $Q_{e'}$ either all have diamond shear $\Psi (s_e) (\cdot)= 0$, or there are two consecutive edges with the same nonzero diamond shear followed by two edges with zero diamond shear. In both cases, $\Phi(\Psi(s_e))(e') = 0$. 

For $e'=e_i$ for $i=1,2,3,4$, one can check that two non-consecutive edges around $Q_{e'}$ have nonzero diamond shear coordinates of opposite sign, and the other two edges have diamond shear coordinate $0$. 
\end{proof}

\begin{df}\label{df:H_class}
We define the \emph{diamond shear coordinate} $\vartheta_h : = \Psi(s_h)$ of a circle homeomorphism $h$ if the shear coordinate $s_h \in \mc P$.
We let 
$$\mc S = \{s \in \m R^E \colon \sum_{e\in E} s(e)^2<\infty\} \text{ and } \mc H = \{s \in \mc P \colon \sum_{e^*\in E^*} \vartheta(e^*)^2<\infty
\text{ where } \vartheta = \Psi (s)\}.$$ 
We say that a circle homeomorphism $h$ has \emph{a square summable diamond shear coordinate} if $\vartheta_h \in \mc H$. We endow $\mc S$ and $\mc H$ with the topology of $\ell^2$ convergence in $s$ and $\vartheta$ respectively. 
\end{df}

In the finite support case, it follows from Proposition~\ref{prop:right-inverse} that $\Psi (s_h)$ is the diamond shear coordinate defined in Definition~\ref{df:finite_theta_coor}.
Here and in the rest of the paper we identify $E$ with $E^*$ to simplify the notation.

\begin{lem}\label{lem:H_in_P0}
Assume that $s \in \mc H$ and $\vartheta = \Psi(s)$. Then we have $s \in \mc S$ and 
\begin{equation}\label{eq:H_implies_P0}
\sum_{v \in V} \sum_{e \in \fan (v)} p_{s,v}(e \splus)^2 = 2 \sum_{e\in E} \vartheta(e)^2 + \sum_{e \sim e'} (\vartheta (e) - \vartheta(e'))^2< \infty,
\end{equation}
where  $e \sim e'$ means that $e$ and $e'$ are adjacent in the same fan. In particular, $s \in \mc P_0$.
\end{lem}

\begin{proof}
The Cauchy-Schwarz inequality, $s = \Phi (\vartheta)$, and the assumption of $\vartheta$ is square summable show that $s \in \mc S$.

Now we fix $v\in V$ and let $\fan (v) = (e_n)_{n\in \m Z}$. For $e_k \in \fan (v)$, we compute $p_{s,v} (e_k\splus) = \sum_{n = k+1}^\infty s(e_n)$.
For this, we write the edge of $\Farey$ connecting the endpoints of $e_n$ and $e_{n+1}$ other than $v$ as $e'_n$. 
Since $s = \Phi (\vartheta)$, \eqref{eq:def_Phi} shows that  for $m >k$,
\begin{align*}
\sum_{n = k+1}^m s(e_n) &=\sum_{n = k+1}^m  \vartheta (e_{n-1}) - \vartheta (e_{n-1}') +  \vartheta (e_{n}') - \vartheta(e_{n+1}) \\
&= \vartheta(e_k) + \vartheta(e_{k+1}) -\vartheta(e_k') + \vartheta (e_m') - \vartheta (e_m) - \vartheta(e_{m+1}).
\end{align*}
Since $\vartheta$ is square summable, we have  $\vartheta(e_m)$ and $\vartheta(e_m')$ converging to $0$ as $m \to \infty$. Hence,
$$p_{s,v}(e_k \splus) =  \lim_{m \to \infty} \sum_{n = k+1}^m s(e_n) =   \vartheta(e_k) + \vartheta(e_{k+1}) - \vartheta(e_k').$$
When we sum $p_{s,v}(e \splus)^2$ over all $v \in V$ and $e  \in \fan (v)$, the triplet $(\vartheta(e_k), \vartheta(e_{k+1}), \vartheta(e_k'))$ in the above identity appears three times but with different signs, once in each fan at the vertices of the triangle formed by $e_k, e_{k+1}$ and $e_k'$. 
Using the identity 
$$(a+b-c)^2 + (b+c-a)^2 + (c+a-b)^2 = a^2 +b^2 +c^2 + (a-b)^2 + (a-c)^2 + (b-c)^2,$$
we obtain
$$\sum_{v \in V} \sum_{e \in \fan (v)} p_{s,v}(e \splus)^2 = 
2 \sum_{e\in E} \vartheta(e)^2 + \sum_{e \sim e'} (\vartheta (e) - \vartheta(e'))^2 < \infty$$
as claimed. Here we note that the constant in front of the first sum is $2$ since every edge appears in two triangles, while the constant in front of the second is $1$ because every pair of adjacent edges appears in only one triangle. 
Finally, \eqref{eq:H_implies_P0} implies that $p_{s,v} (e_k \splus) \to 0$ as $k \to -\infty$. This shows $s \in \mc P_0$.
\end{proof}

Summarizing the above results, we obtain the following inclusions.  
\begin{cor}\label{cor:H_inclusions_summary}
We have
$\mc H \subset \mc P_0 \cap \mc S \subset  \mc P$.    
\end{cor}

Shear functions in $\mc H$ also satisfy another boundedness condition. 

\begin{lem}
    If $s\in \mc H$, then there exists a 
    constant $M=M(s)\geq 0$ such that for all $v\in V$ and all $n,m\in \m Z$, $n \le m$, 
    \begin{equation}\label{eq:shears_bound}
        \abs{\sum_{i=n}^m s(e_i)} \leq M, 
    \end{equation}
    where $\fan(v) = (e_i)_{i\in \m Z}$. 
\end{lem}
\begin{proof}
By Lemma \ref{lem:H_in_P0}, $s\in \mc H$ implies that $\{p_{s,v}(e\splus): v\in V, e\in \fan(v)\}$ is square summable, so there is a constant $C>0$ such that $|p_{s,v}(e\splus)|<C$ for all $v\in V, e\in \fan(v)$. Choose any $v\in V$, and let $\fan(v) = (e_i)_{i\in \m Z}$. For any $n,m$, 
\begin{align*}
    \abs{\sum_{i=n}^m s(e_i)} = \abs{ p_{s,v}(e_m\splus) - p_{s,v}(e_n\splus)} \leq  2C.
\end{align*}
Therefore \eqref{eq:shears_bound} holds with $M = 2C$. 
\end{proof}

\begin{remark} The class of shear functions satisfying Equation \eqref{eq:shears_bound} does not include, nor is contained in $\mc P_0$ or $\mc P$.  
\begin{itemize}[itemsep=-2pt]
    \item \eqref{eq:shears_bound} does not imply that $s\in \mc P$. For example, the map that has shears alternating $\pm 1$ along a fan satisfies \eqref{eq:shears_bound}, but is not in $\mc P$. 
    \item $s\in \mc P_0$ does not imply \eqref{eq:shears_bound} either. The condition $s\in \mc P_0$ implies that for each $v\in V$ there exists a constant $M_v$ such that the sums in $\fan(v)$ are bounded by $M_v$, but these $M_v$ constants may not be the same and the collection $\{M_v : v\in V\}$ may be unbounded for $s\in \mc P_0$.
\end{itemize}
\end{remark}

The condition \eqref{eq:shears_bound} helps us show that any $s\in \mc H$ induces a quasisymmetric homeomorphism $h$ of the circle. 
Shears for quasisymmetric homeomorphisms have been characterized. 
\begin{thm}[See \cite{Saric_circle},\cite{Saric_new}]\label{thm:qs_condition}
A shear function $s: E \to \m R$ is induced by a quasisymmetric map if and only if there exists $C\geq 1$ such that for all $v\in V$ with $\fan(v) = (e_i)_{i\in \m Z} $ and for all $k\in \m Z$, $n\in \m N$, 
\begin{align*}
    \frac{1}{C}\leq s(k,n;v) \leq C.
\end{align*}
Here $s(k,n;v)$ is 
\begin{align*}
    s(k,n;v) = \frac{e^{s(e_k)} + e^{s(e_k) + s(e_{k+1})} + \cdots + e^{s(e_k) + \cdots + s(e_{k+n})}}{1 + e^{-s(e_{k-1})} + \cdots + e^{-s(e_{k-1}) - \cdots - s(e_{k-n})}}. 
\end{align*}
\end{thm}

We obtain from this theorem the following corollary which is also considered in the paper of Parlier and the first author \cite{ParlierSaric}.
\begin{cor}\label{cor:H_in_QS}
    If $s: E\to \m R$ satisfies \eqref{eq:shears_bound}, then $s$ induces a quasisymmetric homeomorphism $h:\m T \to \m T$. In particular, if $s_h\in \mc H$, then $h\in \QS(\m T)$. 
\end{cor}

\begin{remark}
    Given the result above, in later sections we often abuse notation and write $h\in \mc H$ to mean that the homeomorphism $h$ has shear coordinates $s_h\in \mc H$. Despite the fact that not all shear functions in $\mc P,\mc P_0$ induce homeomorphisms, we also sometimes write $h\in \mc P$ or $h\in \mc P_0$ to mean that $h$ has shear function $s_h\in \mc P$ or $s_h \in \mc P_0$ respectively.
\end{remark}

\subsection{Analytic definition of diamond shear}
\label{sec:analytic_diamond}

In the section we show that the diamond shear coordinate of a circle homeomorphism can be described directly using derivatives of $h$. This description of diamond shears also leads to a relationship with coordinates called \textit{$\log \Lambda$-lengths} for decorated Teichm\"uller space studied in \cite{Penner1993UniversalCI,Penner2002OnHF,PennerBook}, see Section~\ref{sec:diamond_logL}. 
The following lemma is reminiscent of Lemma~\ref{lem:finite_balance} for finite support shears. 

\begin{lem}\label{lem:P_derivative}
\begin{enumerate}[label=\roman*)]
    \item  If $h \in \mc P$, then $h$ admits left and right derivatives at all rational points,
i.e., $\forall v \in V$, $h'(v\splus)$ and $h'(v\sminus)$ exist.   \label{it:P_derivative_exist}
\item \label{it:P0_differentiable} If $h \in \mc P_0$, then $h$ is differentiable at all rational points, i.e., $\forall v \in V$,  $h'(v\splus)=h'(v\sminus)$.
\item \label{it:C1_P0}  Conversely, if $h \in \mc C^1$, i.e., $h$ is continuously differentiable and $h' \neq 0$ everywhere,  then $h \in \mc P_0$.
\end{enumerate}

\end{lem}
\begin{proof}
Assume that $s = s_h \in \mc P$. We fix $v \in V = \m Q^2 \cap \mathbb T$ and $\fan (v) = (e_k)_{k \in \m Z}$.
As in Lemma~\ref{lem:finite_balance}, we define $\varphi : \m R \to \m R$ to be the homeomorphism $\varphi = \cayley_1 \circ h \circ \cayley_2^{-1}$, where $\cayley_1, \cayley_2$ are two M\"obius transformations $\m D \to \m H$ such that $\cayley_1 (h(v)) = \infty$, $\cayley_2 (v) = \infty$, $\cayley_2 (e_0) = (0,\infty)$, and $\cayley_2 (e_1) = (1,\infty)$. We have $\varphi$ fixes $\infty$, and $\cayley_2 (e_n) = (n,\infty)$ for all $n \in \m Z$. 
Since the limits as $n \to \infty$ of $\sum_{k = -n}^{-1} s(e_k)$ and $\sum_{k = 0}^n s(e_k)$ exist by definition of $\mc P$, 
\begin{equation}\label{eq:ell}
\ell : = \lim_{n \to \infty} \varphi(n+1) - \varphi(n) = (\varphi(0) - \varphi(-1))\exp\Big(\sum_{k = 0}^\infty s(e_k)\Big) \in (0,\infty)
\end{equation}
and 
\begin{equation}\label{eq:ell'}
    \ell' := \lim_{n \to \infty}\varphi(-n) - \varphi(-n-1) = (\varphi(0) - \varphi(-1))\exp\Big(-\sum_{k = -\infty}^{-1} s(e_k)\Big)\in (0,\infty) 
\end{equation}
also exist.  In particular, $\varphi(n) = n \ell + o(n)$ and $\varphi(-n) = - n \ell' + o(n)$ as $n \to \infty$ by Ces\`aro summation.

To show the left and right derivatives of $h$ at $v$ exist, it suffices to show that $\tilde \varphi = \iota \circ \varphi \circ \iota$ has left and right derivatives at $0$, where $\iota (x)= -1/x$. Note that $\tilde \varphi$ fixes $0$ and is an increasing function.
We have
$$\frac{\tilde \varphi (-n^{-1})}{-n^{-1}} = \frac{n}{\varphi(n)} \xrightarrow[]{n\to \infty} \ell^{-1}.$$
From the monotonicity of $\tilde \varphi$, we have
$$ \frac{\tilde \varphi(-(n+1)^{-1})}{-n^{-1}} \le \frac{\tilde \varphi (x)}{x} \le \frac{\tilde \varphi(-n^{-1})}{-(n+1)^{-1}},\qquad \forall x \in [-n^{-1}, -(n+1)^{-1}].$$
Hence, $\tilde \varphi$ admits the left derivative $\ell^{-1}$ at $0$. Similarly, we can show that $\tilde \varphi$ admits the right derivative $(\ell')^{-1}$. This concludes the proof of \ref{it:P_derivative_exist}.

Now we assume that $h \in \mc P_0$. The equations \eqref{eq:ell}, \eqref{eq:ell'} show that $\ell = \ell'$. Hence, the left and right derivatives of $h$ coincide, which shows \ref{it:P0_differentiable}.
 
For \ref{it:C1_P0}, if $h \in \mc C^1$, then $\varphi$ is continuously differentiable with $\varphi'(0) \neq 0$. We have
\begin{align*}
    \varphi (n+1) - \varphi (n) &= \frac{1}{\tilde \varphi (-n^{-1})} - \frac{1}{\tilde \varphi (-(n+1)^{-1})}  = \frac{ \tilde \varphi (-(n+1)^{-1})- \tilde \varphi (-n^{-1}) }{\tilde \varphi (-(n+1)^{-1})\tilde \varphi (-n^{-1})}\\
    & = \frac{n^{-1} - (n+1)^{-1}}{(n+1)^{-1}n^{-1}}\frac{\tilde \varphi'(x)}{\tilde \varphi'(y) \tilde \varphi'(z)}  =  \frac{\tilde \varphi'(x)}{\tilde \varphi'(y) \tilde \varphi'(z)} 
\end{align*}
for some $x \in [-n^{-1}, -(n+1)^{-1}]$, $y \in [-(n+1)^{-1},0]$, $z\in [-n^{-1},0]$. 
Therefore, 
\begin{equation}\label{eq:limit_varphi_l}
    \varphi (n+1) - \varphi (n) \xrightarrow[]{n \to \infty} \tilde \varphi'(0)^{-1},
\end{equation} which shows that $\sum_{k = 0}^n s(e_k)$ converges. Similarly, $\sum_{k = -n}^{-1} s(e_k)$ converges as well. We conclude with  \eqref{eq:ell}, \eqref{eq:ell'} that $h \in \mc P_0$.
\end{proof}

\begin{remark}
The converse statement \ref{it:C1_P0} is slightly weaker and we do not have an equivalent description for circle homeomorphisms whose shear coordinate satisfies the generalized balanced  condition $\mc P_0$. The naive converse of \ref{it:P0_differentiable} is not true. This lemma is trickier than Lemma~\ref{lem:finite_balance} as the set of vertices is dense in $\m T$ and we do not have the a priori smoothness of piecewise M\"obius maps. 
\end{remark}

\begin{prop}\label{prop:theta_analytic}
    If a circle homeomorphism $h \in \mc P_0$, then $\vartheta_h$
    is given by
    \begin{align} \label{eq:theta_derivative}
        \vartheta_h (e
        ) &= \frac{1}{2} \log h'(a) h'(b) - \log \frac{h(a)-h(b)}{a-b}
    \end{align}
    for all $e = (a,b) \in E$. 
\end{prop}
\begin{proof}
We assume first that $e = (-1,1)$ and recall that $h$ fixes $\pm 1, \ii$. The Cayley map $\mf c$ sends $1 \mapsto \infty$, $-1 \mapsto 0$, $\ii \mapsto -1$. We index edges $(e_n)_{n \in \m Z}$ in $\fan (1)$ such that $\mf c (e_n)$ is the geodesic $(n,\infty)$ in $\m H$. In this way, $e = e_0$. We write also $(e'_n)_{n\in \m Z} = \fan(-1)$, such that $e'_0 = e$.
Let $\varphi = \mf c \circ h \circ \mf c^{-1}$, and $\tilde \varphi = \iota \circ \varphi \circ \iota$, where $\iota (x) = -1/x$. We note that $\varphi$ fixes $-1, 0, \infty$ and $\tilde \varphi$ fixes $0, 1, \infty$.

Let $s = s_h$ and $\vartheta = \vartheta_h = \Psi (s_h)$.
Since $s \in \mc P_0$, we have from \eqref{eq:def_Psi} that 
$$\vartheta_h (e_0) = - \frac{1}{2} \Big(p_{s,1} (e \splus) + p_{s,-1} (e \splus) +s(e)\Big).$$

It follows from the proof of Lemma~\ref{lem:P_derivative}, \eqref{eq:ell}, and \eqref{eq:limit_varphi_l} that
\begin{equation}\label{eq:sum_shear_plus_limit}
   p_{s,1}(e\splus) + s(e) = \sum_{k = 0}^{\infty} s(e_k) = - \log \tilde \varphi'(0).
\end{equation}
Similarly, applying the same proof to $\fan (-1)$ with the homeomorphism $\tilde \varphi$, and $\varphi = \iota \circ \tilde \varphi \circ \iota$, we obtain 
$$\varphi'(0)^{-1} = \lim_{n\to \infty} \tilde \varphi(n+1) - \tilde \varphi(n) = (\tilde \varphi(1) - \tilde \varphi(0)) \exp \Big(\sum_{k = 1}^\infty s(e'_k)\Big) = \exp \Big(\sum_{k = 1}^\infty s(e'_k)\Big).$$
Hence
$$p_{s,-1}(e\splus) = \sum_{k = 1}^\infty s(e'_k) = -\log \varphi'(0).$$
On the other hand,
$$\varphi'(0) = \mf c'(-1) h'(-1) (\mf c^{-1})'(0) = h'(-1), \qquad  \tilde \varphi'(0) = (\iota \circ \mf c)'(1) h'(1) (\iota \circ \mf c)^{-1}{}'(0) = h'(1).$$
We obtain \eqref{eq:theta_derivative} since in this case $\frac{h(1) - h(-1)}{1 - (-1)} = 1$.

For a general edge $e = (a,b)$, $h$ might not fix $a, b$. We choose $\g, \delta \in \PSU(1,1)$, such that $\g$ maps $\mf F$ to $\mf F$, sending $-1 \mapsto a$ and $1 \mapsto b$ (see Section~\ref{sec:Farey}); and $\delta$ is such that the homeomorphism $\tilde h = \delta \circ h \circ \g$ fixes $\pm 1, \ii$. In particular, $\delta$ maps $h(a) \mapsto -1$ and $h(b) \mapsto 1$.
The homeomorphism $\tilde h$ has the shear coordinate $\tilde s = s \circ \g$ and therefore the diamond shear coordinate
$\tilde \vartheta = \vartheta \circ \g$.
Applying the previous result, we have
$$\vartheta (e) = \tilde \vartheta ((-1,1)) = \frac{1}{2}\log \tilde h'(-1) \tilde h'(1).$$
We use the fact that for any M\"obius transformation $A$, as long as it is well defined, $A'(a) A'(b) = \frac{(A(a) - A(b))^2}{(a-b)^2}$.
We obtain
$$\tilde h'(-1) \tilde h'(1) = [\delta'(h(a)) \delta'(h(b))] h'(a)  h'(b)[\g'(-1) \g'(1)]  =  h'(a)  h'(b) \frac{(a-b)^2}{(h(a) - h(b))^2}$$
which concludes the proof.
\end{proof}

\begin{remark}\label{rem:theta_expression_R}
We can see directly that the right-hand side of \eqref{eq:theta_derivative} is real-valued. In fact, for $e = (a,b)$, consider two M\"obius transformations $\cayley_1, \cayley_2$  sending $\m D$ onto $\m H$, such that 
\begin{itemize}[itemsep=-2pt]
    \item If $x = \cayley_1 (a) \in \m R$, $y = \cayley_1 (b) \in \m R$, and  $\cayley_2 (h(a)),  \cayley_2 (h(b)) \in \m R$, then $\varphi := \cayley_2 \circ h \circ \cayley_1^{-1}$ satisfies
\begin{equation*}
\vartheta_h (e) = \frac{1}{2} \log h'(a) h'(b) - \log \frac{h(a)-h(b)}{a-b} = \frac{1}{2} \log \varphi'(x) \varphi'(y) - \log \frac{\varphi(x)-\varphi(y)}{x-y} \in \m R;
\end{equation*}
\item If $\cayley_1 (a) = \cayley_2 (h(a)) = \infty$, and  $y = \cayley_1 (b) \in \m R$, then $\varphi := \cayley_2 \circ h \circ \cayley_1^{-1}$ satisfies
\begin{equation*}
\vartheta_h (e) = \frac{1}{2} \log {\varphi'(y)}{\varphi'(\infty)},
\end{equation*}
where $\varphi'(\infty):= \tilde \varphi'(0)$ for $\tilde \varphi = \iota \circ \varphi \circ \iota$.
\end{itemize}

\end{remark}

\subsection{Diamond shear in terms of $\log \Lambda$-length}\label{sec:diamond_logL}
In this section we show a simple relation  (Lemma~\ref{lem:theta_logL}) between diamond shear coordinates of a circle homeomorphism and the $\log \Lambda$-lengths, which are coordinates on the \textit{decorated Teichm\"uller space}, denoted $\widetilde{T(\m D)}$ introduced by Penner. See \cite{Penner1993UniversalCI,Penner2002OnHF,PennerBook}. This relation will play an essential role in Section~\ref{sec:symplectic}.

We view the Teichm\"uller space $T(\m D)$ as a space of \textit{tessellations} by identifying 
$$h\in T(\m D) \quad \iff \quad h(\Farey).$$
A point in $\widetilde{T(\m D)}$ is a tessellation $h(\Farey)$ plus a ``decoration'', namely a choice of horocycle at each vertex $h(v)\in h(V)$. The \emph{$\log \Lambda$-length} along an edge $(a,b)$ decorated with horocycles $\rho_a,\rho_b$ at $a,b\in \m T$ is
$$\log \Lambda(\rho_a,\rho_b) := \delta/2$$
where $\delta$ is the signed hyperbolic distance between $\rho_a\cap e$ and $\rho_b \cap e$ with the convention that if $\rho_a\cap \rho_b = \emptyset$, then $\delta$ is positive. Since hyperbolic distances are invariant under M\"obius transformations, so are $\log \Lambda$-lengths. Moreover, the $\log \Lambda$-length can be computed in terms of the Euclidean diameters of the 
horocycles in $\m H$.
\begin{lem}
[See {\cite[Chap.\,1,\,Sec.\,1.4,\,Cor.\,4.6]{PennerBook}}] \label{lem:logL_diameters}
Let $x= \cayley(a), y=\cayley(b)$, $c_x = \cayley (\rho_a)$, and $\rho_y = \cayley (\rho_b)$. If $x,y\neq \infty$,
\begin{align*}
    \log \Lambda(\rho_a,\rho_b) = \log |x-y| -\frac{1}{2}\log d_x d_y 
\end{align*}
where $d_x,d_y$ are the Euclidean diameters of $\rho_x,\rho_y$ respectively. If $x=\infty$, then 
\begin{align*}
    \log \Lambda(\rho_a,\rho_b) = \frac{1}{2}\log H - \frac{1}{2} \log d_y
\end{align*}
where $\rho_\infty = \{z \in \m H \colon \Im (z) = H\}$. \textnormal(Recall that horocycles at $\infty$ are horizontal lines, and we call $H$ the Euclidean height of $\rho_\infty$.\textnormal)
Further, for any M\"obius transformation $A = \begin{psmallmatrix} a & b \\ c & d \end{psmallmatrix} \in \PSL(2,\m R)$, if $A(x) \in \m R$ then $A(\rho_x)$ has diameter $d_x/(cx+d)^2 = A'(x)\, d_x$. If $A(x) = \infty$, then $A(\rho_x)$ has height $1/( c^2 d_x)$.
\end{lem}

The Farey tessellation admits a very special decoration by a collection of horocycles called the \textit{Ford circles}. In the Farey tessellation $\cayley (\Farey)$ of $\m H$, the Ford circle $\rho_{p/q}$ is the horocycle centered at $p/q$ with Euclidean diameter $1/q^2$. The Ford circle at infinity is the line $\{\Im (z)=1\}$. See Figure \ref{fig:fordcircles}. This collection of horocycles has the property that: 
\begin{itemize}[itemsep=-2pt]
    \item $\rho_{p/q}$ is tangent to $\rho_{r/s}$ if $(p/q,r/s)\in \cayley(E)$, 
    \item $\rho_{p/q}$ is disjoint from $\rho_{r/s}$ if $(p/q,r/s)\not\in \cayley(E)$.
\end{itemize}

\begin{figure}[ht]
    \centering
    \includegraphics[scale=0.7]{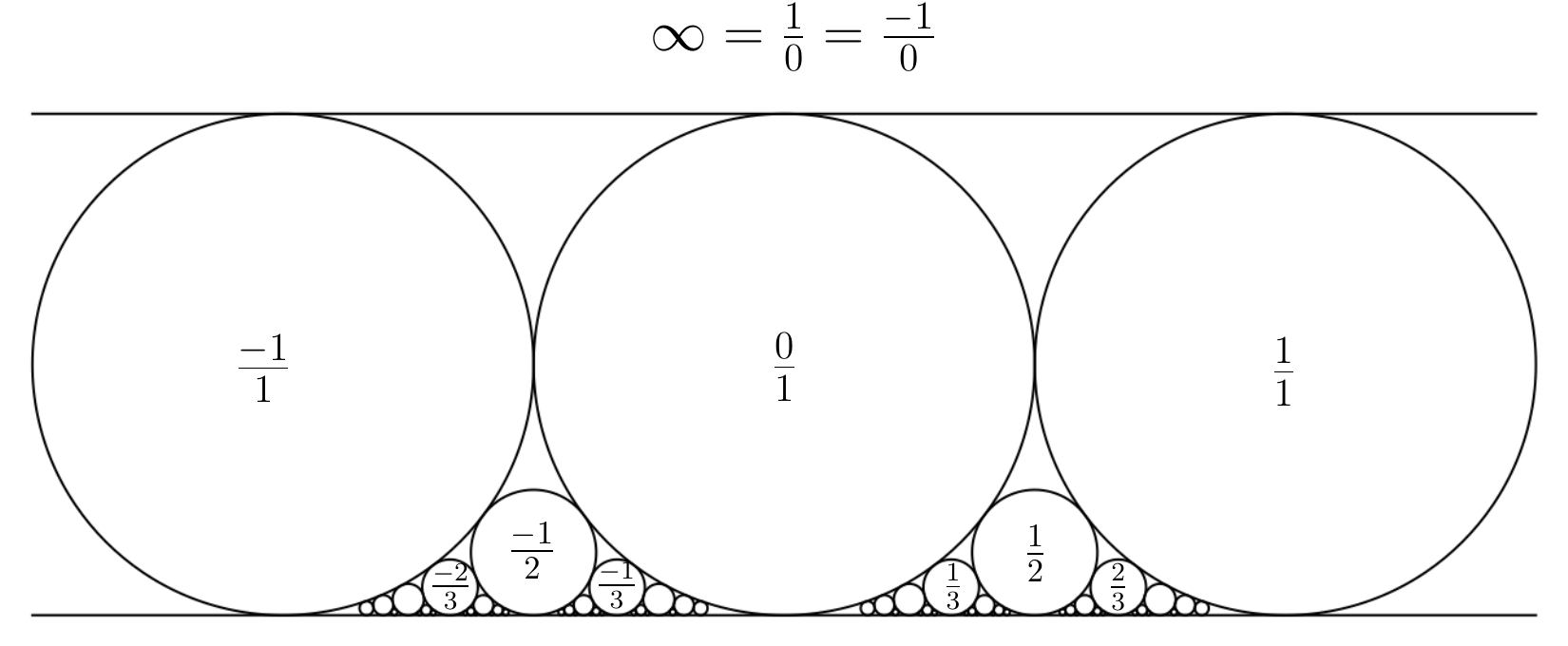}
    \caption{A few generations of Ford circles between $-1$ and $+1$ 
    labeled by their center points in $\cayley(\Farey)$.}
    \label{fig:fordcircles}
\end{figure}

The Ford circles in the disk are the pullback of the Ford circles in the upper half plane by $\cayley$. For all $e=(a,b)\in E$, the Farey tessellation decorated by the Ford circles has $\log \Lambda(\rho_{a},\rho_{b}) = 0$. Starting from the Ford circle decoration of $\Farey$ for the identity map, one can define a section $\sigma$ from homeomorphisms $h\in \mc P_0$ to $\widetilde{T(\m D)}$ motivated by Lemma~\ref{lem:logL_diameters}. 
\begin{df}[See {\cite[p.\,119-120]{PennerBook}}] \label{df:section}
If $h\in \mc P_0$, we define a section $\sigma$ to $\widetilde{T(\m D)}$ as follows:
\begin{itemize}[itemsep=-2pt]
    \item $\sigma(\Id_{\m T})$ assigns the Ford circle as the horocycle at each $v\in V$. In $\m H$, this means the horocycle at $p/q\in \m R$ has diameter $q^{-2}$ and the horocycle at $\infty$ has height $1$. 
    \item For any other $h\in \mc H$, let 
    $\varphi=\cayley \circ h\circ \cayley^{-1}:\m R\to \m R.$ 
  At each point $\varphi(p/q)\in \m R$, $\sigma(h)$ assigns the horocycle with diameter $|\varphi'(p/q)| q^{-2}$. At $\varphi(\infty)=\infty$, $\sigma$ assigns the horocycle of height $|\varphi'(\infty)|^{-1}$ at $\infty$. Here recall that $\varphi'(\infty) = \tilde\varphi'(0)$ for $\tilde \varphi(x) = -1/\varphi(-1/x)$. 
\end{itemize}
\end{df}
\begin{remark}
The section $\sigma$ was first introduced in \cite{MALIKOV1998282} for diffeomorphisms and is defined for any homeomorphism $h: \m T \to \m T$ fixing $\pm 1,\ii$ such that $h$ is differentiable at all rational points of the circle. By Lemma~\ref{lem:P_derivative}~\ref{it:P0_differentiable}, $\sigma$ is well-defined  
if $h\in \mc P_0$. It is also not hard to see that the decoration does not depend on the choice of conformal map $\m D \to \m H$, we choose $\cayley$ for simplicity. 
\end{remark}
When $e= (a,b)\in E$ and $h\in \mc P_0$, we define the notation
\begin{align*}
    \log \Lambda_{h}(e) := \log \Lambda(\rho_a,\rho_b)
\end{align*}
where $\rho_a,\rho_b$ are the horocycles at $h(a),h(b)$ chosen by the section $\sigma$. Using this, we can describe the relationship between diamond shear coordinates and $\log \Lambda$-lengths. 

\begin{lem}\label{lem:theta_logL}
If $h\in \mc P_0$, then for any $e=(a,b)\in E$, 
\begin{align*}
    \vartheta_h(e) = -\log \Lambda_{h}(e).
\end{align*}
\end{lem}
\begin{proof}
Let $\varphi:=\cayley \circ h \circ \cayley^{-1}$ which is a homeomorphism of $\m R$ fixing $\infty$. Choose $(a,b) \in E$ and let $\cayley(a) = x, \cayley(b) = y$. If $x, y \neq \infty$, let $d_x,d_y$ denote the diameters of the Ford circles at $x,y$. We have
\begin{align*}
    \log \Lambda_{\text{Id}}(e) = \log |x-y| - \frac{1}{2} \log {d_x d_y} = 0,
\end{align*}
which follows from direct computation or the fact that the Ford circles at $x$ and $y$ are tangent.
Using Definition~\ref{df:section} and Lemma \ref{lem:logL_diameters},
\begin{align*}
    \log \Lambda_h(e) &=\log \Lambda_h(e) - \log \Lambda_{\text{Id}}(e) \\&= \log |\varphi(x)-\varphi(y)| - \frac{1}{2}\log {|\varphi'(x)||\varphi'(y)|d_x d_y} - \log |x-y| + \frac{1}{2}\log {d_x d_y}\\
    &= -\frac{1}{2}\log |\varphi'(x)\varphi'(y)| + \log \frac{|\varphi(x)-\varphi(y)|}{|x-y|}.
\end{align*}
If $x = \infty = \varphi(x)$, then the horocycle at $\varphi(x)$ has height $H = \varphi'(\infty)^{-1}$,  $y \in \m Z$ and hence $d_y = 1$. We have
\begin{align*}
    \log \Lambda_h(e) &= \frac{1}{2}\log |\varphi'(\infty)|^{-1}- \frac{1}{2}\log {|\varphi'(y)| d_y} = - \frac{1}{2} \log {\varphi'(y)}{\varphi'(\infty)}.
\end{align*}
In both cases, we have $
    \vartheta_h(e) = - \log \Lambda_h(e)$
by Remark \ref{rem:theta_expression_R}. 
\end{proof}

\section{Relation to H\"older classes}\label{sec:Holder}

Here we relate the class $\mc H$ of homeomorphisms with square-summable diamond shears to the H\"older class $\mc C^{1,\alpha}$ defined in \eqref{eq:holder_def}. 
\begin{thm}\label{thm:C1alpha}
We have $\mc C^{1,\alpha}\subset \mc H$ if and only if $\alpha > 1/2$.
\end{thm}
For comparison, recall Lemma \ref{lem:C_alpha_in_WP}, which says analogously that $\mc C^{1,\alpha}\subset \WP(\m T)$ if and only if $\alpha >1/2$. The ``only if'' direction of Theorem \ref{thm:C1alpha} will follow from the fact that $\mc H \subset \WP(\m T)$ and the latter does not contain $\mc C^{1,1/2}$, see Theorem \ref{thm:H_Beltrami_bound}. In this section, we show that $\mc C^{1,\alpha} \subset \mc H$ for $\alpha > 1/2$.

For $(a,b) \in E$, let $\ell(a,b)$ denote the arclength of the arc in $\m T$ from $a$ to $b$ containing the child of $a,b$ (which is the shorter of the two arcs between $a$ and $b$, we call it a \emph{Farey segment}). We call these lengths the \textit{Farey lengths}. 

\begin{prop}\label{prop:farey-lengths}
The Farey lengths are $\ell^r$-summable if and only if $r>1$, e.g.
\begin{align*}
    \sum_{(a,b)\in E} \ell(a,b)^r <\infty
\end{align*}
if and only if $r>1$.
\end{prop}
\begin{proof}
Since $\sum_{(a,b) \in E_n\smallsetminus E_{n-1}} \ell(a,b) =2\pi$ for all $n$, the sum diverges for $r \leq 1$. 

Now we show the sum converges when $r>1$. We sort the sum over edges $(a,b)$ by the endpoint of the earlier generation. Note that by Lemma \ref{lem:same-gen} there are no edge between vertices of the same generation except $e_0$. Therefore,
\begin{align*}
\sum_{(a,b)\in E} \ell(a,b)^{r} = \pi^r + \sum_{a\in V} \sum_{b \in \child(a)} \ell(a,b)^r.
\end{align*}
The $\pi^r$ term corresponds to $e_0 = (-1,1)$. We refer to the set of vertices 
$$\child(a) = \{b \in V \colon (a,b)\in E, \, \gen(b) > \gen(a)\}$$ as the \textit{children of $a$}. 

Let $\Gamma_1,...,\Gamma_4$ denote the closed quarter circles with vertices $1,\ii,-1,-\ii$. All the Farey segments (other than the one corresponding to $e_0=(-1,1)$) are contained in one of these closed quarter circles, so it suffices to show that the lengths in $\Gamma:=\Gamma_3$, the arc from $-1$ to $-\ii$, are $\ell^r$-summable. 
The inverse Cayley transform $\cayley^{-1}$ sends $[0,1]$ onto $\Gamma$ and is Lipschitz on $[0,1]$. The image $\cayley (V\cap \Gamma)$ consists of the rational points between $0$ and $1$. If $\lambda$ is the Lipschitz constant of $\cayley^{-1}|_{[0,1]}$, then
\begin{align*}
\sum_{a\in V\cap \Gamma} \sum_{\child(a)} \ell(a,b)^r \leq \lambda^r \sum_{\frac{p}{q} \in \cayley (V\cap \Gamma)} \sum_{\frac{p'}{q'}\in \child(\frac{p}{q})} \bigg\lvert \frac{p}{q} - \frac{p'}{q'}\bigg\rvert^{r}. 
\end{align*}
If $k_1/m_1 < p/q < k_2/m_2$ are the parents of $p/q$, the children of $p/q$ are of the form
\begin{align*}
	\frac{p'}{q'} = \frac{k_i + n p}{m_i + n q} \qquad i = 1,2 \quad \text{and} \quad n\in \mathbb{N}_{\ge 1}
\end{align*}
by Lemma~\ref{lem:farey-children}.
Hence the distances we must bound are of the form
\begin{align*}
\bigg\lvert \frac{p}{q} - \frac{np+k_i}{nq+m_i}\bigg\rvert^{r} = \bigg\lvert \frac{pm_i - q k_i}{q(nq+m_i)} \bigg\rvert^{r}. 
\end{align*}
Lemma~\ref{lem:farey-children} also shows that $|pm_i-qk_i|^r = 1$, hence 
\begin{align*}
\abs{\frac{pm_i-qk_i}{q(nq+m_i)}}^r \leq \frac{1}{q^{2r} n^r}.
\end{align*}
Therefore 
\begin{align*}
 \sum_{\frac{p}{q} \in \cayley (V\cap \Gamma)} \sum_{\frac{p'}{q'}\in \child(\frac{p}{q})} \bigg\lvert \frac{p}{q} - \frac{p'}{q'}\bigg\rvert^{r}  &= \sum_{\frac{p}{q} \in \cayley (V\cap \Gamma)} \sum_{\substack{n\geq 1\\i=1,2}} \bigg\lvert \frac{p}{q} - \frac{np+k_i}{nq+m_i}\bigg\rvert^{r} 
\leq \sum_{\frac{p}{q} \in \cayley (V\cap \Gamma)}  \sum_{n\geq 1} \frac{2}{q^{2r}n^r}\\
&= \sum_{\frac{p}{q} \in \cayley (V\cap \Gamma)} 2 q^{-2r} \zeta(r).
\end{align*}
There are $\phi(q)$ rational points in $[0,1]$ with denominator $q$, where $\phi$ is Euler's $\phi$ function. Since $\phi(q) \leq q$, 
\begin{align*}
\sum_{\frac{p}{q} \in \cayley (V\cap \Gamma)} 2 q^{-2r} \zeta(r) = 2\zeta(r)\sum_{q\geq 1} \phi(q) q^{-2r} \leq 2\zeta(r) \sum_{q\geq 1} q^{1-2r} 
\end{align*}
which is finite exactly if $r>1$. 
\end{proof}

Theorem \ref{thm:C1alpha} follows straightforwardly from Proposition~\ref{prop:farey-lengths} and the analytic description of diamond shear coordinates. 

\begin{proof}[Proof of Theorem \ref{thm:C1alpha}] By Lemma \ref{lem:P_derivative} \ref{it:C1_P0}, $h\in \mc C^{1,\alpha}$ implies $s_h\in \mc P_0$. For any closed interval $I \subset \m T$ which is a proper subset of $\m T$, we can find 
M\"obius transformations $\cayley_1,\cayley_2:\m D \to \m H$ such that $\cayley_1(I),\cayley_2(h(I))$ are bounded intervals of $\m R$. 

Define $\varphi_I := \cayley_2\circ h \circ \cayley_1^{-1}:\m R \to \m R$. By Remark \ref{rem:theta_expression_R}, for all $(a,b)\in E$ such that $a,b\in I$, 
\begin{align*}
		\vartheta_h(e) = \frac{1}{2} \log \varphi_I'(x) \varphi_I'(y) - \log \frac{\varphi_I(x)-\varphi_I(y)}{x-y}
	\end{align*}
where $ x= \cayley_1(a), y = \cayley_1(b)$.

By the mean value theorem applied to $\varphi_I$, there exists $z\in (x,y)$ such that $\varphi_I'(z) = \frac{\varphi_I(x)-\varphi_I(y)}{x-y}$. Thus
\begin{align*}
	\vartheta_h(e) = \frac{1}{2}(\log \varphi_I'(x) - \log \varphi_I'(z)) + \frac{1}{2}(\log \varphi_I'(y) - \log \varphi_I'(z)). 
\end{align*}
Further, $c:=\cayley_1^{-1}(z)$ is contained in the interval $I(a,b)\subset \m T$. 
	
Since $\log h'$ is $\alpha$-H\"older and since $\cayley_2(h(I))$ and $\cayley_1(I)$ are bounded intervals, $\log \varphi_I'$ is $\alpha$-H\"older. Thus there is a constant $C>0$ such that for all $s, t\in \cayley_1(I)$, 
	\begin{align*}
		|\log \varphi_I'(s) - \log \varphi_I'(t)| \leq C |s-t|^\alpha.
	\end{align*}
Given also that $\cayley_1:I\to \cayley_1(I)$ is Lipschitz, it follows that 
	\begin{equation}\label{eq:farey-length-bound}
		|\vartheta(e)| \leq \frac{C}{2}\bigg(|x-z|^{\alpha} + |z-y|^{\alpha}\bigg) \leq 
		\frac{K}{2} \bigg(\ell(a,c)^{\alpha} + \ell(b,c)^{\alpha}\bigg)
		\leq K \ell(a,b)^{\alpha}
	\end{equation}
	for a constant $K = K(\cayley_1,I)$ and all $e=(a,b) \in E$ with $a,b\in I$. 
	
	We cover $\m T$ by the intervals $I_1 = \{e^{\ii \vartheta} : \vartheta \in [0,\pi]\}$ and $I_2 = \{e^{\ii \vartheta}: \vartheta \in [\pi,2\pi]\}$. 
	For every $(a,b)\in E$, there exists $j$ such that $a,b\in I_j$. Summing the bounds given in \eqref{eq:farey-length-bound} using an appropriate interval for each $e\in E$, 
	\begin{align*}
		\sum_{e\in E} \vartheta(e)^2 \leq K^2 \sum_{(a,b) = e \in E} \ell(a,b)^{2\alpha},
	\end{align*}
	where $K^2 = \max(K_1^2,K_2^2)$ and $K_j$ is the constant in \eqref{eq:farey-length-bound} for $I_j$. By Proposition~\ref{prop:farey-lengths}, this bound is finite if and only if $\alpha > 1/2$.	
\end{proof}

\section{Relation to Weil--Petersson homeomorphisms}\label{sec:relation_wp}

We now describe relationships between two classes of homeomorphisms defined in terms of shears, namely $\mc H$ and $\{h:s_h\in\mc S\}$, and the Weil--Petersson class $\WP(\m T)$. To summarize, the main results of this section are that $\mc H \subset \WP(\m T) \subset \{h: s_h\in\mc S\}$, and the inclusions are strict.

\subsection{Cell decomposition of $\m D$ or $\m H$ along $\dualtree$}

We say that an embedding of the dual tree of an ideal tessellation of $D = \m D$ or $\m H$ is \textit{centered} if the vertices of the tree are at the centers of the triangles in the tessellation and \textit{geodesic} if the edges of the tree are geodesics for the hyperbolic metric on $D$. Given a tessellation, there is a well-defined centered geodesic embedding of its dual tree. Further, since M\"obius transformations preserve angles and hyperbolic distances, if $\mc T = h (\Farey)$ is a tessellation embedded in 
$D$ with dual tree $\mc T^*$ embedded as a centered geodesic tree, then for any M\"obius transformation $A: D\to D'$, $A(\mc T^*)$ is a centered geodesic embedding of the dual tree of $A(\mc T)$ in $D'$. 
The complementary region $D \smallsetminus \mc T$ is a union of disjoint simply connected regions with piecewise-geodesic boundary, each of which contains exactly one vertex of $V(\mc T)$. We call these regions \textit{cells}.

We embed the dual tree $\dualtree$ as a centered geodesic tree in $\m D$. Applying the Cayley map $\cayley: \m D \to \m H$, the centered geodesic embedding of $\dualtree$ in $\m D$ is sent to a centered geodesic tree $\cayley(\dualtree)$ in $\cayley(\Farey)$ in $\m H$. 
We denote the cells of $\Farey$ or $\cayley(\Farey)$ simply by $C_v$ or $C_{\cayley(v)}$ for $v\in V$.  The group $\PSL(2,\m Z)$ acts transitively on the set of cells. Further, for any $v\in V$, there is $A\in \PSL(2,\m Z)$ such that $A\circ \cayley$ sends $C_v$ to $C_\infty$. See the pink area of the left-hand side of Figure~\ref{fig:infinity-cell} for an illustration of $C_\infty$.

\begin{figure}[ht]
\centering
    \includegraphics[width=.9\textwidth]{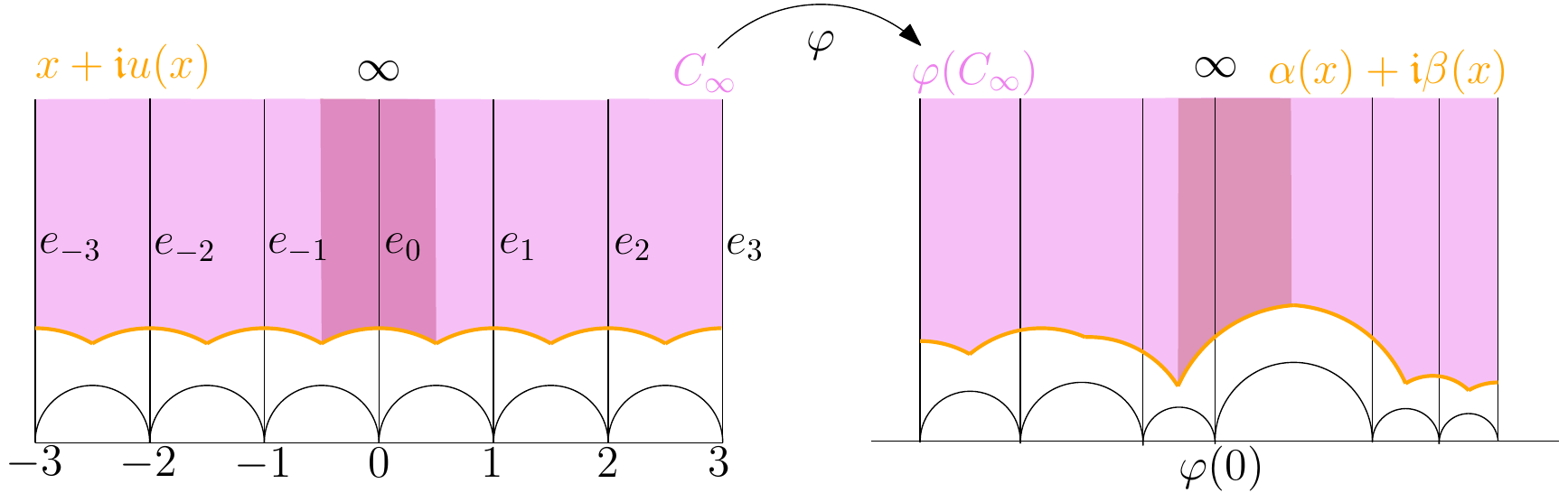}
    \caption{The cell at $\infty$ in the Farey tessellation between $-3$ and $3$ with boundary given by $x + \ii u(x)$, and an example of its image under a map $\varphi$ with boundary given by $\alpha(x) + \ii \beta(x)$. The strip $A_0$ and its image are illustrated in dark pink. }
    \label{fig:infinity-cell}
\end{figure}

We now describe the cell $C_\infty$ more explicitly. We denote $e_0 = (0,\infty)$, and let $\fan(\infty) = (e_n = (n,\infty))_{n\in \m Z}$. Let $a_n = (n,n+1)\in \cayley(E)$ and let $\tau_n$ 
be the ideal triangle bounded by $\{e_n,a_n, e_{n+1}\}$. The center of $\tau_n$ is 
\begin{align*}
    c_n 
    = 
    \frac{1}{2} + n + \ii \frac{\sqrt{3}}{2}.
\end{align*}
The geodesic arc connecting $c_n,c_{n+1}$ is an arc of the circle centered at $n+1$ of radius $1$. We define $u(x)$ to be the function whose graph is these arcs. Explicitly, on the interval $[n-1/2,n+1/2]$, $u(x) = \sqrt{1-(x-n)^2}$. In terms of $u(x)$, the cell at $\infty$ is 
\begin{align*}
    C_\infty = \{x+ \ii y : y \geq u(x)\}. 
\end{align*}
Note that $u(x)$ is continuous everywhere and differentiable everywhere except the half-integers. We further split $C_\infty$ into infinitely many \textit{strips} $A_n$, broken at the half-integers
\begin{align*}
    A_n = \bigg\{x + \ii y : y\geq u(x), \, n -\frac{1}{2} \leq x \leq n + \frac{1}{2}\bigg\}.
\end{align*}

Homeomorphisms $h\in \m T \to \m T$ act naturally on geodesic centered dual trees. Since $h$ determines the tessellation $h(\Farey)$, it determines the images of the centers of the triangles and thus determines the centered geodesic tree $h(\dualtree)$. We write $h (C_v)$ for the connected component of $\m D\smallsetminus h (\dualtree)$ containing $v$. 
We define the corresponding {\em hyperbolic stretching map} $F_h$, which is a homeomorphism from $\dualtree$ to $h(\dualtree)$ defined as follows:
\begin{itemize}[itemsep=-2pt]
    \item $F_h$ maps the center of a triangle $\tau$ of the Farey tessellation to the center of $h(\tau)$;
    \item $F_h$ linearly stretches the hyperbolic length along each geodesic edge.
\end{itemize}

Similarly, for $\varphi:\m R \to \m R$ fixing $\infty$, let $\varphi(C_\infty)$ denote the image of the cell at $\infty$ under $\varphi$ and define similarly the hyperbolic stretching map $F_\varphi : \cayley (\dualtree) \to \varphi (\cayley (\dualtree))$. We will give an explicit expression for $F_\varphi$ on $\partial C_\infty$. For this, we define functions $\alpha(x),\beta(x)$ such that 
\begin{align*}
    F_\varphi(x + \ii u(x)) = \alpha(x) + \ii \beta(x),
\end{align*}
so that the image cell is
\begin{align*}
    \varphi(C_\infty) = \{ \alpha(x) + \ii y : y \geq \beta(x)\}
\end{align*}
and the image strips are
\begin{align*}
    \varphi(A_n) = \bigg\{ \alpha(x) + \ii y : y \geq \beta(x) ,\, n - \frac{1}{2} \leq x \leq n + \frac{1}{2}\bigg\}.
\end{align*}
The maps $\alpha,\beta$ are continuous for all $x\in \m R$ and differentiable everywhere except the half-integers. 
Restricted to $[n -1/2, n+1/2]$, $\alpha(x) + \ii \beta(x)$ is a parametrization of the hyperbolic geodesic connecting $F_\varphi(c_n)$ to $F_\varphi(c_{n+1})$. 
In particular, $F_\varphi(c_n),F_\varphi(c_{n+1})$ are the centers of the triangles $\{\varphi(n-1),\varphi(n),\infty\}$, $\{\varphi(n),\varphi(n+1),\infty\}$ respectively, meaning that 
\begin{align*}
    F_{\varphi}(c_n) &= \varphi(n) - \frac{\rho}{2} + \ii \frac{\rho\sqrt{3} }{2}, \qquad \rho = \varphi(n)-\varphi(n-1), \\
    F_{\varphi}(c_{n+1}) &= \varphi(n) + \frac{\lambda}{2} + \ii \frac{\lambda\sqrt{3} }{2}, \qquad \lambda = \varphi(n+1) - \varphi(n).
\end{align*}
Using this, we now explicitly compute the functions $\alpha,\beta$ on a single strip. Translating, it suffices to compute the following:
\begin{lem}\label{lem:alpha_beta}
For $\rho, \l > 0$, let  $\gamma_{\rho,\lambda}$ denote the geodesic connecting $-\frac{\rho}{2} + \ii \frac{\rho\sqrt{3}}{2}$ to $\frac{\lambda}{2} + \ii \frac{\lambda \sqrt{3}}{2}$.  
We have $\gamma_{\rho,\lambda}$ is an arc of the circle centered at $\lambda - \rho$ with Euclidean radius $ r = \sqrt{\lambda^2 - \lambda \rho + \rho^2}$ and hyperbolic length $\ell = \log \frac{\rho + \lambda +r}{\rho + \lambda - r}$. 
    Let $F_{\rho,\lambda}:\gamma_{1,1}\to \gamma_{\rho,\lambda}$ be the hyperbolic stretching.
    If $\gamma_{1,1}$ is parametrized by $\gamma_{1,1}(x) = x + \ii u(x)$, then $\gamma_{\rho,\lambda}(x) = F_{\rho,\lambda}(\gamma_{1,1}(x)) = \alpha_{\rho,\lambda}(x) + \ii \beta_{\rho,\lambda}(x)$ where 
    \begin{align*}
        \alpha_{\rho,\lambda}(x) &= \frac{(r+\lambda -\rho) e^{\ell} \bigg(\frac{1+x}{1-x}\bigg)^{\ell/\log 3} - (r-\lambda + \rho) K(\rho,\lambda)^2}{e^{\ell} \bigg(\frac{1+x}{1-x}\bigg)^{\ell/\log 3} + K(\rho,\lambda)^2}, \\ 
        \beta_{\rho,\lambda}(x) &= \frac{2 r e^{\ell/2} \bigg(\frac{1+x}{1-x}\bigg)^{\ell/2\log 3} K(\rho,\lambda)}{e^{\ell} \bigg(\frac{1+x}{1-x}\bigg)^{\ell/\log 3} + K(\rho,\lambda)^2},
    \end{align*}
    and $K(\rho,\lambda) = \frac{2r + 2\lambda - \rho}{\sqrt{3}\rho}$. 
\end{lem}
\begin{proof}

Direct computations show that  $\gamma_{\rho,\lambda}$ is an arc of the circle centered at $\lambda-\rho$ of Euclidean radius $\sqrt{\lambda^2-\lambda\rho +\rho^2}$. On the imaginary axis, the hyperbolic stretching map $C_{a_1,a_2} : \ii \m  R_+\to \ii \m R_+$ that sends $(\ii, a_1 \ii)$ to $(\ii, a_2 \ii)$ is given by 
\begin{align*}
    C_{a_1,a_2}(z) = \ii\, (z/\ii)^{\log a_2/ \log a_1}.
\end{align*}
To compute $F_{\rho,\lambda}$, we conjugate by M\"obius transformations that send $\gamma_{1,1}$ and $\gamma_{\rho,\lambda}$ to segments of the form $(\ii,a_1 \ii)$ and $(\ii, a_2 \ii)$ respectively, where $\log a_1 = \text{length}(\gamma_{1,1})$ and $\log a_2 = \text{length}(\gamma_{\rho,\lambda})$.

Let $B_{\rho,\lambda}:\m H\to \m H$ be the M\"obius transformation such that
$B_{\rho,\lambda}(\gamma_{\rho,\lambda}) \subset \ii \m R_+$ (with negative endpoint of the geodesic containing $\gamma_{\rho,\lambda}$ sent to $0$ and positive endpoint sent to $\infty$) and $B_{\rho,\lambda}(-\frac{\rho}{2}+\ii \frac{\rho \sqrt{3}}{2}) = \ii$. To compute the length of $\gamma_{1,1}$, we note that
\begin{align*}
    B_{1,1}(z) = \sqrt{3}\, \frac{1+z}{1-z}, \qquad B_{1,1}(\frac{1}{2}+\ii \frac{\sqrt{3}}{2}) = 3 \ii.
\end{align*}
Thus $\text{length}(\gamma_{1,1}) = \log 3$. Given this and denoting $\ell = \text{length}(\gamma_{\rho,\lambda})$,
\begin{align*}
    F_{\rho,\lambda}(x + \ii u(x)) = B_{\rho,\lambda}^{-1}\circ C_{3, e^\ell}\circ B_{1,1}(x+\ii u(x)).
\end{align*}
To find $\ell$, we compute 
\begin{align*}
    B_{\rho,\lambda}(z) = K(\rho,\lambda) \frac{r -\lambda + \rho + z}{r + \lambda - \rho - z}
\end{align*}
where $K(\rho,\lambda)$ is the multiplicative constant so that $B_{\rho,\lambda}$ sends $-\frac{\rho}{2}+ \frac{\rho \sqrt{3}}{2} \ii$ to $\ii$. For $x + \ii y\in \gamma_{\rho,\lambda}$, we get that
\begin{align*}
    \frac{B_{\rho,\lambda}(x + \ii y)}{K(\rho,\lambda)}  = \frac{2 r y \ii}{(r+\lambda-\rho -x)^2 + y^2}.
\end{align*}
Thus 
\begin{align*}
    K(\rho,\lambda) = \frac{(r+\lambda-\rho +\rho/2)^2 + (\sqrt{3}\rho/2)^2}{2 r \sqrt{3}\rho /2} = \frac{2r +2\lambda - \rho}{\sqrt{3}\rho}.
\end{align*}
Similarly,
\begin{align*}
    B_{\rho,\lambda}(\frac{\lambda}{2} + \frac{\sqrt{3} \lambda}{2} \ii) = K(\rho,\lambda) \frac{\sqrt{3}\lambda}{2r+\lambda - 2\rho} = \frac{\lambda(2r + 2\lambda - \rho)}{\rho(2r + \lambda - 2\rho)}.
\end{align*}
From this and simplification, we find
\begin{align*}
    \ell = \log \frac{\lambda(2r + 2\lambda - \rho)}{\rho(2r + \lambda - 2\rho)} = \log \frac{\rho + \lambda +r}{\rho + \lambda - r}.
\end{align*}
Putting these together we compute $F_{\rho,\lambda}(x+\ii u(x))$. First
\begin{align*}
    B_{1,1}(x+ \ii u(x)) = \sqrt{3} \bigg(\frac{1+x}{1-x}\bigg)^{1/2} \ii
\end{align*}
so then 
\begin{align*}
    C_{3,e^\ell}\circ B_{1,1}(x+ \ii u(x)) = \ii  \sqrt{3}^{\ell/\log 3} \bigg(\frac{1+x}{1-x}\bigg)^{\ell/2 \log 3} =  \ii e^{\ell/2}  \bigg(\frac{1+x}{1-x}\bigg)^{\ell/2 \log 3}. 
\end{align*}
Note that 
\begin{align*}
    B_{\rho,\lambda}^{-1}(z) = \frac{(r+\lambda -\rho)z - K(\rho,\lambda)(r - \lambda +\rho)}{z + K(\rho,\lambda)}
\end{align*}
so finally
\begin{align*}
    F_{\rho,\lambda}(x+ \ii u(x)) &= \frac{(r+\lambda -\rho)\ii  e^{\ell/2}  \bigg(\frac{1+x}{1-x}\bigg)^{\ell/2 \log 3} - K(\rho,\lambda)(r - \lambda +\rho)}{\ii e^{\ell/2}  \bigg(\frac{1+x}{1-x}\bigg)^{\ell/2 \log 3} + K(\rho,\lambda)}.
\end{align*}
Taking real and imaginary parts completes the proof. 
\end{proof}

The following observation follows directly from the explicit formulas of $\a$ and $\b$.

\begin{cor}\label{cor:alpha_beta_props} 
The functions $(s, t, x) \mapsto \alpha_{e^s,e^t}(x)$, $\alpha'_{e^s,e^t}(x)$, $\b_{e^s,e^t}(x)$, and $\b'_{e^s,e^t}(x)$ are real analytic and bounded on $[-M, M] \times [-M, M] \times [-1/2,1/2]$
for all $M >0$. Moreover, $\alpha_{\rho,\lambda}'(x) >0$ for all $\rho,\lambda > 0$ and all $x\in [-1/2,1/2]$.
\end{cor}

\subsection{Proof of $\mc H\subset \WP(\m T)$}\label{sec:H_subset_WP}

In this section, we fix a circle homeomorphism $h\in \mc H$. We explicitly construct a homeomorphism $f: \m D \to \m D$ which extends $h$ using the cell structure described above. The extension coincides with the hyperbolic stretching map along the centered geodesic tree. We show that $f$ is quasiconformal (Theorem \ref{thm:extension_qc})
and then show that the Beltrami coefficient $\mu$ of $f$ is $L^2$-integrable on the disk with respect to the hyperbolic metric (Theorem \ref{thm:H_Beltrami_bound}) which shows that $h$ is Weil--Petersson.  
The construction is an adaption of a construction of Kahn--Markovic from \cite{KahnMarkovic} to the infinite tessellation setting.

\bigskip
\textbf{Construction of an extension $f:\m D\to \m D$ of $h:\m T \to \m T$.}

Recall that we embed the Farey dual tree $\dualtree$ as a centered geodesic tree. 
We first define $f|_{\dualtree}$ to be the hyperbolic stretching map from $\dualtree \to h(\dualtree)$, where $h(\dualtree)$ is the centered geodesic tree associated to the tessellation induced by $h$. 
Since $f$ sends $\dualtree$ to the centered geodesic tree in $h(\Farey)$, for any $v\in V$, $f(v) = h(v)$. Since $V$ is dense in $\m T$ and $h$ is continuous, if $f$ is also continuous then $f$ extends $h$. 

We now define $f$ on $\m D\smallsetminus \dualtree$ cell-by-cell. Choose $v\in V$ and let $\fan(v) = (e_k)_{k\in \m Z}$ centered on an arbitrary but fixed middle edge $e_0$. By Lemma~\ref{lem:H_in_P0}, $h\in \mc P_0$, so 
$$ M := \sum_{i=0}^\infty s_h(e_i) = -\sum_{i=-1}^{-\infty} s_h(e_i) < \infty.$$ 
If we choose M\"obius transformations $\cayley_1,\cayley_2:\m D\to \m H$ as in Lemma \ref{lem:P_derivative} so that $\cayley_1 (v) = \infty$, $\cayley_2 (h(v)) = \infty$, $\cayley_1(e_0) = \cayley_2 (e_0) = (0,\infty)$, and $\cayley_1(e_{-1}) = \cayley_2(e_{-1}) = (-1,\infty)$, then as in \eqref{eq:ell}, \eqref{eq:ell'}, for all $n\geq 0$, 
$\cayley_2\circ h\circ \cayley_1^{-1}$ satisfies 
\begin{align*}
    \lim_{n\to \infty} \cayley_2\circ h\circ \cayley_1^{-1}(n+1) - \cayley_2\circ h\circ \cayley_1^{-1}(n) =
    \lim_{n \to \infty} \cayley_2\circ h\circ \cayley_1^{-1}(-n) - \cayley_2\circ h\circ \cayley_1^{-1}(-n-1) = e^{M}.
\end{align*}
We construct an extension $\psi$ on $C_\infty$ for 
$$\varphi := e^{-M} \cayley_2\circ h\circ \cayley_1^{-1} : \m R \to \m R.$$ 
This way, $\varphi(0) = 0$ and $\varphi(\infty) = \infty$, and $\varphi$ is asymptotic to the identity near $\infty$. 
The restriction of $\psi$ to $\partial C_\infty\subset \cayley_1 (\dualtree)$ is given by the hyperbolic stretching, already studied in the previous section, and we denoted the map by
\begin{align*}
    \psi(x + \ii u(x) ) = \alpha(x) + \ii \beta(x).
\end{align*}

We extend $\psi$ to  $C_\infty$ by  
\begin{align*}
    \psi(x + \ii y) := \alpha(x) + \ii (\beta(x) - u(x) + y), \qquad \forall x\in \m R, \quad y \geq u(x).
\end{align*}
With this definition, $\psi$ sends the strip $A_n$ onto $\varphi(A_n)$ and is a homeomorphism $C_\infty\to \varphi(C_{\infty})$. 
Conjugating back to $\m D$, we obtain a continuous extension $\cayley_2^{-1} \circ (e^M \psi)\circ \cayley_1 $ of $f|_{\dualtree}$ to $C_v \cup \dualtree$. This construction applied to all $v\in V$ gives a continuous map $f: \m D \to \m D$ extending $h$. 

\begin{thm}\label{thm:extension_qc}
If $h\in\mc H$, then the extension $f:\m D \to \m D$ constructed above is quasiconformal. 
\end{thm}

\begin{proof}
Choose $v\in V$ and consider the map $\psi = e^{-M} \cayley_2 \circ f\circ \cayley_1^{-1}$. On $C_\infty = \cayley_1(C_v)$,  
\begin{align*}
    \psi(x + \ii y) := \alpha(x) + \ii (\beta(x) - u(x) + y), \qquad y \geq u(x).
\end{align*}
Since $\psi$ differs from $f$ by M\"obius transformations, $\psi$ being $K$-quasiconformal on $C_\infty$ is equivalent to $f$ being $K$-quasiconformal on $C_v$ (See e.g. \cite[Sec.\,1.2.8]{nag_1988}). Since $u,\alpha,\beta$ are differentiable almost everywhere, so is $\psi$. A direct computation shows that the Beltrami coefficient of $\psi$ on $C_\infty$ is given by
\begin{align*}
\mu(x+\ii y) = \frac{\psi_{\bar z}}{\psi_z}(x+\ii y) = \frac{\alpha'(x)-1 + \ii(\beta'(x) - u'(x))}{\alpha'(x)+1 + \ii(\beta'(x) - u'(x))}.
\end{align*}
Another direct computation shows that 
\begin{equation}\label{eq:mu}
    |\mu(x+\ii y)|^2 = 1 - \frac{4\alpha'(x)}{(\alpha'(x) + 1)^2 + (\beta'(x) - u'(x))^2}.
\end{equation}

It is clear that the ratio in \eqref{eq:mu} takes value in $(0,1]$, we now show it to be uniformly bounded away from $0$. For $x \in [n -1/2, n+1/2]$, the infimum of the ratio is a continuous function of $\rho = \varphi (n) - \varphi (n-1)$ and $\lambda = \varphi(n+1) - \varphi(n)$ by Lemma~\ref{lem:alpha_beta}. 
Moreover,
$$\varphi (n+1) - \varphi (n) = \exp(-M) \exp (\sum_{i = 0}^n s_h(e_i)) = \exp (\sum_{i = -\infty}^n s_h(e_i)) = \exp (- p_{s_h,v} (e_{n}\splus)).$$
Here we use the convention that if $n \le -1$ then
$\sum_{i = 0}^n s_h(e_i) = -\sum_{i = n}^{-1} s_h(e_i)$.
Since $p_{s_h,v} (e_{n}\splus)$ is uniformly bounded for all $v \in V$, $n \in \m Z$ by Lemma~\ref{lem:H_in_P0}, from the continuity, we obtain that there exists $k < 1$ independent of the cell chosen, such that 
$\norm{\mu}_{\infty, C_\infty} \le k.$
This shows that $f$ is $K$-quasiconformal in $\m D \smallsetminus \dualtree$, where $K = (1+k)/(1-k)$.  Points and $C^1$ Jordan curves are quasiconformally removable, see \cite[Thm.\,3.1.3]{FletcherMarkovic}, thus $\dualtree$ is quasiconformally removable. Since $f$ is a homeomorphism of $\m D$, we obtain that $f$ is a $K$-quasiconformal extension of $h$ to $\m D$. 
\end{proof}

To show that $\mathcal{H}\subset \WP(\m T)$, we need the following lemma. 

\begin{lem}\label{lem:mu_shear_bound}
Let $\psi_{s,t}(x + \ii y) = \alpha_{e^s,e^t}(x) + \ii (\beta_{e^s,e^t}(x) - u(x) + y)$ for $x\in [-1/2,1/2]$ and $y\geq u(x)$, and let $\mu_{s,t}$ be the Beltrami coefficient of $\psi_{s,t}$. For any $\vare>0$ small, there is a constant $C=C(\vare)>0$ such that if $|s|,|t|< \vare$, 
\begin{align*}
    |\mu_{s,t}(x+\ii y)| \leq C(|s| + |t|)
\end{align*}
for all $x\in [-1/2,1/2]$ and all $y\geq u(x)$.
\end{lem}
\begin{proof}
Recall that 
\begin{align}\label{eq:mu_alpha_beta}
    \mu_{s,t}(x+\ii y) = \frac{\alpha_{e^s,e^t}'(x) - 1 + \ii(\beta_{e^s,e^t}'(x) - u'(x))}{\alpha_{e^s,e^t}'(x) + 1 + \ii (\beta_{e^s,e^t}'(x) - u'(x))}.
\end{align}
Using the explicit formulas for $\alpha,\beta$ in Lemma \ref{lem:alpha_beta}, notice that $\mu_{0,0}(x+\ii y) \equiv 0$. Since the modulus of the denominator of $\mu$ is bounded below by $1$,
\begin{align*}
    |\mu_{s,t}(x+\ii y)| \leq |\alpha_{e^s,e^t}'(x) - 1| + |\beta_{e^s,e^t(x)}' - u'(x)|. 
\end{align*}
By Corollary \ref{cor:alpha_beta_props}, the right hand side of \eqref{eq:mu_alpha_beta} is analytic in $s,t,x$ on the appropriate interval. Further, when $(s,t) = (0,0)$, the right hand side is $0$. Fixing $x\in [-1/2,1/2]$ and expanding around $(s,t) = (0,0)$, we find that for $|s|,|t|<\vare$,
\begin{align*}
    |\mu_{s,t}(x + \ii y)| \leq C(\vare, x) (|s| + |t|).
\end{align*}
By Corollary \ref{cor:alpha_beta_props}, $C(\vare, x)$ can be chosen to be a continuous function of $x\in [-1/2,1/2]$ and hence achieves a maximum value $C=C(\vare)$, which completes the proof.
\end{proof}

\begin{thm}\label{thm:H_Beltrami_bound}
If $h\in \mc H$, then the extension $f$ constructed above has Beltrami coefficient $\mu$ such that
\begin{align*}
    \int_{\m D} \frac{|\mu(z)|^2}{(1-|z|^2)^2} \, \dd A(z) \lesssim \sum_{v\in V} \sum_{e\in \fan(v)} p_{s_h,v}^2(e\splus) < \infty.
\end{align*}
In particular, $\mc H\subset \WP(\m T)$.
\end{thm}

\begin{proof}
Choose $v\in V$ and again consider $\psi = e^{-M} \cayley_2\circ f \circ \cayley_1^{-1}$. For $x + \ii y \in C_\infty = \cayley_1(C_v)$, recall that \begin{align*}
    \mu(x+\ii y) = \frac{\alpha'(x) - 1 + \ii(\beta'(x) - u'(x))}{\alpha'(x) + 1 + \ii (\beta'(x) - u'(x))}.
\end{align*}

Further, for $x\in [n-1/2,n+1/2]$, $\alpha(x) = \alpha_{\lambda_{n-1},\lambda_n}(x-n)$ and $\beta(x) =\beta_{\lambda_{n-1},\lambda_n}(x-n)$, where 
$$\lambda_n = \varphi(n+1) - \varphi(n) = \exp(-p_{s_h,v}(e_n\splus)).$$
Fix a small threshold $\vare>0$. By Lemma~\ref{lem:H_in_P0}, there are only finitely many $(v,e)$ such that $|p_{s_h,v}(e\splus)|>\vare$. We say that a strip $A_n$ is bad if $|p_{s_h,v}(e\splus)| >\vare$ for $e = e_n$ or $e = e_{n-1}$. Let $N(\vare)<\infty$ be the number of bad strips across all cells. Since each strip $A_n$ is contained in an ideal triangle, $\text{Area}_{\text{hyp}}(A_n)\leq \pi$. Since $|\mu(x+\ii y)|<1$, 
\begin{align*}
    \int_{\text{bad strips}} \frac{|\mu(z)|^2}{(1-|z|^2)^2} \, \dd A(z) < \pi N(\vare). 
\end{align*}
On the other hand, if $|p_{s_h,v}(e_n\splus)|, |p_{s_h,v}(e_{n-1}\splus)| <\vare$, then by Lemma \ref{lem:mu_shear_bound}, 
\begin{align*}
    |\mu(x+\ii y)| \leq C(\vare)(|p_{s_h,v}(e_n\splus)| + |p_{s_h,v}(e_{n-1}\splus)|)
\end{align*}
for all $x+\ii y\in A_n$. Therefore
\begin{align*}
    \int_{C_v\smallsetminus \text{bad strips}} \frac{|\mu(z)|^2}{(1-|z|^2)^2} \, \dd A(z) \leq 4\pi C(\vare)^2 \sum_{e\in \fan(v)} p_{s_h,v}(e\splus)^2.
\end{align*}
Summing over $v\in V$, adding back the bad strips, and applying Lemma \ref{lem:H_in_P0}, we get that
\begin{align*}
    \int_{\m D} \frac{|\mu(z)|^2}{(1-|z|^2)^2} \, \dd A(z) \leq N(\vare) \pi + 4 \pi C(\vare)^2 \sum_{v\in V}\sum_{e\in \fan(v)} p_{s_h(e\splus)}^2 < \infty. 
\end{align*}
By Theorem \ref{thm:WP_Beltrami_L2}, $h\in \WP(\m T)$.
\end{proof}

\subsection{Counterexample: an element of $\WP(\m T)$ which is not in $\mc H$}\label{sec:counterexample}

We saw in the last section that $\mc H \subset \WP(\m T)$. It is straightforward to see that the reverse inclusion does not hold. 
\begin{prop}
    $\WP(\m T) \not\subset \mc P$. As a consequence, $\WP(\m T) \not\subset \mc H$.
\end{prop}
\begin{proof}
By Lemma \ref{lem:P_derivative}\,\ref{it:P_derivative_exist}, if $h\in \mc P$ then $h$ has left and right derivatives at all rational points. A Weil--Petersson homeomorphism may not have left or right derivatives since functions in the Sobolev space $H^{1/2}$ may not have left or right limits everywhere, so $\WP(\m T) \not\subset \mc P$ from Definition~\ref{def:WP}. By Corollary \ref{cor:H_inclusions_summary}, $\WP(\m T) \not\subset \mc H$.
\end{proof}

We illustrate an explicit example of a homeomorphism $\varphi:\m R \to \m R$ which is Weil--Petersson but not in $\mc P$. 

\begin{ex}
Consider $\varphi:\m R \to \m R$ such that $\varphi(x) = x \log |x| - x$ for $|x|> 2$, and is smooth on $\m R$. 
On one hand, $\varphi(x)$ grows faster than linear functions as $x \to \infty$ or $x \to -\infty$, and hence $\varphi\not\in \mc P$.

To show that $\varphi$ is Weil--Petersson, we show\footnote{This characterization of the Weil--Petersson class is obtained in \cite[Thm.\,2.2]{Shen-Tang}.} that $u(x) := \log \varphi'(x)$ is in $H^{1/2}(\m R)$  by showing that it is the trace of a map $f:\m H\to \m H$ which has finite Dirichlet energy (by the classical Douglas formula, see, e.g. \cite[Eq.(2.2), (2.3)]{VW1}).
We compute 
\begin{align*}
    u(x) = \log \log |x|
\end{align*}
for $x$ outside $(-2,2)$. Hence $u$ is the trace of a smooth function $f$ on $\overline {\m H}$ which takes the values $f(z) = \log \log |z|$ on $\overline{\m H}\smallsetminus \m D (0,2)$. The gradient of $f$ satisfies $|\nabla f(z)| = \frac{1}{r \log r}$ if $|z| = r > 2$. Thus,
\begin{align*}
    \int_{\m H \smallsetminus \m D(0,2)} |\nabla f(z)|^2 \, \dd A(z) 
    &= \int_{0}^{\pi} \int_{2}^{\infty} \frac{r}{r^2 (\log r)^2} \, \dd r \dd \theta = \pi \bigg[-\frac{1}{\log r}\bigg]_{r = 2}^{\infty} = \frac{\pi}{\log 2} < \infty,
\end{align*}
and $\varphi$ is Weil--Petersson. 

We can also see explicitly that $\varphi\not \in \mc P$ by computing its shears $s_\varphi(e_n)$ for $e_n = (n,\infty)\in \fan(\infty)$. For $n\geq 0$, 
\begin{align*}
    \varphi(n+1) - \varphi(n) 
    &= (n+1)\log (n+1) - n\log n - 1 \\
    &= (n+1)[\log n + \log ( 1 + 1/n)] - \log n - 1 \\
    &= \log n + (n+1) \bigg(\frac{1}{n} - \frac{1}{2n^2} + O(\frac{1}{n^3})\bigg) - 1\\
    &= \log n + \frac{1}{2n} + O(\frac{1}{n^2}).
\end{align*}
Analogously, 
\begin{align*}
    \varphi(n) - \varphi(n-1) 
    &=\log n - \frac{1}{2n} + O(\frac{1}{n^2}).
\end{align*}
Therefore 
\begin{align*}
    s_\varphi(e_n) = \log \frac{\varphi(n+1) - \varphi(n)}{\varphi(n) - \varphi(n-1)} = \log\bigg( 1 + \frac{1}{n \log n} + O(\frac{1}{n^2})\bigg) = \frac{1}{n\log n} + O(\frac{1}{n^2}).
\end{align*}
In particular, $s_\varphi(e_n)$ is not summable. Note however that $s_\varphi(e_n)$ is square-summable (as it must be by Theorem \ref{thm:WP_in_S}).

\end{ex}

\subsection{Convergence in $\mc H$ implies convergence in $\WP(\m T)$}
In this section we show that convergence in $\mc H$ is stronger than convergence in the Weil--Petersson metric (Corollary~\ref{cor:H_convergence_cor}).
\begin{thm}\label{thm:H_convergence}
Suppose that $g,h\in \mc H$, and let $q = g\circ h^{-1}$. Then there exists $C (\vare, h) > 0$ such that $q$ has a quasiconformal extension $f_q$ with Beltrami coefficient $\mu$ satisfying
\begin{align*}
    \int_{\m D} \frac{|\mu(z)|^2}{(1-|z|^2)^2} \, \dd A(z) \le C(\vare, h) \sum_{e\in E} (\vartheta_g(e) -\vartheta_h(e))^2,
\end{align*}
for any $g \in \mc H$ such that $\sum_{e\in E} (\vartheta_g(e) -\vartheta_h(e))^2 \le \vare$.
\end{thm}

\begin{proof}
Let $f_g, f_h$ be the quasiconformal 
extensions of $g,h$ respectively constructed in the section above. We will show that the extension $f_q:= f_g\circ f_h^{-1}$ of $q$ has Beltrami coefficient satisfying the bound above cell by cell.

Choose $v\in V$, and let $\fan(v) = (e_n)_{n\in \m Z}$. For $g$, choose Cayley maps $\cayley_1,\cayley_2$ as in Section \ref{sec:H_subset_WP} so that $\cayley_1(v) = \infty$, $\cayley_2(g(v)) = \infty$, and $\cayley_1(e_0) = \cayley_2(e_0) = (0,\infty)$, $\cayley_1(e_{-1}) = \cayley_2(e_{-1}) = (-1,\infty)$. Define $\psi_g = e^{-M_g} \cayley_2 \circ f_g\circ \cayley_1^{-1}$, where $M_g = \sum_{n=0}^\infty s_g(e_n)$, and let $\alpha_g, \beta_g$ be such that $\psi_g(x+\ii u(x)) = \alpha_g(x) + \ii \beta_g(x)$. Analogously choose Cayley maps $\cayley_3,\cayley_4$ to define $\psi_h,\alpha_h, \beta_h$.
Since $\psi_g,\psi_h$ both fix $\infty$, $\psi_q:=\psi_g\circ \psi_h^{-1}$ fixes $\infty$. In particular, $\psi_q$ maps $C_\infty(h(\Farey))$ onto $C_\infty(g(\Farey))$.

The boundary curve of $C_\infty(h(\Farey))$ is given by
\begin{align*}
    x + \ii u_h(x), \qquad u_h(x) = \beta_h\circ \alpha_h^{-1}(x).
\end{align*}
We define $\alpha_q,\beta_q$ so that 
\begin{align*}
    \psi_q(x + \ii u_h(x)) = \alpha_q(x) + \ii \beta_q(x).
\end{align*}
Since $\psi_q = \psi_g\circ \psi_h^{-1}$, 
\begin{align*}
    \alpha_q(x) = \alpha_g\circ \alpha_h^{-1}(x)\\
    \beta_q(x) = \beta_g\circ \alpha_h^{-1}(x).
\end{align*}
Hence the Beltrami coefficient $\mu$ of $\psi_q$ on $C_{\infty}(h(\Farey))$ is 
\begin{align*}
    \mu(x + \ii y) = \frac{\alpha_g'-\alpha_h'+\ii (\beta_g'-\beta_h')}{\alpha_g'+\alpha_h'+\ii(\beta_g'-\beta_h')} \circ \alpha_h^{-1}(x).
\end{align*}
For $w = \alpha_h^{-1}(x) \in [n-1/2,n+1/2]$, $\alpha_h'(w)$ is a continuous function of $p_{s_h,v}(e_n\splus)$ and $p_{s_h,v}(e_{n-1}\splus)$. By Lemma \ref{lem:H_in_P0}, $p_{s_h,v}(e_n\splus)$ is uniformly bounded for all $v\in V,n\in \m Z$. Further, by Corollary \ref{cor:alpha_beta_props}, $\alpha_h'(w) > 0$. Combining this with the fact that $\alpha_h'$ is continuous, there exists a constant $K(h)>0$ independent of cell such that $\alpha_h'(w) \geq  K(h)$. Therefore 
\begin{align*}
    |\mu(x+ \ii y)| \leq \frac{|\alpha_g' - \alpha_h'| + |\beta_g' - \beta_h'|}{K(h)}\circ \alpha_h^{-1}(x).
\end{align*}
Fix a threshold $\vare>0$. 
From the assumption and Lemma \ref{lem:H_in_P0}, there are only finitely many $(v,e)$ such that $|p_{s_h,v}(e\splus) - p_{s_g,v}(e\splus)|>\vare$. Let $N(\vare)$ be the number of strips $A_n$ across all cells where $|p_{s_g,v}(e_n\splus)-p_{s_h,v}(e_n\splus)|$ or $|p_{s_g,v}(e_{n-1}\splus)-p_{s_h,v}(e_{n-1}\splus)|$ is larger than $\vare$. Any other strip $A_n$ we call a good strip. 

By Corollary \ref{cor:alpha_beta_props}, $\alpha',\beta'$ are analytic functions of the $(s,t,w)$. Expanding around any $(s_0,t_0)$, there is a constant $C$ depending on $\vare$ and $(s_0,t_0)$ such that if $|s-s_0|,|t-t_0|<\vare$, then
\begin{align*}
    |\alpha_{e^s,e^t}'(w) -\alpha_{e^{s_0},e^{t_0}}'(w)| \leq  C( |s-s_0| + |t-t_0|) \\
    |\beta_{e^s,e^t}'(w) -\beta_{e^{s_0},e^{t_0}}'(w)|\leq  C( |s-s_0| + |t-t_0|)
\end{align*}
for all $w\in [-1/2,1/2]$. For each good strip $A_n$, we use this expansion for $(s_0,t_0) = (p_{s_h,v}(e_{n-1}\splus),p_{s_h,v}(e_{n}\splus))$. By Lemma \ref{lem:H_in_P0} applied to $h$ these are uniformly bounded for all $v \in V, n \in \m Z$, so we can take the constant $C$ to depend only on $\vare$ and $h$ 
to find
\begin{align*}
    |\alpha_g'(w)-\alpha_h'(w)| \leq C (|p_{s_g,v}(e_{n-1}\splus)-p_{s_h,v}(e_{n-1}\splus)| +  |p_{s_g,v}(e_{n}\splus)-p_{s_h,v}(e_{n}\splus)|)\\
    |\beta_g'(w)-\beta_h'(w)| \leq C ( |p_{s_g,v}(e_{n-1}\splus)-p_{s_h,v}(e_{n-1}\splus)| +  |p_{s_g,v}(e_{n}\splus)-p_{s_h,v}(e_{n}\splus)|).
\end{align*}
Therefore there is another constant $C(\vare,h)$ such that for all $x+\ii y\in A_n$,
\begin{align*}
    |\mu(x+\ii y)| \leq C(\vare, h) \bigg( \abs{p_{s_g,v}(e_{n-1}\splus) - p_{s_h,v}(e_{n-1}\splus)} + \abs{p_{s_g,v}(e_{n}\splus) - p_{s_h,v}(e_{n}\splus)}\bigg).
\end{align*}
Every strip $h(A_n)$ is a geodesic triangle, so it is contained in an ideal triangle and its hyperbolic area is bounded by $\pi$. Summing over $v\in V, n\in \m Z$, adding back the bad strips, and integrating, we find
\begin{align*}
    \int_{\m D} \frac{|\mu(z)|^2}{(1-|z|^2)^2} \, \dd A(z) \leq \pi N(\vare)  + 4\pi C(\vare, h)  \sum_{v\in V} \sum_{e\in \fan(v)} (p_{s_g,v}(e\splus) - p_{s_h,v}(e\splus))^2.
\end{align*}
Note that 
$$N (\vare) \le \frac{1}{\vare^2} \sum_{v\in V} \sum_{e\in \fan(v)} (p_{s_g,v}(e\splus) - p_{s_h,v}(e\splus))^2.$$
Applying Lemma~\ref{lem:H_in_P0} with $s = s_g - s_h$, we obtain using Cauchy-Schwarz
\begin{align*}
     &\sum_{v\in V} \sum_{e\in \fan(v)} (p_{s_g,v}(e\splus) - p_{s_h,v}(e\splus))^2 \\
     & \le  2 \sum_{e\in E} (\vartheta_g(e) - \vartheta_h(e))^2 +  \sum_{e\sim e'} 2 (\vartheta_g(e) - \vartheta_h(e))^2 + 2(\vartheta_g(e') - \vartheta_h(e'))^2 \\
    & = 6 \sum_{e\in E} (\vartheta_g(e) - \vartheta_h(e))^2,
\end{align*}
which completes the proof.
\end{proof}

\begin{lem}\label{lem:beltrami_convergence_WP}
Suppose that $h,(h_n)_{n\geq 1}$ are Weil--Petersson homeomorphisms fixing $\pm 1, \ii$, and let $\mu_n$ be the Beltrami coefficient of a quasiconformal extension of $h_n\circ h^{-1}$ in $\m D$. If
\begin{align*}
    \lim_{n\to \infty} \int_{\m D} \frac{|\mu_n(z)|^2}{(1-|z|^2)^2}\, \dd A(z)  = 0,
\end{align*}
then $h_n$ converges to $h$ in the Weil--Petersson metric. 
\end{lem}

This result must be well-known. 
For readers' convenience we sketch the proof using several lemmas from \cite{TT06}. The results in \cite{TT06} are stated using $\mu$ defined in the outer disk $\m D^*= \{z\in \m C \colon |z|>1\}$, by pre-composing quasiconformal maps by $z \mapsto 1/\ad z$ we can easily translate those results to $\m D$.
\begin{proof}
Theorem I.3.8 in \cite{TT06} shows that $\WP(\m T)$ is a topological group. Therefore, it suffices to show the claim when $h = \Id_{\m T}$.

We write 
$$\norm{\mu}_2^2 :=  \int_{\m D} \frac{|\mu(z)|^2}{(1-|z|^2)^2}\, \dd A(z) \text{ and } \norm{\mu}_\infty = \sup_{z\in \m D} |\mu (z)|. $$
We note that two measurable Beltrami differentials $\mu, \nu$ with $\norm{\mu}_\infty <1$ and $\norm{\nu}_\infty <1$ are said to be \emph{equivalent}, if they are the Beltrami coefficients of a quasiconformal extension of the same circle homeomorphism fixing $\pm 1, \ii$. 

If  $\norm{\mu}_\infty  <1$, the Bers embedding of the equivalence class of $\mu$, denoted as $B ([\mu])$, is a holomorphic function $\phi \in \mc Q (\m D^*)$.  See Section~\ref{sec:harmonic} for more details.
If furthermore $\norm\mu_2< \infty$, \cite{TT06} shows that
$$\norm{\phi}_{A_2(\m D^*)}^2 : = \int_{\m D^*} |\phi|^2 (1-|z|^2)^2 \, \dd A(z) <\infty.$$

Now let $\mu_n$ be a family of Beltrami differentials such that $\lim_{n \to \infty} \norm {\mu_n}_2 =  0$,
\cite[Lem.\,I.2.9]{TT06} implies that
there exists $C > 0$, such that 
$$\norm{B([\mu_n])}_{A_2(\m D^*)} \le C \norm{\mu_n}_2 \to 0.$$
Since $B|_{T_0(\m D)}$ is a biholomorphic mapping of Hilbert manifolds \cite[Thm.\,I.2.13]{TT06}, where 
$$T_0(\m D) = \{[\mu] \colon \norm{\mu}_2 < \infty, \norm{\mu}_\infty <1\} \simeq \WP(\m T)$$
and the identification $\WP(\m T) \to T_0(\m D)$ is the map from a circle homeomorphism $h$ to the equivalence class of Beltrami coefficients of any quasiconformal extension $[\mu]$,  
it shows $[\mu_n]$ converges to $[0]$ in $T_0(\m D)$ which is by definition equivalent to $h_n$ converges to $\Id_{\m T}$ for the Weil--Petersson metric.
\end{proof}

Lemma \ref{lem:beltrami_convergence_WP} and Theorem \ref{thm:H_convergence} applied to $g = h_n$ combine to give the following corollary.

\begin{cor}\label{cor:H_convergence_cor}
Suppose that $h,(h_n)_{n\geq 1}\in \mc H$ with diamond shear coordinates $\vartheta,\vartheta_n$ respectively. If 
$    \lim_{n\to \infty} \sum_{e\in E} (\vartheta_n(e) - \vartheta(e))^2 = 0, $
then $h_n$ converges to $h$ in the Weil--Petersson metric. 
\end{cor}

\subsection{Square summable shears}

In this section we prove three results about square summable shear functions $\mc S$. First we show that $\mc S$ is not contained in $\text{Homeo}(\m T)$, nor does it contain $\QS(\m T)$.
\begin{prop}\label{prop:S_not_homeo_qs}
    There exists $s\in \mc S$ such that $s$ does not induce a homeomorphism. Conversely, there exists $h\in \QS(\m T)$ such that its shear $s_h\not\in \mc S$. 
\end{prop}
\begin{proof}
    To prove this result, we just need to exhibit two examples and apply the suitable conditions from \cite{Saric_circle,Saric_new}.

For $\mc S \not\subset \text{Homeo}(\m T)$: let $(e_n)_{n\in \m Z}$ denote the fan of edges incident to $1$ in counterclockwise order. By \cite[Theorem C]{Saric_circle}, if $s:E\to \m R$ has 
\begin{align*}
    \sum_{n=1}^{\infty} \exp(s(e_1) + \dots + s(e_n)) < \infty,
\end{align*}
then $s$ does \textit{not} induce a homeomorphism. Given this, we choose $1/2<\alpha<1$, and define the shear function $s:E\to \m R$ so that $s(e_n) = -\frac{1}{n^{\alpha}}$ for $n\geq 1$, and is $0$ on all other edges in $E$. On one hand, since $\alpha > 1/2$, 
$$\sum_{e\in E} s(e)^2 = \sum_{n\geq 1} \frac{1}{n^{2\alpha}} < \infty,$$
so $s\in \mc S$. On the other hand, $s(e_1)+\dots s(e_n) = -H(n,\alpha)$ is minus the generalized harmonic number with parameter $\alpha$, and since $\alpha < 1$,
\begin{align*}
    \sum_{n=1}^{\infty} \exp(s(e_1) + \dots + s(e_n)) = \sum_{n=1}^{\infty} \exp(-H(n,\alpha)) <\infty.
\end{align*}
Therefore $s$ does not induce a homeomorphism.

For $\QS(\m T) \not\subset \mc S$: if $s:E\to \m R$ has $s(e)=1$ for infinitely many edges $e\in E$, then $s\not \in \mc S$. In particular, consider $s:E\to \m R$ where for each $n$, there is one edge $e$ connecting vertices of generations $2n$ and $2n+1$ of $\Farey$ where $s(e) = 1$, and all other shears are $0$. Clearly this includes infinitely many edges $e$ where $s(e) = 1$. Note also that this has the property that every fan contains either zero or one edge with nonzero shear. One can check the condition for a shear to induce a quasisymmetric homeomorphism (from \cite{Saric_new,Saric_circle}, and included here as Theorem \ref{thm:qs_condition}) is satisfied with $C = e$. 
\end{proof}

Finally we show the following inclusion:
\begin{thm}\label{thm:WP_in_S}
    If $h\in \WP(\m T)$, then $s_h\in \mc S$. 
\end{thm}
Note that the reverse statement is not true: a circle homeomorphism $h$ with shear coordinate $s_h$ supported on a single edge is has $s_h\in\mc S$ but does not satisfied the finite balanced condition from Definition \ref{df:finite_balance}, so Lemma \ref{lem:finite_balance} shows that such a homeomorphism is not Weil--Petersson. To show the inclusion we use the following property of Weil--Petersson homeomorphisms due to Wu.

\begin{thm}[See \cite{Wu2011}]\label{thm:Wu_necessary}  Suppose $h\in \WP(\m T)$. Then there is a 
constant $C=C(h)>0$ such that for any pairwise disjoint collection of quads $Q_1,\dots, Q_n$ with vertices on $\m T$, 
\begin{align*}
    \sum_{i=1}^n d^2(\Cr(Q_i), \Cr(h(Q_i))) \lambda(\Cr(Q_i)) < C
\end{align*}
where $d(\cdot, \cdot)$ is the hyperbolic metric on $\m C\smallsetminus\{-1,0\}$ and $\lambda(x) = \exp(d(1,x)-|\log x|/2)$.
\end{thm}
We clarify that quads are considered as open sets bounded by hyperbolic geodesics. In other words, quads sharing only boundary edges or vertex are also considered as disjoint. 
Note that $\l (1)  = 1$ and since the hyperbolic metric has smooth conformal factor with respect to the Euclidean metric, there exists $a > 0$ such that 
\begin{equation}\label{eq:cr_dist_shear}
d(1,e^s) = a |s| + O(s^2), \quad s \to 0.
\end{equation}

\begin{proof}[Proof of Theorem \ref{thm:WP_in_S}]
Let  $h\in \WP(\m T)$.
If $Q$ is a Farey quad, then $\Cr(Q) = 1$. Therefore for an infinite sequence  $Q_1, Q_2, \ldots$ of pairwise disjoint Farey quads,  Theorem \ref{thm:Wu_necessary} implies that
\begin{equation}\label{eq:Wu_Farey}
    \sum_{i=1}^\infty d^2(1, \Cr(h(Q_i))) < C,
\end{equation}
as $C$ is independent of the number of quads.

Farey quads $Q_e$ are in one-to-one correspondence with dual edges $e^*\in E^*$. If two dual edges $e^*,f^*$ are disjoint and do not share a vertex, then the quads $Q_e$ and $Q_f$ are disjoint. Since $\dualtree$ is a trivalent tree, the dual edges $E^*$ can be colored red, blue, or green so that no two edges of the same color intersect. Let $R$, $B$, $G$ be the collections of dual edges colored red, blue, or green respectively. Each of $R,B,G$ corresponds to a collection of disjoint Farey quads, so \eqref{eq:Wu_Farey} applies. On the other hand, $E^* = R\cup B\cup G$, and hence 
\begin{align}\label{eq:3C}
    \sum_{e\in E} d^2(1, \Cr(h(Q_e))) < 3C.
\end{align}
Since this sum converges, for any $\vare>0$ there are only finitely many $e\in E$ such that $|\Cr(h(Q_e))-1| > \vare$. 
Recall that $s_h(e) = \log\Cr(h(Q_e))$.
Hence, there are only finitely many edges $e\in E$ such that $|s_h(e)| >  \vare$. We now choose $\vare$ such that $|s| < \vare$ implies $|s| < 2 a^{-1} d(1,e^s)$ 
by \eqref{eq:cr_dist_shear}. In particular, if $|s_h (e)| <\vare$, $s_h(e)^2 < 4 a^{-2} d(1, \Cr(h(Q_e)))^2$. We obtain $s_h\in \mc S$ from \eqref{eq:3C}.
\end{proof}

\section{Weil--Petersson metric tensor and symplectic form}\label{sec:infinitesimal}

\subsection{Finite shears and Zygmund functions}

The tangent space to the universal Teichm\"uller space
$$T(\mathbb{D}) : = \QS (\m T) / \mob(\m T) \simeq \{h : \m T \to \m T, \text{ quasisymmetric and fixing } -1, \ii, 1\}$$
at the origin $\Id_{\m T}$ consists of all Zygmund functions on the unit circle 
that vanish at $1$, $\ii$ and $-1$ 
(see \cite[Sec.\,16.6]{Gardiner-Lakic}). More precisely, consider a differentiable path (for the Banach manifold structure of $T(\mathbb{D})$) $t\mapsto h_t$ with $t\in (-\vare ,\vare)$ of quasisymmetric maps 
such that $h_0=\Id_{\m T}$. Then $\frac{\dd}{\dd t}h_t(x)|_{t=0}=u(x)$ is a Zygmund function on the unit circle and conversely, every Zygmund function on the unit circle is the tangent vector to a differentiable path of quasisymmetric maps at $\Id_{\m T} \in T(\mathbb{D})$. 

For any $h\in T(\m D)$ we can identify $T_h T(\m D)$ with the space of Zygmund functions on the circle by pullback, meaning if $t\mapsto h_t$, $t\in (-\vare,\vare)$ is a differentiable path of quasisymmetric maps fixing $1$, $\ii$, and $-1$ with $h_0 = h$, then we identify it with the Zygmund function $u(x)=\frac{\dd}{\dd t}h_t\circ h^{-1}(x)|_{t=0}$. 

The set of \textit{finitely supported shear functions} is
\begin{align*}
    \mc F := \{ \dot s: E\to \m R : \dot s(e) \neq 0 \text{ for finitely many } e\in E\}. 
\end{align*}
For any $h\in T(\m D)$ with shear coordinate $s_h$ and $\dot s\in \mc F$, the path of shear functions $ s_h + t\cdot \dot s$ for $t\in (-\vare, \vare)$ induces a path of homeomorphisms $(h_t)_{t\in (-\vare,\vare)}\subset T(\m D)$. Using the developing algorithm of Section \ref{sec:examples}, $h_t$ is of the form $h_t = H_t \circ h$, where $H_t$ is a piecewise-M\"obius homeomorphism with breakpoints in $h(V)$. Another way to view this is that the shear function of $H_t$ on $h(\Farey)$ (instead of $\Farey$) 
$$S_t (h(e)) : = s(H_t \circ h (Q_e), h(e)) - s(h (Q_e), h(e)) = s_h(e) + t\cdot \dot s(e) - s_h (e) = t  \cdot \dot s(e).$$
is finitely supported. The first author \cite{Saric2006} proved that $h_t\circ h^{-1}$ is a differentiable path in $t$ for the Banach manifold structure of $T(\mathbb{D})$.
We obtain
$$u=\frac{\dd (h_t\circ h^{-1})}{\dd t} \Big|_{t=0} = \frac{\dd H_t}{\dd t}\Big|_{t=0} \in T_{\Id} T(\m D)$$
is a piecewise $\mf {psu} (1,1)$ vector field with break points  in $h(V)$.

We now compute $u$ explicitly in terms of the shear and the computation is often simpler in the half plane model. By conjugating by the Cayley transform $\cayley$, the differentiable path of quasisymmetric maps of $\mathbb{T}$ that fixes $1$, $\ii$ and $-1$ is transformed to a differentiable path $t\mapsto \varphi_t$ with 
$t\in (-\vare ,\vare )$ of quasisymmetric maps of $\widehat{\mathbb{R}}=\mathbb{R}\cup\{\infty\}$ that fix  $-1$, $0$ and $\infty$, namely in
$$ T(\m H) := \{ \varphi : \widehat{\mathbb{R}} \to \widehat{\mathbb{R}}, \text{ quasisymmetric and fixing } -1, 0, \infty\}.$$
If $\varphi_0=\varphi$, then $\frac{\dd}{\dd t} \varphi_t\circ \varphi^{-1}(x)|_{t=0}=u(x)$ is a Zygmund function on ${\mathbb{R}}$ that vanishes at $-1$ and $0$ and satisfies $|u(x)|=O(|x|\log |x|)$ as $|x|\to\infty$.

Let $e \in E$ and $h \in T(\m D)$. Let $a, b \in \widehat{\m R}$ such that $(a,b) = \cayley (h (e))$.
Let $\varphi_t: \widehat{\mathbb{R}}\to\widehat{\mathbb{R}}$ be the path of normalized homeomorphisms conjugate to the circle homeomorphism $H_t$ with shear coordinates $t \cdot \dot s_{e}$ on $h(\Farey)$, where $\dot s_{e} (e) = 1$ and $\dot s_{e} (e') = 0 $ for all $e' \in E$, $e' \neq e$. We define
\begin{equation}\label{eq:u_ab}
u_{(a,b)} : = \frac{\dd (\cayley \circ H_t \circ \cayley^{-1})}{ \dd t}\Big|_{t  = 0}.
\end{equation}
Example~\ref{ex:single_shear} or \cite{Saric2006} gives the following explicit formulas for $u_{(a,b)}$.

When $a\ge 0$ and $b=\infty$
\begin{equation*}
u_{(a,\infty)}(x)=
\begin{cases}
x-a,\qquad &\text{for } x>a\\
0, \qquad & \text{for } x\leq a;
\end{cases}
\end{equation*}
for $a\le -1$ and $b=\infty$
\begin{equation*}
u_{(-\infty ,a)}(x)=
\begin{cases}
-(x-a), \qquad & \text{for } x<a\\
0, \qquad & \text{for } x\geq a;
\end{cases}
\end{equation*}
and for $a<b$, such that the open interval $(a,b) \subset \m R$ does not contain $-1$ or $0$, 
\begin{equation*}
u_{(a,b)}(x)=
\begin{cases}
\frac{(x-a)(x-b)}{a-b}, \qquad & \text{for } a<x<b\\
0, \qquad & \text{otherwise.}
\end{cases}
\end{equation*}
Since we assumed that $h(\Farey)$ contains the triangle $(-1,\ii, 1)$ and $(a,b) = \cayley ( h (e))$, the above cases cover all possible scenarios.

More generally, suppose that $\dot s\in \mc F$ is supported on $\{e_1,...,e_n\}\subset E$ and let $\varphi_t:\widehat{\mathbb{R}}\to\widehat{\mathbb{R}}$ be the homeomorphism conjugate to the circle homeomorphism of shear coordinate $s_h + t\cdot \dot s$. 
By developing and the chain rule
\begin{equation}
\label{eq:finite-shears-zygmund}
u=\frac{\dd (\varphi_t\circ \varphi^{-1})}{\dd t}\Big|_{t=0} = \sum_{j=1}^n \dot s(e_j) u_{\cayley (h(e_j))}.
\end{equation}
Note that the above formula which gives a Zygmund function in terms of a shear function does not extend to the case of a shear function with infinite support. The first author \cite{Saric2011} proved that a summation along each fan followed by the sum over all fans is a correct notion for extending the above formula. 

\begin{df}\label{df:Omega_shears_zygmund}
    For each $h\in T(\m D)$, we define the linear operator $\Omega_h: \mc F \to T_{\Id} T(\m H)$ by 
    \begin{align*}
        \Omega_h(\dot s) := \sum_{j=1}^n \dot s(e_j)u_{\cayley (h(e_j))}.
    \end{align*}
\end{df}

\subsection{Finite shears and harmonic differentials} \label{sec:harmonic}

A point $h \in T(\m D)$ (or its conjugate $\varphi \in T(\m H)$) can be represented by an equivalence class of Beltrami coefficients in $\mathbb{H}$, which consists of $\mu\in L^{\infty}_1(\mathbb{H})$ satisfying $\norm{\mu}_\infty < 1$ and $\mu = \bar \partial w/\partial w$ for some quasiconformal extension $w : \m H \to \m H$ of $\varphi$. A differentiable path in $T(\mathbb{H})$ is represented by a differentiable path of Beltrami coefficients with respect to the $L^{\infty}$-norm. By taking derivative of this path with respect to $L^{\infty}$-norm, we conclude that a tangent vector to $T(\mathbb{H})$ at the identity is represented by an equivalence class of Beltrami differentials of $\dot \mu\in L^{\infty}(\mathbb{H})$ (for example, see \cite{Gardiner-Lakic}). A special representative of the equivalence class is given by the Ahlfors-Weill section.

More precisely, let $\dot \mu\in L^{\infty}(\mathbb{H})$ and $|\vare| < 1/\norm{\dot \mu}_\infty$. We define $w_\vare$ to be the solution of the Beltrami equation 
$$\bar \partial w_\vare (z)= \begin{cases}
\vare \dot \mu (z) \partial w_\vare (z), \quad & z \in \m H\\
0, \quad & z \in \m H^*
\end{cases}$$
 normalized to fix $-1, 0, \infty$. Here $\m H^*$ denotes the lower half plane. 
Then,
\begin{equation}
\label{eqn:lower_half}
\widehat{u}(z): =  \frac{\dd w_\vare (z)}{ \dd \vare} \Big|_{\vare = 0} =-\frac{z(z+1)}{\pi}\iint_{\mathbb{H}}\frac{\dot \mu (\zeta )\dd\xi \dd\eta}{\zeta (\zeta +1)(\zeta -z)}
\end{equation}
for all $z\in\m C$ (see \cite[Sec.\,6.5]{Gardiner-Lakic}).
 Note that $\widehat{u}(z)$ is holomorphic in the lower half-plane.

The \emph{Bers embedding} 
$$B :T(\mathbb{H})\to \mc Q(\mathbb{H}^*) = \{\phi : \m H^* \to \m C \text{ holomorphic and } \sup_{z \in \m H^*} |\phi(z)| \Im(z)^2 < \infty\}$$ is given by $[\mu] \mapsto S[w|_{\m H^*}]$, where $w$ solves the Beltrami equation
$$\bar \partial w (z)= \begin{cases}
\mu (z) \partial w (z), \quad & z \in \m H\\
0, \quad & z \in \m H^*
\end{cases}$$
and $S[w] = \frac{w'''}{w'} - \frac{3}{2} (\frac{w''}{w'})^2$ is the Schwarzian derivative of $w$. The Nehari bound shows that $S[w] \in \mc Q(\m H^*)$. Moreover, $\mu$ and $\nu$ represent the same element in $T(\m H)$ if and only if they give the same $S[w|_{\m H^*}]$. Therefore the map $B$ is well-defined and is an embedding.

The derivative of $B$ at the origin of $T(\mathbb{H})$ evaluated at the tangent vector represented by an infinitesimal Beltrami differential $\dot \mu$ is given by
\begin{equation}
\label{eqn:der_bers}
\phi (z):=(\dd B)_{\Id}([\dot \mu])(z)= \frac{\dd S[w_\vare] (z) }{\dd \vare} \Big|_{\vare = 0}  = \widehat{u}'''(z),\qquad  z\in\mathbb{H}^{*}.
\end{equation}

The Ahlfors-Weill section 
is the \emph{harmonic Beltrami differential} $\dot \mu_u$ in the equivalence class $[\dot \mu]$ representing a tangent vector $u$ (a Zygmund vector field on $\m R$) at the origin of $T(\m H)$ and is 
given by
\begin{equation}
\label{eq:harmonic}
\dot \mu_u(z):=-2 y^2 \phi (\bar{z})
\end{equation}
where $z=x+ \ii y\in\m H$.

Our goal is to express $\phi (z)$ in terms of the infinitesimal shear function. Since $(\dd B)_{\Id}$ is linear and by equation (\ref{eq:finite-shears-zygmund}), it is enough to compute $(\dd B)_{\Id}([\dot \mu_{(a,b)}])$ where $\dot \mu_{(a,b)}$ is the harmonic Beltrami differential corresponding to $u_{(a,b)}$ defined in \eqref{eq:u_ab} and computed explicitly. 
Extend $\dot{\mu}_{(a,b)}$ to $\mathbb{C}$ such that $\dot{\mu}_{(a,b)}(z)=\overline{\dot{\mu}_{(a,b)}(\bar{z})}$. Then we have  for $x\in \m R$
$$
{u}_{(a,b)}(x)=-\frac{1}{\pi}\iint_{\mathbb{C}} \dot{\mu}_{(a,b)}(\zeta ) R(x,\zeta )\, \dd \xi  \dd\eta =-\frac{2}{\pi} \Re \iint_{\mathbb{H}} \dot{\mu}_{(a,b)}(\zeta ) R(x,\zeta )\, \dd \xi  \dd\eta,
$$
where $R(x,\zeta )=\frac{x(x+1)}{\zeta (\zeta +1)(\zeta -x)}$ and $\zeta =  \xi + \ii \eta$. 
The Hilbert transform for ${u}_{(a,b)}(x)$ on $\mathbb{R}$ is given by the formula
$$
H{u}_{(a,b)}(x)=\frac{1}{\pi} \, \mathrm{p.v.} \int_{-\infty}^{\infty} u_{(a,b)}(\xi )R(x,\xi) \dd \xi .
$$
An application of Stokes' theorem gives, for $x\in\mathbb{R}$, 
$$
H{u}_{(a,b)}(x)=\frac{2\ii}{\pi} \iint_{\mathbb{H}} \dot{\mu}_{(a,b)}(\zeta )R(x,\zeta )\, \dd \xi  \dd\eta +\ii u_{(a,b)}(x)
$$
and
$$
H{u}_{(a,b)}(x)=-\frac{2\ii}{\pi} \iint_{\mathbb{H}^*} \dot{\mu}_{(a,b)}(\zeta )R(x,\zeta )\, \dd \xi  \dd\eta -\ii u_{(a,b)}(x).
$$
By adding the above two equations we obtain
$$
H{u}_{(a,b)}(x)=\frac{\ii}{\pi} \iint_{\mathbb{H}} \dot{\mu}_{(a,b)}(\zeta )R(x,\zeta )\, \dd \xi  \dd\eta -\frac{\ii}{\pi} \iint_{\mathbb{H}^*} \dot{\mu}_{(a,b)}(\zeta )R(x,\zeta )\, \dd \xi  \dd\eta
$$
and together with the above formula for ${u}_{(a,b)}$ gives
$$
{u}_{(a,b)}(x)+\ii H{u}_{(a,b)}(x)= -\frac{2}{\pi} \iint_{\mathbb{H}} \dot{\mu}_{(a,b)}(\zeta )R(x,\zeta )\, \dd \xi  \dd\eta .
$$
By replacing $x$ with $z\in\mathbb{C}$ in the above integral, we obtain the function $2\widehat{u}_{(a,b)} (x)$ where $\widehat{u}_{(a,b)}$ is defined in \eqref{eqn:lower_half} with $\dot \mu = \dot \mu_{(a,b)}$,
that is holomorphic in $\mathbb{H}^*$ and whose $\bar{\partial}$ derivative in $\mathbb{H}$ is $2 \dot{\mu}_{(a,b)}$. 
A direct computation of the Hilbert transform (see \cite[Section 3]{Saric2011}) and extending ${u}_{(a,b)}(x)+\ii H{u}_{(a,b)}(x)$ to a holomorphic function in $\mathbb{H}^*$ 
gives (up to an addition of a linear polynomial) 
$$
\widehat{u}_{(a,b)}(z)=\frac{\ii}{2\pi}\frac{(z-a)(z-b)}{a-b}\log \frac{z-b}{z-a}
$$
for $(a,b)$ with $a < b \neq\infty$, and
$$
\widehat{u}_{(a,b)}(z)=-\frac{\ii}{2\pi}(z-a)\log (z-a)
$$
for $e_j=(a,\infty)$. 

In the formulas above, $\frac{z-b}{z-a}\in\mathbb{H}^*$ when $z\in\mathbb{H}^*$. The natural logarithm $\log z$ for $z\in\mathbb{H}^*$ has the imaginary part in $[-\pi ,0]$ with $-\pi$ corresponding to the negative axis. When $x\in [a,b]$, then $\frac{x-b}{x-a}$ is on the negative real axis and the imaginary part of the logarithm is $-\pi$.
Using (\ref{eqn:der_bers}) we obtained the following formula.

\begin{thm}\label{thm:infBers}
 Let $\dot s \in \mc F$ with support $\{e_1, \ldots e_n\} \subset E$ and $h \in T(\m D)$. Let $\dot \mu$ be any Beltrami differential representing the Zygmund vector field $\O_h (\dot s)$ \textnormal(Definition~\ref{df:Omega_shears_zygmund}\textnormal). The infinitesimal Bers embedding of $\dot \mu$ is given by
 \begin{equation}\label{eq:inf_bers}
  (\dd B )_{\Id}(\dot \mu)(z)=\frac{\ii}{2\pi}\ \sum_{j=1}^{n} \dot s(e_j)
\frac{(a_j-b_j)^2}{(z-a_j)^2(z-b_j)^2}  , \qquad z\in \m H^*
\end{equation}
where $(a_j, b_j) = \cayley (h (e_j))$.
\end{thm}

Note that the right-hand side of \eqref{eq:inf_bers} is symmetric in $a_j$ and $b_j$, therefore we do not require $a_j < b_j$. When $a_j$ or $b_j = \infty$, the ratio is understood as the limit as $a_j$ or $b_j \to \infty$.

\subsection{Weil--Petersson metric on $\mc H$} 
In this section we compute the Weil--Petersson metric tensor on $\mc H$  (see Theorem~\ref{thm:WP_one_diamond}, Corollary~\ref{cor:full_metric}).
Recall that $\mc H$ is equipped with the topology induced by the $\ell^2$ norm, so for any $h\in \mc H$, the tangent space at $h$ is 
\begin{align*}
    \mf h := T_{h} \mc H = \{\dot\vartheta:\sum_{e\in E} \dot\vartheta(e)^2<\infty\}.
\end{align*}
For $\dot \vartheta \in \mf h$, the path $h_t$ defined by 
\begin{align*}
    \vartheta_{h_t}(e) = \vartheta_{h}(e) + t \cdot \dot\vartheta (e) \qquad \forall e\in E
\end{align*}
is contained in $\mc H$ for $t\in (-\vare, \vare)$, and has tangent vector $\dot \vartheta \in \mf h$. Since the coordinate-change map $\Phi$ from diamond shears to shears is linear (see Equation \eqref{eq:def_Phi}), $\mf h$ can also be written in terms of infinitesimal shears as $\{\dot s = \Phi(\dot \vartheta) : \sum_{e\in E}\dot\vartheta(e)^2 < \infty \}$. 

\begin{remark}
We have seen that the tangent space for quasisymmetric homeomorphisms consists of Zygmund functions, where the notion of \emph{a differentiable path} uses the Teichm\"uller metric on $T(\m D)$.  It is known that quasisymmetric homeomorphisms and Zygmund functions have \textit{different} characterizations in terms of shears \cite{Saric2011}.   In particular, when an infinitesimal shear $\dot s$ corresponding to a Zygmund vector field has infinite support, the one-parameter family $\{t \dot s(e) : e\in E\}$ is not necessarily contained in $T(\m D)$. (In fact, some might not even be induced by circle homeomorphisms.)
On the other hand, $\mc H$ is defined in terms of diamond shears, and we are using its $\ell^2$ topology in diamond shears, so the identification of $\mf h$ with $\mc H$ is automatic. 
\end{remark}

In the half-plane model $\mathbb{H}$, the Weil--Petersson Riemannian pairing of two Zygmund vector fields $u_1$ and $u_2$ is given by
\begin{equation} \label{eq_def_wp_metric}
\langle u_1,u_2\rangle_{\WP} =\Re\iint_{\m H}\dot \mu_{u_1}(z)\overline{\dot \mu_{u_2}(z)}\frac{1}{y^2} \,\dd x \dd y, \qquad z = x + \ii y.    
\end{equation}
We say that $u \in T_{\Id}\WP(\m T)$ if $\norm{u}_{\WP}^2 = \brac{u,u}_{\WP} <\infty$. In terms of Fourier coefficients  \cite{NagVerjovsky}
$$\brac{u_1, u_2}_{\WP} = 2 \pi \Re \sum_{n\ge 2} (n^3 - n) \tilde u_{1, n} \tilde u_{2, -n}, \quad \text { where } u_j= \sum_{n \in \m Z} \tilde u_{j, n} e^{\ii n \t} \frac{\partial}{\partial \t}, \quad \tilde u_{j, n} = \overline {\tilde u_{j, -n}} \in \m C. $$
This shows that $T_{\Id}\WP(\m T) = H^{3/2} (\m T)$, the $H^{3/2}$ Sobolev space of vector fields on $\m T$.

We show first that for $h\in \mc H$, $\dot \vartheta \in \mf h$ induces a vector in $T_h \WP (\m T) \simeq T_{\Id} \WP(\m T)$ (where the identification is the isometry given by the right-composition by $h^{-1}$).

\begin{lem}\label{lem:bounded_operator}
  Fix $h \in \mc H$. Let $\dot \vartheta$ be a function $E \to \m R$ with finite support and $h_t$ be the Weil--Petersson homeomorphism induced by the diamond shear function $\vartheta_h + t \cdot \dot \vartheta$. We write $u = \dd (h_t \circ h^{-1}) /\dd t |_{t = 0} \in T_{\Id } \WP(\m T)$. There exists $C(h) > 0$, such that
  $$ \norm{u}_{\WP} \le C(h) \norm{\dot \vartheta}_{\mf h}. $$
\end{lem}
From Proposition~\ref{prop:add} it is clear that $u$ is a piecewise $\mf {psu} (1,1)$ vector field, $C^{1,1}$ regular, with finitely many break points all in $h(V)$. This implies that $\norm{u}_{\WP} < \infty$ as $C^{1,1} \subset H^{3/2}$.
The point of the lemma is the quantitative bound of in terms of $\dot \vartheta$ which implies the following: 
\begin{cor} \label{cor:h_is_in_H32}
The linear map 
    $\dot \vartheta \mapsto u$ in Lemma~\ref{lem:bounded_operator} extends by continuity to a bounded linear operator $\emb_h:  \mf h \simeq T_h \mc H \to T_{\Id} \WP(\m T) (\simeq T_h \WP(\m T))$. 
\end{cor}
\begin{remark}\label{rem:xi_omega}
By linearity, if $\Phi(\dot \vartheta)$ is a finitely supported infinitesimal shear function, then $\emb_h(\dot\vartheta) = \cayley^* \Omega_h(\Phi(\dot \vartheta))$, where $\Omega_h$ is as in Definition \ref{df:Omega_shears_zygmund} and $\cayley^*$ is the pull-back map sending Zygmund vector fields on $\m R$ to Zygmund vector fields on $\m T$.
\end{remark}
\begin{proof}[{Proof of Lemma~\ref{lem:bounded_operator}}]
We use the quasiconformal extension of $h_t \circ h^{-1}$ as in Section~\ref{sec:relation_wp} and let $\mu_t$ be the associated Beltrami differential. By fixing $\vare = 1$ and $t$ small enough,
Theorem~\ref{thm:H_convergence} shows that there exists $C(h) >0$,
\begin{align*}
    \iint_{\m D} \frac{4|\mu_t(z)|^2}{(1-|z|^2)^2} \, \dd A(z) \le C(h) t^2 \sum_{e\in E} \dot \vartheta(e)^2.
\end{align*}
From the explicit expression of our quasiconformal extension we see that $\mu_t$ depends on $t$ analytically, and therefore we have $\mu_t (z) = t \cdot \dot \mu (z)+ O(t^2)$.  Letting $t \to 0$ we obtain the bound
$$\iint_{\m D} \frac{4 |\dot \mu (z)|^2}{(1-|z|^2)^2} \, \dd A(z) \le C(h) \sum_{e\in E} \dot \vartheta(e)^2 = C(h) \norm{\dot \vartheta}_{\mf h}^2.$$
However, $\dot \mu$ is not a harmonic Beltrami differential. 
Let $\dot \mu_u$ be the corresponding harmonic Beltrami differential  (in the half-plane model) as  defined in \eqref{eq:harmonic}, we have
 $\dot \mu_u (z) = -2y^2(\dd B)_{\Id}([\dot \mu])(\bar z) = : - 2 y^2 \phi (\bar z)$. 
 By \eqref{eqn:der_bers}, denoting the pushforward of $\dot \mu$ to $\m H$ also by $\dot \mu$,
\begin{align*}
    \norm{u}_{\WP}^2 = \iint_{\m H} |\dot \mu_u (z)|^2 y^{-2} \, \dd x \dd y = 4 \iint_{\m H} |\phi(\bar z)|^2 y^2\, \dd x \dd y 
    = 4\iint_{\m H^*} |\hat{u}'''( z)|^2  y^2 \, \dd x \dd y,
\end{align*}
where $\hat u$ is defined in Equation \eqref{eqn:lower_half} and we have
\begin{align*}
    \hat{u}'''(z) = - \frac{6}{\pi} \iint_{\m H} \frac{\dot \mu(\zeta)\, \dd \xi \dd \eta}{(\zeta-z)^4}, \qquad \zeta = \xi + \ii \eta.
\end{align*}
By Cauchy-Schwarz, 
\begin{align*}
    \norm{u}_{\WP}^2 \leq \frac{4 \cdot 6^2}{\pi} \iint_{\m H^*} y^2 \bigg( \iint_{\m H} \frac{\dd \xi_1 \dd \eta_1}{|\zeta_1-z|^4} \cdot \iint_{\m H} \frac{|\dot \mu(\zeta_2)|^2 \, \dd \xi_2 \dd \eta_2}{|\zeta_2-z|^4} \bigg) \, \dd x \dd y.
\end{align*}
Using the identities
\begin{align*}
    \iint_{\m H} \frac{\dd \xi \dd \eta}{|\zeta - z|^4} 
    = \frac{\pi}{4 y^2} , \qquad z = x + \ii y \in \m H^*,\\
    \iint_{\m H^*} \frac{\dd x \dd y}{|\zeta - z|^4} 
    = \frac{\pi}{4 \eta^2} \qquad \zeta = \xi + \ii \eta \in \m H,
\end{align*}
we get that 
\begin{align*}
    \norm{u}_{\WP}^2 \leq 
    9 \int_{\m H} \eta_2^{-2} |\dot \mu(\zeta_2)|^2 \, \dd \xi_2 \dd \eta_2 \le C'(h)  \norm{\dot \vartheta}_{\mf h}^2
\end{align*}
as claimed.

\end{proof}

Now we explicitly compute the Weil--Petersson metric tensor on $\mc H$. Let $\{\dot \vartheta_e\}_{e \in E}$ denote a basis of $\mf h$, where $\dot \vartheta_e (e) = 1$ and $0$ otherwise. 
It suffices to compute for all $h \in \mc H$, $e_1, e_2 \in E$, 
$$\brac{\Xi_h(\dot \vartheta_{e_1}), \Xi_h (\dot \vartheta_{e_2})}_{\WP},$$
namely, the inner product between two unit infinitesimal diamond shears on $h (\Farey)$.

More precisely, assume that a quad $Q$ has vertices $a_1,a_2,a_3,a_4\in \m T$ in counterclockwise order. Then \emph{the unit infinitesimal diamond shear on $Q$ with diagonal $(a_1,a_3)$} is the infinitesimal shear with coordinates $\dot s$ which is only non-zero at the edges of $Q$ with  
$$
\dot{s}(a_1,a_2)=\dot{s}(a_3,a_4)=1
$$
and
$$
\dot{s}(a_2,a_3)=\dot{s}(a_4,a_1)=-1.
$$
\begin{lem}\label{lem:diamond_shear_xi_explicitly}
 The unit infinitesimal diamond shear on $h(Q_e)$ with diagonal $h(e)$ is $\Xi_h( \dot \vartheta_e)$.
\end{lem}
\begin{proof}
This follows directly from Proposition~\ref{prop:add}.
\end{proof}

Theorem~\ref{thm:infBers} implies that in the half-plane model, the corresponding quadratic differential is
\begin{equation}
\label{eq:diamond-bers}
\begin{split}
\phi (z):=
\frac{\ii}{2\pi}\Big{[}\frac{(\cayley(a_1)-\cayley(a_2))^2}{(z-\cayley(a_1))^2(z-\cayley(a_2))^2}-\frac{(\cayley(a_2)-\cayley(a_3))^2}{(z-\cayley(a_2))^2(z-\cayley(a_3))^2}+\\ \frac{(\cayley(a_3)-\cayley(a_4))^2}{(z-\cayley(a_3))^2(z-\cayley(a_4))^2}-\frac{(\cayley(a_4)-\cayley(a_1))^2}{(z-\cayley(a_4))^2(z-\cayley(a_1))^2}\Big{]}
\end{split}
\end{equation}
for all $z\in \m  H^*$, and 
we have

\begin{equation}\label{eq:diamond-bers-simple}
    \phi (z)=
\frac{ \ii}{\pi}\sum_{j=1}^4 (-1)^{j-1}\left( \frac{1}{\cayley(a_{j+1})-\cayley(a_j)} +  \frac{1}{\cayley(a_{j})-\cayley(a_{j-1})} \right)\frac{1}{z-\cayley(a_j)} ,
\end{equation}
where the subscripts are considered modulo 4.

\begin{thm}\label{thm:WP_one_diamond}
If $u$ is the Zygmund vector field associated with the unit infinitesimal diamond shear 
on the quad with vertices $(a_1,a_2,a_3,a_4)$ in $\m T$ with diagonal $(a_1,a_3)$, then
$$
\|u\|^2 _{\WP}=\frac{2}{\pi} \sum_{j,k= 1}^4 \frac{(-1)^{j+k} a_j^2 \bar a_k^2 (a_{j+1}-a_{j-1})(\bar{a}_{k+1}-\bar{a}_{k-1})}{(a_{j+1}-a_j)(a_{j}-a_{j-1})(\bar{a}_{k+1}-\bar{a}_k)(\bar{a}_{k}-\bar{a}_{k-1})} \sigma (a_j,a_k).
$$
where for $a, b \in \m T$,
\begin{equation}\label{eq:sigma}
\sigma (a,b)=\sum_{p=0}^{\infty} \frac{(a\bar{b})^{p+1}}{(1+p)(2+p)(3+p)}.
\end{equation}

Further, let $u_1, u_2: E\to \m R$ be two unit infinitesimal diamond shears on quads with vertices $Q_1 = (a_1,a_2,a_3,a_4)$ and $Q_2 = (b_1,b_2,b_3,b_4)$ in $\m T$ and diagonals $e_1 = (a_1, a_3)$ and $e_2 = (b_1, b_3)$ respectively. Then 
\begin{align}\label{eq:two_quad_product}
    \langle u_1,u_2\rangle_{\WP} = \frac{2}{\pi} \Re \sum_{j,k=1}^4 \frac{(-1)^{j+k} a_j^2 \bar b_k^2 (a_{j+1}-a_{j-1})(\bar b_{k+1} -\bar b_{k-1})}{(a_{j+1}-a_{j})(a_{j}-a_{j-1})(\bar b_{k+1}-\bar b_{k})(\bar b_{k} - \bar b_{k-1})} \sigma(a_j,b_k).
\end{align}
\end{thm}
\begin{proof}
Let $\zeta : = \cayley^{-1} (z) = \xi + \ii \eta$. We have from change of variables, \eqref{eq:diamond-bers-simple}, and \eqref{eq:harmonic} that in the disk model
\begin{align*}
    \dot \mu_u (\zeta) = - \frac{\ii}{2\pi} (1- |\zeta|^2)^2 \sum_{j = 1}^4 \frac{(-1)^{j-1}}{\bar \zeta^3} \left(\frac{1}{a_{j+1} - a_j} + \frac{1}{a_{j} - a_{j-1}} \right) \frac{1}{ 1 - a_j \bar \zeta}.
\end{align*}
We notice that 
$$\sum_{j = 1}^4 (-1)^j\left(\frac{1}{a_{j+1} - a_j} + \frac{1}{a_{j} - a_{j-1}} \right) a_j^p = 0$$
for $p =0, 1, 2$. Therefore, 
\begin{align*}
    \dot \mu_u (\zeta)  &=  \frac{\ii}{2\pi} (1- |\zeta|^2)^2 \sum_{j = 1}^4 \frac{(-1)^{j}}{\bar \zeta^3} \left(\frac{1}{a_{j+1} - a_j} + \frac{1}{a_{j} - a_{j-1}} \right) \sum_{p \ge 0} (a_j \bar \zeta)^p \\
    & =  \frac{\ii}{2\pi} (1- |\zeta|^2)^2 \sum_{j = 1}^4 \frac{(-1)^{j}}{\bar \zeta^3} \left(\frac{1}{a_{j+1} - a_j} + \frac{1}{a_{j} - a_{j-1}} \right) \sum_{p \ge 3} (a_j \bar \zeta)^p \\
    & =  \frac{\ii }{2\pi} (1- |\zeta|^2)^2 \sum_{j = 1}^4 (-1)^{j}a_j^2\left(\frac{1}{a_{j+1} - a_j} + \frac{1}{a_{j} - a_{j-1}} \right) \frac{a_j}{1 - a_j \bar \zeta}.
    \end{align*}
    Define for $a, b \in \m T$,
\begin{equation}\label{eq:bluesigma}
    \sigma(a,b) = \frac{1}{2\pi} \iint_{\m D} \frac{a}{1-a\bar z} \frac{\bar b}{1-\bar b z}(1-|z|^2)^2 \, \dd x \dd y.
\end{equation}
Using polar coordinates we find
\begin{align*}
\sigma(a,b) & =\frac{1}{2\pi}\iint_{\m D}\frac{a}{1-a\bar z}\frac{\bar b}{1-\bar{b}z}(1-|z|^2)^2 \,\dd x\dd y \\
& =   \frac{a \bar b}{2\pi} \int_0^1 \int_{0}^{2\pi} \sum_{p,q \ge 0}  (a re^{-\ii \theta})^p ( \bar b re^{\ii \theta})^q (1-r^2)^2 r\, \dd \theta \dd r \\
& =  \sum_{p\ge 0}(a \bar b )^{p+1} \int_0^1 r^{2p + 1}(1-r^2)^2 \, \dd r\\
& =  \sum_{p\ge 0}\frac{(a \bar b )^{p+1}  }{(p+1)(p+2)(p+3)}.    
\end{align*}
    
    The square of the Weil--Petersson norm of $u$ equals
    \begin{align*}
        \norm{u}^2_{\WP} & =  \Re \iint_{\m D} |\dot \mu_u|^2 \frac{4}{(1-|\zeta|^2)^2} \, \dd \xi \dd \eta \\
        & =\frac{1}{\pi^2} \sum_{j,k= 1}^4 (-1)^{j+k} a_j^2 \bar a_k^2 \left(\frac{1}{a_{j+1} - a_j} + \frac{1}{a_{j} - a_{j-1}} \right)  \left(\frac{1}{\bar a_{k+1} - \bar a_k} + \frac{1}{\bar a_{k} - \bar a_{k-1}} \right) \\
        & \qquad \iint_{\m D}   \frac{a_j}{1 - a_j \bar \zeta}  \frac{\bar a_k}{1 - \bar a_k  \zeta} (1-|\zeta|^2)^2 \, \dd \xi \dd \eta \\
        & =  \frac{2}{\pi} \sum_{j,k= 1}^4 \frac{(-1)^{j+k} a_j^2 \bar a_k^2 (a_{j+1}-a_{j-1})(\bar{a}_{k+1}-\bar{a}_{k-1})}{(a_{j+1}-a_j)(a_{j}-a_{j-1})(\bar{a}_{k+1}-\bar{a}_k)(\bar{a}_{k}-\bar{a}_{k-1})} \sigma (a_j,a_k).
    \end{align*}
    Notice that the first integral is real so we omit $\Re$ in the second equality.
    The same computation gives the claimed formula for $\brac{u_1, u_2}_{\WP} = \Re \iint_{\m D} \dot \mu_{u_1} \overline {\dot \mu_{u_2}} \frac{4}{(1-|\zeta|^2)^2} \, \dd \xi \dd \eta $.
\end{proof}

\begin{remark}\label{rem:wp_norm_single_diamond}

Since the Weil--Petersson metric is invariant under the adjoint action of $\PSL(2,\m R)$, therefore also under $\PSL(2,\m Z)$, the Weil--Petersson norm of a unit infinitesimal diamond shear is constant for all Farey quads. To compute it, we consider the case where $a_1 = 1$, $a_2 = \ii$, $a_3 = -1$, and $a_4 = -\ii$, with diagonal $(-1,1)$. We obtain 
\begin{align*}
   & \s (1,1) = \s (\ii, \ii) = \s (-1, -1) = \s (-\ii, -\ii) = \sum_{p=1}^\infty \frac{1}{p(1+p)(2+p)}=\frac{1}{4} \\
   &   \s (1,-\ii) = \s (\ii, 1) = \s (-1, \ii) = \s (-\ii, -1) = \sum_{p=1}^\infty \frac{\ii^p}{p(1+p)(2+p)}= \frac{3-\pi}{4} - \ii \frac{\log 2 - 1}{2}.\\
    &   \s (-\ii, 1) = \s (1,\ii) = \s (\ii, -1) = \s (-1, -\ii) = \overline{\s (1,-\ii)}  = \frac{3-\pi}{4} + \ii \frac{\log 2 - 1}{2} \\
    &  \s (1,-1) = \s (\ii, -\ii) = \s (-1, 1) = \s (-\ii, \ii) = \sum_{p=1}^\infty \frac{(-1)^p}{p(1+p)(2+p)}= \frac{5}{4} - 2\log 2.
\end{align*}
Therefore Theorem~\ref{thm:WP_one_diamond} implies that 
\begin{align*}
    \norm{u}_{\WP}^2 = \frac{8}{\pi} \log 2
\end{align*}
for $u$ the Zygmund vector field corresponding the unit single infinitesimal diamond shear supported on $e\in E$.
\end{remark}

\begin{cor}\label{cor:full_metric}
We define  $g (Q_1, e_1, Q_2, e_2)$ to be the right-hand side of \eqref{eq:two_quad_product}.
    For $j = 1,2$, let $\dot \vartheta_j\in \mf h \simeq T_h \mc H$ and $u_j : = \emb_h (\dot \vartheta_j) \in T_{\Id} \WP(\m T) \simeq T_h \WP(\m T)$ be the corresponding vector field. 
    Then 
    $$\brac{u_1, u_2}_{\WP} = \sum_{e_1 \in E} \sum_{e_2 \in E} \dot \vartheta_1(e_1) \dot \vartheta_2 (e_2) \, g (h(Q_{e_1}), h(e_1), h(Q_{e_2}), h(e_2)).$$
\end{cor}
\begin{proof}
This follows from Lemma~\ref{lem:diamond_shear_xi_explicitly} and Theorem~\ref{thm:WP_one_diamond} if $\dot \vartheta_1$ and $\dot \vartheta_2$ are finitely supported. Lemma~\ref{lem:bounded_operator} extends the result to all $\dot \vartheta_j \in \mf h$.
\end{proof}

\subsection{Weil--Petersson symplectic form on $\mc H$}\label{sec:symplectic}

We give an expression for the Weil--Petersson symplectic form $\o$ restricted to $\mc H$ in terms of a mixture of shears and diamond shears. Similar to the computation of the metric, the Weil--Petersson symplectic form is also right-invariant on $\WP(\m T) \simeq T_0(\m D)$, so we simply use $\o$ to denote its alternating bilinear form on $T_{\Id} \WP(\m T)$ (and compute only for pull-backs to $T_{\Id} \WP(\m T)$). 

\begin{thm}\label{thm:symplectic_form}
Let $h \in \mc H\subset \WP(\m T)$.  For $j = 1,2$, let  $\dot \vartheta_j \in \mf h$, $\dot s_j := \Phi (\dot \vartheta_j)$ be the corresponding infinitesimal shears, and let $u_j := \emb_h (\dot \vartheta_j) \in T_{\Id} \WP(\m T) = H^{3/2}(\m T)$ be the pull-back by $h^{-1}$ of the tangent vector   in $T_{h} \WP(\m T)$  represented by 
the differentiable path $t \mapsto \vartheta_h + t \cdot \dot\theta_j$ 
as in Corollary~\ref{cor:h_is_in_H32}.  Then
\begin{align*}
    \omega(u_1,u_2) = \sum_{e\in E} \dot \vartheta_1(e) \dot s_2(e) = - \sum_{e\in E} \dot s_1(e) \dot \vartheta_2(e).
\end{align*}
\end{thm}
\begin{remark}
It is remarkable that unlike the expression of the metric tensor (Corollary~\ref{cor:full_metric}), the expression of the symplectic form in shear and diamond shear coordinates is very simple and  independent of $h$.
\end{remark}

Concretely, the above theorem shows if $u_1 = \Xi_h (\dot \vartheta_{e_1}), u_2 =  \Xi_h (\dot \vartheta_{e_2})$ are the vector fields given by the unit infinitesimal diamond shears on quads $Q_1 = h(Q_{e_1}), Q_2 = h(Q_{e_2})\in h(\Farey)$ of diagonals $h(e_1), h(e_2)$ respectively, then 
\begin{itemize}[itemsep=-2pt]
    \item $\omega(u_1,u_2) = 1$ if $Q_1,Q_2$ overlap in one triangle and $e_1, e_2$ are adjacent in a fan and the index of $e_2$ is the index of $e_1$ plus $1$;
    \item $\omega(u_1,u_2) = -1$ if $Q_1,Q_2$ overlap in one triangle and $e_1, e_2$ are adjacent in a fan and the index of $e_2$ is the index of $e_1$ minus $1$;
    \item $\omega(u_1,u_2) = 0$ if $Q_1 = Q_2$ or if $Q_1, Q_2$ are disjoint. Note that two quads are considered disjoint if they overlap in only a vertex or edge. See Figure~\ref{fig:quad_cases}.
\end{itemize}
\begin{figure}[ht]
    \centering
    \includegraphics[scale=0.8]{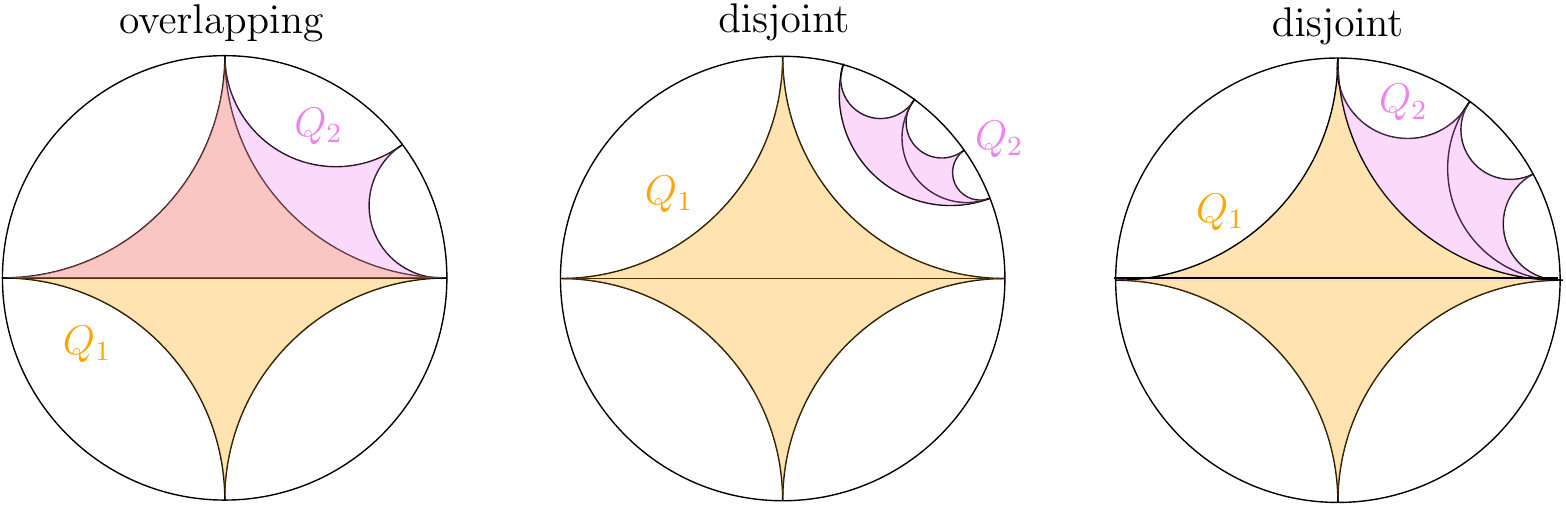}
    \caption{An example of quads overlapping in a triangle and gives value $\o(u_1,u_2) = -1$ (left) and two examples of disjoint quads (middle and right).}
    \label{fig:quad_cases}
\end{figure}

The most direct way to prove Theorem \ref{thm:symplectic_form} would be to use the same computations as in Theorem \ref{thm:WP_one_diamond}, as in the half-plane model $\mathbb{H}$, 
\begin{equation}\label{eq:symplectic}
\o( u_1,u_2) = - \Im\iint_{\m H}\dot \mu_{u_1}(z)\overline{\dot \mu_{u_2}(z)}\frac{1}{y^2} \,\dd x \dd y, \qquad z = x + \ii y.    
\end{equation}
However, we have not been able to prove in general that $\omega$ has such a simple formula directly from Equation~\eqref{eq:symplectic} (see  Remark~\ref{rem:direct_symplectic?} and Lemma~\ref{lem:special_case} below for discussion). Instead, we use an expression for a symplectic form $\tilde{\omega}$ on decorated Teichm\"uller space from \cite{Penner1993UniversalCI}, and the relationship between diamond shears and $\log \Lambda$-lengths described in Section~\ref{sec:diamond_logL}. 

Recall from Definition~\ref{df:section} the section $\sigma:\mc P_0 \to \widetilde{T(\m D)}$ which gives a canonical way to choose the decoration for each $h\in \mc H\subset \mc P_0$. This allows us to compare the diamond shear coordinate of $h$ with the $\log \Lambda$-coordinate of $\sigma(h)$. For $e=(a,b)\in E$, if $\rho_a,\rho_b$ are the horocycles chosen by $\sigma$ at $h(a),h(b)$, recall the notation
\begin{align*}
    \log \Lambda_{h}(e) := \log \Lambda(\rho_a,\rho_b).
\end{align*}
Lemma~\ref{lem:theta_logL} shows that 
$$\vartheta_h(e) = -\log \Lambda_h(e).$$ 
As a corollary, the same relationship passes to $\mf h$. 
\begin{cor}\label{cor:tangent_theta_logL}
    Choose $h\in \mc H$ with diamond shear coordinate $\vartheta_{h}$and $\dot \vartheta \in \mf h$ and let $h_t$ be the Weil--Petersson homeomorphism induced by the diamond shear function $\vartheta_h + t \cdot \dot \vartheta$. Then we have
    \begin{align*}
        \dot \vartheta(e) = \frac{\dd}{\dd t}\bigg\lvert_{t=0} -\log \Lambda_{h_t}(e).
    \end{align*}
\end{cor}

The following result from \cite[Sec.\,5.1]{Penner1993UniversalCI} shows that the symplectic form $\widetilde{\omega}$ defined below on $\widetilde{T(\m D)}$ 
in terms of $\log \L$-lengths projects to the Weil--Petersson symplectic form restricted to $\Diff(\m T)/\mob$. 

\begin{thm}[See {\cite[Thm.\,5.5]{Penner1993UniversalCI}}]\label{thm:Penner_symplectic}
Let $\tau$ be a triangle in $\Farey$, and let $\{e_1,e_2,e_3\}$ be its edges in counterclockwise order. For any $h\in \Diff(\m T)/\mob(\m T)$, define 
\begin{align*}
    (\widetilde{\omega}_{\tau})_{h} := & \dd \log \Lambda_{h}(e_1)\wedge  \dd \log \Lambda_{h}(e_2) +  \dd \log \Lambda_{h}(e_2) \wedge  \dd \log \Lambda_{h}(e_3)  \\
     & +  \dd \log \Lambda_{h}(e_3) \wedge  \dd \log \Lambda_{h}(e_1).
\end{align*}
Then 
$$\widetilde{\omega}:= - \sum_{\tau\in \Farey}\; \widetilde{\omega}_\tau$$
projects to the Weil--Petersson symplectic form $\omega$ under forgetting the decoration.
\end{thm}

\begin{remark}
The expression of the Weil--Petersson symplectic form differs from the one in \cite{Penner1993UniversalCI} by a factor $2$ due to a different choice of scalar factor. Indeed, a direct computation shows that our symplectic form \eqref{eq:symplectic} can be expressed in terms of the Fourier coefficients of the vector fields on the circle as 
$$\o(u_1,u_2) = \ii \pi \sum_{n\in \m Z} (n^3 - n) \tilde u_{1,n} \tilde u_{2,-n},$$
first proved in \cite{NagVerjovsky}, where $$u_j= \sum_{n \in \m Z} \tilde u_{j,n} e^{\ii n \t} \frac{\partial}{\partial \t}, \qquad \tilde u_{j,n} = \overline {\tilde u_{j,-n}} \in \m C. $$
Whereas Penner uses the symplectic form $2 \ii \pi \sum_{n\in \m Z} (n^3 - n) \tilde u_{1,n} \tilde u_{2,-n}$. We also verify Theorem~\ref{thm:Penner_symplectic} in a special case by direct computation in terms of the shears in Lemma~\ref{lem:special_case}.
\end{remark}

\begin{proof}[Proof of Theorem \ref{thm:symplectic_form}]
First suppose that $\dot s_1, \dot s_2$ are finitely supported shear functions with the end points of the support edges in $\{a_1,...,a_n\} \subset V$. By Remark \ref{rem:xi_omega} and Definition \ref{df:Omega_shears_zygmund}, $u_1=\emb_h(\dot \vartheta_1)$ and $u_2=\emb_h(\dot \vartheta_2)$ depend only on the points $h(a_1),...,h(a_n)$. In particular, if we replace $h$ by any function $g$ which agrees with $h$ at $a_1,...,a_n$, then we still have $\emb_{g}(\dot \vartheta_j) = \emb_{{h}}(\dot \vartheta_j)$ for $j=1,2$. Therefore we can replace $h$ with $g\in \Diff(\m T)/\mob(\m T)$. We use Corollary~\ref{cor:tangent_theta_logL} to identify infinitesimal diamond shears with infinitesimal $\log \Lambda$-lengths and then apply Theorem \ref{thm:Penner_symplectic}.  

Indeed, for each $v\in V$, let $(e_n)_{n\in \m Z}$ be a labelling of $\fan(v)$ in counterclockwise order. We rewrite the expression for $\omega$ from Theorem \ref{thm:Penner_symplectic} as a sum over fans instead of triangles. 
\begin{equation}
    \omega(u_1,u_2) = \sum_{v\in V}\sum_{e_n\in \fan(v)} \dot\vartheta_1(e_n) \dot\vartheta_2(e_{n+1}) - \dot\vartheta_1(e_{n+1})\dot\vartheta_2(e_n).
\end{equation}
On the other hand, if $e=(a,b)$, we can label the edges around $Q_e$ in counterclockwise order so that the counterclockwise order in $\fan(a)$ is $e_1, e, e_4$ and the counterclockwise order in $\fan(b)$ is $e_3, e, e_2$. Using Equation \eqref{eq:def_Phi},
\begin{align*}
    \sum_{e\in E} \dot \vartheta_1(e) \dot s_2(e) &= \sum_{e\in E}  \dot\vartheta_1(e) (-\dot\vartheta_2(e_1) + \dot\vartheta_2(e_2) -\dot\vartheta_2(e_3) + \dot\vartheta_2(e_4))\\
    &=\sum_{v\in V}\sum_{e_n\in \fan(v)} \dot \vartheta_1(e_n) \dot\vartheta_2(e_{n+1}) - \dot\vartheta_1(e_{n+1}) \dot \vartheta_2(e_n).
\end{align*}
Switching $u_1,u_2$, it is clear that $\sum_{e\in E}  \dot s_1(e)  \dot \vartheta_2(e)= - \sum_{e\in E} \dot \vartheta_1(e) \dot s_2(e)$. This completes the proof in the case that $\dot s_1 , \dot s_2$ are finitely supported. 

In general, given $\dot \vartheta_1,\dot \vartheta_2\in \mf h$ for which $\dot s_1, \dot s_2$ are not finitely supported, we can find sequences $(\dot \vartheta_j^{n})_{n\geq 1}$, $j=1,2$ such that 
\begin{itemize}[itemsep=-2pt]
    \item $\dot s_j^{n} = \Phi(\dot \vartheta_j^{n})$ is finitely supported for $j=1,2$ and all $n\geq 1$. 
    \item $\dot \vartheta_j^{n}$ converges to $\dot \vartheta_j$ in $\ell^2$ in diamond shears for $j=1,2$.
\end{itemize}
(For example, the sequences $(\dot \vartheta_j^{n})_{n\geq 1}$ where $\dot \vartheta_j^{n}$ is $\dot \vartheta_j$ restricted to edges with generation less than $n$ for $j=1,2$ satisfy these properties.)

For each finite $n$, we have 
\begin{align*}
    \sum_{e\in E} \dot \vartheta_1^{n}(e) \dot s_2^{n}(e) = \omega(u_1^n,u_2^n),
\end{align*}
where $u_j^n:=\emb_h(\dot\vartheta_j^{n})$ for $j=1,2$. It remains to compute the limits of both sides as $n\to \infty$. 

By Lemma \ref{lem:bounded_operator},
$u_j^n\xrightarrow{n\to \infty} u_j$ in $H^{3/2}$ for $j=1,2$. By Remark \ref{rem:kahler},
\begin{align*}
    \omega(u_1^n, u_2^n) = \langle u_1^n, J(u_2^n)\rangle_{\WP},
\end{align*}
where $J: T_{\Id} \WP(\m T) \to T_{\Id} \WP(\m T)$ is the (almost) complex structure and an isometry, hence
\begin{align*}
     \lim_{n\to \infty} \omega(u_1^n,u_2^n) =\omega(u_1,u_2).
\end{align*}
On the other hand, since $\ell^2$ convergence in diamond shears implies $\ell^2$ convergence in shears by the expression $\Phi(\dot \vartheta) = \dot s$, \eqref{eq:def_Phi}, and Cauchy-Schwarz inequality, we have 
\begin{align*}
    \lim_{n\to \infty} \sum_{e\in E} \dot \vartheta_1^{n}(e) \dot s_2^{n}(e) =  \sum_{e\in E} \dot \vartheta_1(e) \dot s_2(e).
\end{align*}
This completes the proof for general $\dot\vartheta_1, \dot\vartheta_2\in \mf h$.
\end{proof}

\begin{remark}\label{rem:kahler}
The Weil--Petersson metric $\langle \cdot, \cdot\rangle_{\WP}$ (computed in Theorem  \ref{thm:WP_one_diamond}, Corollary~\ref{cor:full_metric}), Weil--Petersson symplectic form $\omega$ (computed in Theorem \ref{thm:symplectic_form}), and a complex structure $J$ form a K\"ahler structure on Weil--Petersson Teichm\"uller space, meaning that
\begin{align*}
    \langle u_1,u_2\rangle_{\WP} = \omega(u_1, J(u_2)).
\end{align*}
The complex structure $J$ is the Hilbert transform \cite{NagVerjovsky} on the space of Zygmund vector fields, and was computed explicitly in terms of infinitesimal $\log \Lambda$-lengths in \cite[Thm.\,6.8]{Penner2002OnHF}. Combining the symplectic form $\omega$ and complex structure $J$ from \cite{Penner2002OnHF} gives us another way to compute the metric explicitly. Using this, we independently verify that 
$\norm{u}_{\WP}^2 =8\log 2/\pi$ when $u$ is the Zygmund vector field associated to a single infinitesimal diamond shear which coincides with our result in Remark~\ref{rem:wp_norm_single_diamond}.  
\end{remark}

\begin{remark}\label{rem:direct_symplectic?}
Starting from the definition in Equation~\eqref{eq:symplectic}, the same computation (but taking the imaginary part) and the same notation as in Theorem~\ref{thm:WP_one_diamond} gives that
\begin{equation}\label{eq:symplectic_direct_general}
\omega(u_1,u_2) = - \frac{2}{\pi} \Im \sum_{j,k=1}^4 \frac{(-1)^{j+k} a_j^2 \bar b_k^2 (a_{j+1}-a_{j-1})(\bar b_{k+1} -\bar b_{k-1})}{(a_{j+1}-a_{j})(a_{j}-a_{j-1})(\bar b_{k+1}-\bar b_{k})(\bar b_{k} - \bar b_{k-1})} \sigma(a_j,b_k),
\end{equation}
where $u_1$ and $u_2$ are the vector fields representing two unit infinitesimal diamond shears. 

Can one recover Theorem \ref{thm:symplectic_form} directly from \eqref{eq:symplectic_direct_general}? We are unable to prove this in general, but check it in a special case (Lemma~\ref{lem:special_case}). The general case to check, after normalization, would be when $a_1 = 1, a_2 = \ii, a_3 = -1$, but $a_4 \neq -\ii$ and $u_2$ is an arbitrary diamond shear of vertices $(b_1, b_2, b_3, b_4)$, such that $(a_1,a_2,a_3,a_4)$ and $(b_1,b_2,b_3,b_4)$ are 
both quads in the same tessellation $h(\Farey)$ of the disk. 
\end{remark}

\begin{lem}\label{lem:special_case}
Let $u_1 \in T_{\Id} \WP(\m T)$ be the vector field associated with the unit infinitesimal diamond shear on the quad $(a_1, a_2, a_3, a_4) = (1, \ii, -1, \ii)$ of diagonal $(1, -1)$. Let $u_2$ be the unit infinitesimal diamond shear associated with $(b_1, b_2, b_3, b_4) = (1, e^{\ii \t}, \ii, -1)$ of diagonal $(1, \ii)$. Then $\o (u_1, u_2) = -1$ for all $\t \in (0,\pi/2)$.
\end{lem}
\begin{proof}
Rearranging \eqref{eq:symplectic_direct_general} gives 
\small
\begin{align*}
    -\frac{2}{\pi} \Im \sum_{p \ge 1} \frac{1}{p(p+1)(p+2)} \left(\sum_{j = 1}^{4} (-1)^j \frac{a_j^{p+2}  (a_{j+1}-a_{j-1})}{(a_{j+1}-a_{j})(a_{j}-a_{j-1})} \right) \left(\sum_{k = 1}^4 (-1)^k  \frac{\bar b_k^{p+2}(\bar b_{k+1} -\bar b_{k-1})}{(\bar b_{k+1}-\bar b_{k})(\bar b_{k} - \bar b_{k-1})}\right).
\end{align*}
\normalsize
Since $(a_1, a_2, a_3, a_4) = (1, \ii, -1, \ii)$, 
\begin{equation}\label{eq:a_4i}
    \sum_{j = 1}^{4} (-1)^j \frac{a_j^{p+2}  (a_{j+1}-a_{j-1})}{(a_{j+1}-a_{j})(a_{j}-a_{j-1})} = 4\ii
\end{equation}
if $ p = 1$ mod $4$, and $0$ otherwise.

Using the identity 
$$e^{\ii a} - e^{\ii b} = e^{\ii \frac{a+b}{2}} (2\ii \sin (\frac{a- b}{2})),$$
and writing $\a_j = e^{\ii \t_j}$, we find that 
\begin{align*}
\frac{\a_j (\a_{j+1} - \a_{j-1})}{(\a_{j+1}-\a_{j})(\a_{j}-\a_{j-1})} & = e^{\ii (\t_j + \frac{\t_{j+1}+ \t_{j-1}}{2} - \frac{\t_{j+1} + \t_{j}}{2} - \frac{\t_{j} + \t_{j-1}}{2})} \frac{\sin (\frac{\t_{j+1} -\t_{j-1}}{2})}{2 \ii  \sin (\frac{\t_{j+1} -\t_{j}}{2}) \sin (\frac{\t_{j} -\t_{j-1}}{2})}\\
&= - \frac{\ii}{2}  \frac{\sin (\frac{\t_{j+1} -\t_{j-1}}{2})}{ \sin (\frac{\t_{j+1} -\t_{j}}{2}) \sin (\frac{\t_{j} -\t_{j-1}}{2})}.
\end{align*}

Therefore,  for $k = 1,3,4$ and $p = 1$ mod $4$,
$$(-1)^k \frac{\bar b_k^{p+2}(\bar b_{k+1}-\bar b_{k-1})}{(\bar b_{k+1}-\bar b_{k})(\bar b_{k}-\bar b_{k-1})} = (-1)^k \frac{\bar b_k^{3}(\bar b_{k+1}-\bar b_{k-1})}{(\bar b_{k+1}-\bar b_{k})(\bar b_{k}-\bar b_{k-1})} $$
is purely imaginary and does not contribute to $\o (u_1, u_2)$ given \eqref{eq:a_4i}. The remaining term is
\begin{align*}
\frac{\bar b_2^{p+2}(\bar b_{3}-\bar b_{1})}{(\bar b_{3}-\bar b_{2})(\bar b_{2}-\bar b_{1})} =  \frac{\ii}{2}    \frac{e^{-\ii (p+1) \t}\sin (\pi/4)}{ \sin (\pi/4 -\t/2) \sin (\t/2)}=  \frac{\ii}{2 \sqrt 2}    \frac{e^{-\ii (p+1) \t}}{ \sin (\pi/4 -\t/2) \sin (\t/2)} .
\end{align*}
This gives that 
\begin{align*}
    \o(u_1, u_2) & = - \frac{2}{\pi} \Im \left(4 \ii \sum_{n = 0}^\infty  \frac{\ii}{2 \sqrt 2}    \frac{e^{-\ii (4n+2) \t} }{\sin (\pi/4 -\t/2) \sin (\t /2)(4n+1) (4n+2) (4n+3)}  \right) \\
    & = \frac{2 \sqrt 2}{\pi  \sin (\pi/4 -\t/2) \sin (\t /2)} \Im \left( \sum_{n = 0}^\infty   \frac{e^{-\ii (4n+2) \t}}{(4n+1) (4n+2) (4n+3)} \right)
\end{align*}

A simple trigonometry gives 
$$ \sin (\pi/4 -\t/2) \sin (\t /2) = \frac{\sqrt 2}{4} (\sin \t + \cos \t -1).$$

Simplify the imaginary part of the series for $\theta \in (0,\pi/2)$ gives
$$f(\t) :  = \Im \left( \sum_{n = 0}^\infty   \frac{e^{-\ii (4n+2) \t}}{(4n+1) (4n+2) (4n+3)} \right) = - \frac{\pi}{8} (\sin (\t) + \cos (\t) -1).$$
Indeed, it suffices to check that $f''(\t) + f(\t) = \pi /8$, and 
$\lim_{\t \to 0\splus}f(\t) = 0$ and $\lim_{\t \to (\pi/2)\sminus  }f(\t) = 0$.
This gives
$\o(u_1, u_2) = - 1$ as claimed.
\end{proof}

\bibliographystyle{abbrv}
\bibliography{ref}

\begin{thebibliography}{10}

\bibitem{bishop-wp}
C.~J. Bishop.
\newblock Weil-petersson curves, conformal energies, $\beta$-numbers, and
  minimal surfaces.
\newblock {\em Preprint
  \protect\url{http://www.math.stonybrook.edu/~bishop/papers/wpbeta.pdf}},
  2019.

\bibitem{bishop-function-theoretic}
C.~J. Bishop.
\newblock Function-theoretic characterization of weil-petersson curves.
\newblock {\em Preprint
  \protect\url{https://www.math.stonybrook.edu/~bishop/papers/wp_fcnthy.pdf}},
  2021.

\bibitem{Bonahon}
F.~Bonahon.
\newblock {\em Low-dimensional geometry}, volume~49 of {\em Student
  Mathematical Library}.
\newblock American Mathematical Society, Providence, RI; Institute for Advanced
  Study (IAS), Princeton, NJ, 2009.
\newblock From Euclidean surfaces to hyperbolic knots, IAS/Park City
  Mathematical Subseries.

\bibitem{MRW2}
M.~Bonk, J.~Junnila, D.~Marshall, S.~Rohde, and Y.~Wang.
\newblock Piecewise geodesic {J}ordan curves {II}: Loewner energy, complex
  projective structure, and new accessory parameters.
\newblock {\em in preparation}, 2023+.

\bibitem{bowick1987holomorphic}
M.~J. Bowick and S.~G. Rajeev.
\newblock The holomorphic geometry of closed bosonic string theory and
  {D}iff{$(S^1)/S^1$}.
\newblock {\em Nuclear Physics B}, 293:348--384, 1987.

\bibitem{Cui}
G.~Cui.
\newblock Integrably asymptotic affine homeomorphisms of the circle and
  {T}eichm\"{u}ller spaces.
\newblock {\em Sci. China Ser. A}, 43(3):267--279, 2000.

\bibitem{FletcherMarkovic}
A.~Fletcher and V.~Markovic.
\newblock {\em Quasiconformal Maps and Teichm\"uller Theory}.
\newblock Oxford Graduate Texts in Mathematics. Oxford University Press, 2007.

\bibitem{Frenkel_Penner}
I.~Frenkel and R.~Penner.
\newblock Sketch of a program for automorphic functions from universal
  {T}eichm\"{u}ller theory to capture monstrous moonshine.
\newblock {\em Comm. Math. Phys.}, 389(3):1525--1567, 2022.

\bibitem{Gardiner-Lakic}
F.~P. Gardiner and N.~Lakic.
\newblock {\em Quasiconformal {T}eichm\"{u}ller theory}, volume~76 of {\em
  Mathematical Surveys and Monographs}.
\newblock American Mathematical Society, Providence, RI, 2000.

\bibitem{GM}
J.~B. Garnett and D.~E. Marshall.
\newblock {\em Harmonic measure}, volume~2 of {\em New Mathematical
  Monographs}.
\newblock Cambridge University Press, Cambridge, 2005.

\bibitem{Guo}
H.~Guo.
\newblock Integrable {T}eichm\"{u}ller spaces.
\newblock {\em Sci. China Ser. A}, 43(1):47--58, 2000.

\bibitem{Hall}
R.~R. Hall.
\newblock A note on {F}arey series.
\newblock {\em J. London Math. Soc. (2)}, 2:139--148, 1970.

\bibitem{topologyofnumbers}
A.~Hatcher.
\newblock {\em Topology of {N}umbers}.
\newblock American Math Society, online preliminary version
  \url{https://pi.math.cornell.edu/~hatcher/TN/TNpage.html}, 2022.

\bibitem{johansson2021strong}
K.~Johansson.
\newblock Strong {S}zeg\"{o} theorem on a {J}ordan curve, 2021.

\bibitem{KahnMarkovic}
J.~Kahn and V.~Markovic.
\newblock Random ideal triangulations and the {Weil-Petersson} distance between
  finite degree covers of punctured {R}iemann surfaces, 2008.

\bibitem{lehto2012univalent}
O.~Lehto.
\newblock {\em Univalent functions and {T}eichm\"{u}ller spaces}, volume 109 of
  {\em Graduate Texts in Mathematics}.
\newblock Springer-Verlag, New York, 1987.

\bibitem{MALIKOV1998282}
F.~Malikov and R.~C. Penner.
\newblock The {L}ie algebra of homeomorphisms of the circle.
\newblock {\em Adv. Math.}, 140(2):282--322, 1998.

\bibitem{MRW1}
D.~Marshall, S.~Rohde, and Y.~Wang.
\newblock Piecewise geodesic {J}ordan curves {I}: weldings, explicit
  computations, and {S}chwarzian derivatives.
\newblock {\em preprint}, 2022.

\bibitem{nag_1988}
S.~Nag.
\newblock {\em The complex analytic theory of {T}eichm\"{u}ller spaces}.
\newblock Canadian Mathematical Society Series of Monographs and Advanced
  Texts. John Wiley \& Sons, Inc., New York, 1988.
\newblock A Wiley-Interscience Publication.

\bibitem{NagVerjovsky}
S.~Nag and A.~Verjovsky.
\newblock {${\rm Diff}(S^1)$} and the {T}eichm\"{u}ller spaces.
\newblock {\em Comm. Math. Phys.}, 130(1):123--138, 1990.

\bibitem{ParlierSaric}
H.~Parlier and D.~{\v S}ari\'{c}.
\newblock Quasisymmetric maps, shears, lambda lengths and flips, 2022.

\bibitem{pekonen}
O.~Pekonen.
\newblock Universal {T}eichm\"{u}ller space in geometry and physics.
\newblock {\em J. Geom. Phys.}, 15(3):227--251, 1995.

\bibitem{Penner_WPvol}
R.~C. Penner.
\newblock Weil-{P}etersson volumes.
\newblock {\em J. Differential Geom.}, 35(3):559--608, 1992.

\bibitem{Penner1993UniversalCI}
R.~C. Penner.
\newblock Universal constructions in {T}eichm\"{u}ller theory.
\newblock {\em Adv. Math.}, 98(2):143--215, 1993.

\bibitem{Penner2002OnHF}
R.~C. Penner.
\newblock On {H}ilbert, {F}ourier, and wavelet transforms.
\newblock {\em Comm. Pure Appl. Math.}, 55(6):772--814, 2002.

\bibitem{PennerBook}
R.~C. Penner.
\newblock {\em Decorated {T}eichm\"{u}ller theory}.
\newblock QGM Master Class Series. European Mathematical Society (EMS),
  Z\"{u}rich, 2012.
\newblock With a foreword by Yuri I. Manin.

\bibitem{Saric2006}
D.~{\v S}ari\'{c}.
\newblock Real and complex earthquakes.
\newblock {\em Trans. Amer. Math. Soc.}, 358(1):233--249, 2006.

\bibitem{Saric_circle}
D.~{\v S}ari\'{c}.
\newblock Circle homeomorphisms and shears.
\newblock {\em Geom. Topol.}, 14(4):2405--2430, 2010.

\bibitem{Saric2011}
D.~{\v S}ari\'{c}.
\newblock Zygmund vector fields, {H}ilbert transform and {F}ourier coefficients
  in shear coordinates.
\newblock {\em Amer. J. Math.}, 135(6):1559--1600, 2013.

\bibitem{Saric_new}
D.~{\v S}ari\'{c}.
\newblock Shears for quasisymmetric maps.
\newblock {\em Proc. Amer. Math. Soc.}, 149(6):2487--2499, 2021.

\bibitem{STZ_KdV}
M.~E. Schonbek, A.~N. Todorov, and J.~P. Zubelli.
\newblock Geodesic flows on diffeomorphisms of the circle, {G}rassmannians, and
  the geometry of the periodic {K}d{V} equation.
\newblock {\em Adv. Theor. Math. Phys.}, 3(4):1027--1092, 1999.

\bibitem{sharon20062d}
E.~Sharon and D.~Mumford.
\newblock {2D}-{S}hape analysis using {C}onformal {M}apping.
\newblock {\em International Journal of Computer Vision}, 70(1):55--75, 2006.

\bibitem{shen13}
Y.~Shen.
\newblock Weil-{P}etersson {T}eichm\"{u}ller space.
\newblock {\em Amer. J. Math.}, 140(4):1041--1074, 2018.

\bibitem{Shen-Tang}
Y.~Shen and S.~Tang.
\newblock Weil-{P}etersson {T}eichm\"{u}ller space {II}: {S}moothness of flow
  curves of {$H^{\frac 32}$}-vector fields.
\newblock {\em Adv. Math.}, 359:106891, 25, 2020.

\bibitem{Shen-Tang-Wu}
Y.~Shen, S.~Tang, and L.~Wu.
\newblock Weil-{P}etersson and little {T}eichm\"{u}ller spaces on the real
  line.
\newblock {\em Ann. Acad. Sci. Fenn. Math.}, 43(2):935--943, 2018.

\bibitem{TT06}
L.~A. Takhtajan and L.-P. Teo.
\newblock Weil-{P}etersson metric on the universal {T}eichm\"{u}ller space.
\newblock {\em Mem. Amer. Math. Soc.}, 183(861):viii+119, 2006.

\bibitem{VW1}
F.~Viklund and Y.~Wang.
\newblock Interplay between {L}oewner and {D}irichlet energies via conformal
  welding and flow-lines.
\newblock {\em Geom. Funct. Anal.}, 30(1):289--321, 2020.

\bibitem{VW2}
F.~Viklund and Y.~Wang.
\newblock The {Loewner-Kufarev Energy and Foliations by Weil-Petersson
  Quasicircles}.
\newblock {\em arXiv preprint: 2012.05771}, 2020.

\bibitem{W2}
Y.~Wang.
\newblock Equivalent descriptions of the {L}oewner energy.
\newblock {\em Invent. Math.}, 218(2):573--621, 2019.

\bibitem{W3}
Y.~Wang.
\newblock A note on {L}oewner energy, conformal restriction and {W}erner's
  measure on self-avoiding loops.
\newblock {\em Ann. Inst. Fourier (Grenoble)}, 71(4):1791--1805, 2021.

\bibitem{WiegmannZabrodin_2022}
P.~Wiegmann and A.~Zabrodin.
\newblock Dyson gas on a curved contour.
\newblock {\em J. Phys. A}, 55(16):Paper No. 165202, 34, 2022.

\bibitem{Witten}
E.~Witten.
\newblock Coadjoint orbits of the {V}irasoro group.
\newblock {\em Comm. Math. Phys.}, 114(1):1--53, 1988.

\bibitem{Wolpert_symplectic}
S.~Wolpert.
\newblock On the symplectic geometry of deformations of a hyperbolic surface.
\newblock {\em Ann. of Math. (2)}, 117(2):207--234, 1983.

\bibitem{Wu2011}
C.~Wu.
\newblock The cross-ratio distortion of integrably asymptotic affine
  homeomorphism of unit circle.
\newblock {\em Science China Mathematics}, 55(3):625--632, Dec. 2011.

\end{thebibliography}

\end{document}